\documentclass[11pt, a4paper]{article}
\usepackage{amscd,amssymb,amsmath,enumerate}
\usepackage{amsmath,amsthm}
\usepackage[all]{xy}
\usepackage{graphicx}
\usepackage{graphics}
\usepackage{latexsym}
\usepackage[all]{xy}
\usepackage{psfrag}
\xyoption{matrix} \xyoption{arrow}
\usepackage{parskip}
\usepackage{stackrel}
\usepackage{hyperref}

\usepackage[utf8]{inputenc}

\usepackage{float}
\usepackage{tikz}
\usepackage{verbatim}
\usetikzlibrary{calc}

\usetikzlibrary{arrows}
\usepackage{subfigure}
\usepackage{mathtools} 

\tikzstyle{cblack}=[circle, draw, thin,fill=black!10, scale=0.8]
\tikzstyle{cb}=[circle, draw, thin,fill=black!100, scale=0.4]
\tikzstyle{g}=[circle, draw=green, thin,fill=green!100, scale=0.4]
\tikzstyle{cw}=[rectangle, draw, thin,fill=white, scale=0.8]
\tikzstyle{cbl}=[rectangle,blue, draw=blue, thin,fill=white, scale=0.8]
\tikzstyle{cg}=[rectangle, green, draw=green, thin, fill=white, scale=0.8]
\usepackage{subfigure}

\newcommand\dimTwo[2]{
	\begin{smallmatrix}
		#1 \\
		#2 
	\end{smallmatrix}
}
\newcommand\dimThree[3]{
	\begin{smallmatrix}
		#1 \\
		#2 \\
		#3 
	\end{smallmatrix}
}

\newcommand\matrixTwo[4]{
	\begin{smallmatrix}
		  & #1 & \\
		#2&  & #3 \\
		  & #4 & 
	\end{smallmatrix}
}

\newcommand\dimFour[4]{
	\begin{smallmatrix}
		#1 \\
		#2 \\
		#3 \\
		#4
	\end{smallmatrix}
}	
	\newcommand\dimFive[5]{
	\begin{smallmatrix}
		#1 \\
		#2 \\
		#3 \\
		#4 \\
		#5
	\end{smallmatrix}
}
\newcommand\dimEight[8]{
	\begin{smallmatrix}
		#1 \\
		#2 \\
		#3 \\
		#4 \\
		#5 \\
		#6 \\
		#7 \\
		#8
	\end{smallmatrix}
}

\usetikzlibrary{arrows,shapes,positioning}
\usetikzlibrary{decorations.markings}
\usetikzlibrary{arrows.meta}

\newtheorem{proposition}{Proposition}[section]
\newtheorem{theorem}[proposition]{Theorem}
\newtheorem{lemma}[proposition]{Lemma}
\newtheorem{corollary}[proposition]{Corollary}

\newtheorem{definition}[proposition]{{Definition}}

\newtheorem{remark}[proposition]{{Remark}}

\newtheorem{Question}[proposition]{Question}
\newtheorem{Example}[proposition]{Example}
\newtheorem{Examples}[proposition]{Examples}

\newcommand{\cD}{{\mathcal D}}

\newcommand{\cF}{{\mathcal F}}

\newcommand{\cM}{{\mathcal M}}
\newcommand{\cK}{{\mathcal K}}

\newcommand{\cC}{{\mathcal C}}

\newcommand{\calo}{{\mathfrak o}}

\renewcommand{\aa}

\newcommand{\End}{\operatorname{End}\nolimits}
\newcommand{\Hom}{\operatorname{Hom}\nolimits}
\renewcommand{\Im}{\operatorname{Im}\nolimits}

\newcommand{\rad}{\operatorname{rad}\nolimits}

\renewcommand{\mod}{\operatorname{\mathbf{mod}}\nolimits}

\newcommand{\Ext}{\operatorname{Ext}\nolimits}

\renewcommand{\dim}{\operatorname{dim}\nolimits}

\newcommand{\mo}{\mathfrak o}

\usepackage{scrextend} 

\usepackage{color}
\definecolor{candyapplered}{rgb}{1.0, 0.03, 0.0}
\usepackage{etoolbox}
\makeatletter
\patchcmd{\section}{\if@openright\cleardoublepage\else\clearpage\fi}{}{}{}
\makeatother

\makeatletter
\def\thm@space@setup{%
  \thm@preskip=0.3cm \thm@postskip=0.1cm
}
\makeatother

\begin{document}

\setcounter{tocdepth}{5}

\date{today}
\title{Brauer graph algebras}
\title{%
  Brauer graph algebras \\
  \large A survey on Brauer graph algebras, associated gentle algebras   and their connections to cluster theory\\
   }

\author{Sibylle Schroll} 


\date{\today}

\setcounter{tocdepth}{2}
\maketitle         

\begin{abstract}
These lecture notes on Brauer graph algebras are the result of a series of four lectures given at the CIMPA research school in Mar del Plata, Argentina, in March 2016. After motivating the study of Brauer graph algebras by relating them to special biserial algebras, the definition of Brauer graph algebras is given in great detail with many examples to illustrate the concepts. This is followed by a short section on the interpretation of Brauer graphs as decorated ribbon graphs. A section on gentle algebras and their graphs, trivial extensions of gentle algebras, admissible cuts of Brauer graph algebras and a first connection of Brauer graph algebras with Jacobian algebras associated to triangulations of marked oriented surfaces follows. The interpretation of flips of diagonals in triangulations of marked oriented surfaces as derived equivalences of Brauer graph algebras and the comparison of derived equivalences of Brauer graph algebras with derived equivalences of frozen Jacobian algebras is the topic of the next section. In the last section, after defining Green's walk around the Brauer graph, a complete description of the Auslander-Reiten quiver of a Brauer graph algebra is given.
\end{abstract}

\newpage
\tableofcontents                        


\bigskip





\nocite{*}

\begin{addmargin}[-1.5em]{-1.5em}

 \section*{Introduction}

Brauer graph algebras originate in the modular representation theory of finite groups where they first appear in the form of Brauer tree algebras in the work of Janusz \cite{J} based on work of Dade \cite{Dade}.  Brauer graph algebras in general, are defined by Donovan and  Freislich  in \cite{DF}. In particular, they classify the indecomposable representations of a Brauer graph algebra in terms of canonical modules of the first and second kind (string and band modules respectively). This classification is based on the work of Ringel \cite{Ringel} on indecomposable representations of dihedral 2-groups and the work  of Gel'fand and Ponomarev \cite{GP} on indecomposable representations of the Lorentz group.

Brauer graph algebras are defined by combinatorial  data based on graphs:  Underlying every Brauer graph algebra is a finite graph with a cyclic orientation of the edges at every vertex and a multiplicity function. This combinatorial data encodes much of the representation theory of Brauer graph algebras and is part of the reason for  the ongoing interest in this class of algebras. 

In \cite{MS} the point of view of interpreting Brauer graphs as ribbon graphs has been introduced adding a geometric perspective to the representation theory of Brauer graph algebras and relating Brauer graph algebras with surface cluster theory. The idea of relating Brauer graphs to surfaces  was already suggested in \cite{DF} where a description of Brauer graphs is given as graphs embedded in oriented surfaces. In \cite{AAC} the geometric approach based on ribbon graphs has been used to classify two-term tilting complexes over Brauer graph algebras.

The class of Brauer graph algebras coincides with the class of symmetric special biserial algebras \cite{Ro, S}. This connection has  introduced string combinatorics to the subject  as well as a large body of literature on biserial and special biserial algebras, for an overview see for example,  the webpages maintained by Julian K\"ulshammer at  \sloppy 
http://www.iaz.uni-stuttgart.de/LstAGeoAlg/Kuelshammer/topics/biserial.html and by Jan Schr\"oer at \sloppy http://www.math.uni-bonn.de/people/schroer/fd-atlas.html. 

In recent years there has been a renewed interest in Brauer graph algebras. In addition to the results mentioned in these lecture notes, there are new results on 2-term tilting complexes of Brauer graph algebras \cite{AAC, Zvonareva1, Zvonareva2}, on Brauer graph algebras associated to partial triangulations \cite{Demonet},  on coverings of Brauer graphs \cite{GSS},  on the finite generation of the Yoneda (or Ext) algebra of a Brauer graph algebra \cite{Antipov, GSST} and the generalised Koszulity of Brauer graph algebras \cite{GSST}, as well as new results on the non-periodicity of modules of finite complexity over weakly symmetric Brauer graph algebras \cite{Erdmann}.  Note also that  Brauer graph algebras play a central role in the recent survey   on the connection between the representation theory of finite groups and the theory of cluster algebras \cite{Ladkani3}. 

On the other hand from the point of view of modular representation theory of finite groups,  there has been much recent work to identify specific Brauer trees and Brauer graphs arising in that context, see for example \cite{Ds1, Ds2, DR, C}.

In these lecture notes, after motivating the study of Brauer graph algebras from the perspective of symmetric special biserial algebras in Section 1, in Section 2 we give a detailed definition of Brauer graph algebras.  Section 3 focuses on the connection of Brauer graph algebras with gentle algebras and the connection of Brauer graph algebras with surface cluster theory. In Section 4, mutation of Brauer graph algebras and derived equivalences are discussed.  Section 5  describes the Auslander-Reiten theory of Brauer graph algebras.

{\bf Acknowledgements. } At the origin of these lecture notes is a series of four lectures on Brauer graph algebras and cluster theory that I gave at the CIMPA research school in Mar del Plata in March 2016 and my heartfelt thanks goes to the organisers, Ibrahim Assem, Patrick Le Meur and Sonia Trepode, as well as to the scientific and local organising committees.  I also particularly thank Cristian Arturo Chaparro Acosta who contributed all of the tikz-pictures in these lecture notes (mostly at very short notice). Without his extensive tikz skills these lecture notes would not have been possible in this form.  I thank Sefi Ladkani for providing me with a summary of the quoted results of his papers  \cite{Ladkani1, Ladkani3} and for pointing out the links of this survey with his survey on triangulation quivers \cite{Ladkani3}. I further thank him, Drew Duffield, Julian K\"ulshammer, Eduardo Marcos, Klaus Lux, Rowena Paget, Rachel Taillefer as well as two anonymous referees for taking the time to read these lecture notes and for the  helpful comments and suggestions that they have provided. 

\end{addmargin}

\begin{addmargin}[-3em]{-3em}

\newpage \section{Motivation and Connections}

\subsection{Special biserial  algebras}

Let $K$ be an algebraically closed field. A quiver $Q = (Q_0, Q_1, s, t)$ is given by a  set of vertices $Q_0$, a  set of arrows $Q_1$, a function $s: Q_1 \to Q_0$ denoting the start of an arrow and $t: Q_1 \to Q_0$ denoting the end of an arrow. A path in $Q$ is a sequence of arrows $p=\alpha_0 \alpha_1 \ldots \alpha_n$ such  that $t(\alpha_i) = s(\alpha_{i+1})$, for all $ 0 \leq i \leq n-1$. We set $s(p) = s(\alpha_0)$ and $t(p) = t(\alpha_n)$.  Furthermore, we denote by $\ell(p)$ the length of $p$, that is $\ell(p) = n$. 

The \emph{path algebra} of a quiver $KQ$ is the algebra whose underlying vector space has a basis given by all possible paths in the quiver $Q$. This includes a trivial path $e_v$ for every vertex $v \in Q_0$. The multiplication is given by linearly extending the multiplication  of two paths $p$ and $q$ in $Q$, which is given by concatenation if possible and zero otherwise. That is, 

$$p \cdot q = \left\{ \begin{array} {ll}
p  q \mbox{ if } t(p) = s(q) \\
0 \mbox{ otherwise.}
\end{array} \right.$$

The algebra $KQ$ has been extensively studied and we refer the reader to classical representation theory textbooks, see for example \cite{ASS,  ARS,  Ba, BensonBook, Sch}. It follows directly from the definition, that if the quiver $Q$ has a loop or an oriented cycle, the algebra $KQ$ is infinite dimensional. 

Denote by $\mathcal R$  the two-sided ideal of $KQ$  generated by the arrows in $Q$.  We call a two-sided ideal $I$ of $KQ$ \emph{admissible} if there exists a strictly positive integer $n \geq 2$ such that 

$$ {\mathcal R}^n \subseteq I \subseteq {\mathcal R}^2. $$

If $I$ is an admissible ideal of $KQ$, then the algebra $KQ/I$ is a  finite dimensional algebra and we say that $KQ/I$ is given by quiver and relations. The algebra $KQ/ I$ is indecomposable if and only if the quiver $Q$ is connected.

The motivation behind studying the representation theory of algebras by studying algebras given by quiver and relations is the following theorem due to Gabriel. 

\begin{theorem}\cite{Gabriel}
Every connected finite dimensional $K$-algebra is Morita equivalent to an algebra $KQ/I$ for a unique quiver $Q$ and where $I$ is an admissible ideal of $KQ$.  
\end{theorem}
 
The algebras of the form $KQ/I$, $I$ admissible, are still arbitrarily complicated and therefore  subclasses of these algebras are often considered. For example, to restrict the class of finite dimensional algebras considered, further restrictions on $Q$ and on $I$ can be imposed. 

One example is the restriction of the number of arrows starting and ending at each vertex in the quiver. If at each vertex $v \in Q_0$ there is at most one arrow starting and at most one arrow ending at $v$, then $KQ/I$ is monomial  and of {\it finite representation type} that is, up to isomorphism, there are only finitely many distinct indecomposable $KQ/I$-modules. 
These algebras are the so-called \textit{Nakayama algebras} \cite{Nak, ARS, ASS}.
An algebra $KQ/I$ is {\it monomial}, if the ideal $I$ is generated by paths. Note that this definition depends on the presentation of the algebra, that is on the choice of the ideal $I$. In general, there might be many different ideals $J$, such that $KQ/I \simeq KQ/J$. 
It is an open question,  raised by Auslander, to find a homological  characterisation that implies that a given algebra is monomial. 
Nevertheless, monomial algebras have been extensively studied and many open problems  have been answered in the case of  monomial algebras. 

The next level of complexity is to allow at most two arrows to begin and end at every vertex in $Q$: 

{\bf (S0)} At every vertex $v$ in $Q$ there are at most two arrows starting at $v$ and there are at most two arrows ending at $v$. 

While this is  a very strong restriction on the algebras one can consider, most algebras whose quiver satisifies condition (S0) are of {\it wild representation type} (see Section~\ref{Section AR components} for the definition of the representation type of an algebra). 

To pose a further restriction on the class of algebras we consider, we set

{\bf (S1)} For every arrow $\alpha$ in $Q$ there exists at most one arrow $\beta$ such that $\alpha \beta \notin I$ and  there exists at most one arrow $\gamma$ such that $\gamma \alpha \notin I$. 

Almost all algebras satisfying condition (S1) (and not necessarily condition (S0)) are of wild representation type. In fact, algebras satisfying (S1)  are called \textit{special multiserial algebras}. They were first defined in \cite{VHW} and their representation theory has been studied in \cite{VHW, GS1, GS2, GS3, GS4}.

Together conditions (S0) and (S1) are very strong and the corresponding class of algebras has many special properties.

\begin{definition}
{\rm A finite dimensional $K$-algebra $A$ is called {\it special biserial} if  there is a quiver $Q$ and an admissible ideal $I$ in $KQ$ such that $A$ is Morita equivalent to $KQ/I$ and such that $KQ/I$ satisfies conditions (S0) and (S1). }
\end{definition}

Examples of special biserial algebras include the algebras appearing in the work of Gel'fand and Ponomarev \cite{GP} on the classification of Harish-Chandra modules  over the Lorentz group and they are closely linked to the modular representation theory of finite groups, see for example, Ringel's classification of the indecomposable modules over dihedral 2-groups, see e.g. \cite{Ringel}. Other examples of special biserial algebras are string algebras (monomial special biserial algebras) see for example \cite{BR}, discrete derived  algebras (classified by Vossieck in \cite{V}), Jacobian algebras of triangulations of marked oriented surfaces with boundary where all marked points lie in the boundary \cite{ABCP} and Brauer graph algebras (see Theorem~\ref{BGAspecialbiserial}).

Special biserial algebras have been intensely studied. We now give a list of some of their most important properties and results.  

\begin{theorem}\cite{WW}
Special biserial algebras are of tame representation type. 
\end{theorem}

Wald and Waschb\"usch show that special biserial algebras are of tame representation type by  classifying their indecomposable representations. These are given by the so-called \textit{string and band modules} (which first appear in \cite{GP} as modules of the first and second kind, see also \cite{Ringel, DF} and also \cite{WW}). Based on strings and bands a nice combinatorial description of the morphism between string modules is given in \cite{CB, Kr}.

Furthermore, Wald and Waschb\"usch \cite{WW} and Butler and Ringel \cite{BR} give a combinatorial description of the irreducible maps between string modules  and between band modules in terms of string combinatorics giving a description of  Auslander-Reiten sequences.  

\begin{definition}
{\rm A finite dimensional $K$-algebra $A$  is called {\it biserial} if every indecomposable projective left and right $A$-module $P$ is such that $\rad(P) = U + V$, where $U$ and $V$ are uniserial modules and $U \cap V$ is either a simple $A$-module or zero. }
\end{definition}

\begin{theorem}\cite{SW}
A special biserial algebra is biserial. 
\end{theorem}

Pogorzaly and Skowro\'nski prove a partial converse in case that the algebra is {\it standard}, that is if it admits a simply connected Galois covering, see e.g. \cite{Skowronski}.  

\begin{theorem}\cite{PS} 
Let A be a representation-infinite self-injective $K$-algebra. Then $A$ is standard biserial if and only if $A$ is special biserial. 

\end{theorem}

Recall that a linear form $f: A \mapsto K$ is {\it symmetric} if $f(ab) = f(ba)$ for all $a, b \in A$ and that a finite dimensional $K$-algebra $A$ is {\it symmetric} if there exists a symmetric linear form $f: A \mapsto K$ such that the kernel  of $f$ contains no non-zero left or right ideal.  An equivalent formulation of this is that $A$ considered as an $A$-$A$-bimodule is isomorphic  to its $K$-linear dual $D(A) = \Hom_K(A,K)$ as an $A$-$A$-bimodule. More details on equivalent definitions  can be found, for example, in \cite{RickardKZ}.

Every finite dimensional $K$-algebra $A$ is a quotient of a symmetric  $K$-algebra, for example, $A$ is a quotient of its trivial extension $T(A)$, where $T(A)$ is given by the semidirect product of $A$ with its minimal injective co-generator $D(A)$ (see Section~\ref{Section gentle} for more details on trivial extensions).
However, the class of special biserial algebras is special in that every special biserial algebra is a quotient of a symmetric algebra, such that the symmetric algebra is again special biserial. 

\begin{theorem}\cite{WW}\label{SymmetricQuotient}
Every special biserial algebra is a quotient of a symmetric special biserial algebra. 
\end{theorem}

Symmetric special biserial algebras are well understood and are ubiquitous in the modular representation theory of finite groups where they have appeared under a different name, that of Brauer graph algebras. While we will define Brauer graph algebras in Section~\ref{Definition of BGA}, we state the connection of special biserial algebras and Brauer graph algebras here. 

\begin{theorem}\cite{Ro,S}\label{Symmetric=BGA}
Let $KQ/I$ be a finite dimensional $K$-algebra. Then $KQ/I$ is a symmetric special biserial algebra  if and only if $KQ/I$ is a Brauer graph algebra.  
\end{theorem}

This Theorem was proved in \cite{Ro} for the case that the algebra $KQ/I$ is such that the quiver 
$Q$ contains no parallel arrows. In \cite{S} it was proved that the result holds for all quivers. The following directly follows from Theorems~\ref{SymmetricQuotient} and~\ref{Symmetric=BGA}

\begin{corollary}
Every special biserial algebra is a quotient of a Brauer graph algebra. 
\end{corollary}

\section{Brauer graph algebras}\label{Definition of BGA}

Recall from the introduction that Brauer graph algebras have their origin in the modular representation theory of finite groups. We will start by briefly recalling how they  appear in this context. 

\subsection{Brauer graph algebras and modular representation theory of finite groups} Let $G$ be a finite group. Suppose that the characteristic of $K$ is  equal to some prime number $p$ and suppose that $p$ divides the order of $G$. Then it follows from Maschke's Theorem that the group algebra $KG$ is not semi-simple. Instead it decomposes into a direct sum of indecomposable two-sided ideals $B_i$ such that $KG = B_1 \oplus \ldots \oplus B_n$. The  identity element $e$ of $G$ decomposes as $e = e_1 + \ldots + e_n$ where the $e_i$ are orthogonal central idempotents of $KG$ such that $e_i \in B_i$. Each $B_i$ is called a {\it $p$-block of $G$} (or of $KG$) and it is a symmetric finite dimensional $K$-algebra with identity element $e_i$. Each $p$-block has an associated invariant in the form of a $p$-group $D_i$ called the {\it defect group of $B_i$}.
Dade showed that if the defect group $D$ of a $p$-block $B$ of some finite group $G$ is cyclic then $B$ is a Brauer tree algebra \cite{Dade} and  Donovan showed that if the characteristic $K$ is two and if $D$ is a dihedral 2-group then $B$ is a Brauer graph algebra \cite{Donovan}.  We note that the blocks with cyclic defect groups are precisely the $p$-blocks of $G$ which are of finite representation type.

\subsection{Definition of Brauer graphs}

\begin{definition}
{\rm A {\it Brauer graph} $G$ is  a tuple $ G=(G_0, G_1, m, \calo)$ where 
\begin{itemize} 
\item   $(G_0, G_1)$ is a finite (unoriented) connected graph with vertex set $G_0$ and edge set  $G_1$,  
\item  $m: G_0 \to \mathbb{Z}_{>0 }$ is a function, called the \emph{multiplicity} or \emph{multiplicity function} of $G$, 
\item $\mo$ is called the \emph{orientation} of $G$  and  is given, for every vertex $v \in G_0$, by  a cyclic ordering of the edges incident with $v$ such that if $v$ is a vertex incident to a single edge $i$ then if $m(v)  = 1$,  the cyclic ordering at $v$ is given by $ i$ and if $m(v) >1$  the cyclic ordering at $v$ is given by $i<i$.
\end{itemize}

A {\it Brauer tree} is a Brauer graph $G = (G_0, G_1, m, \calo)$ such that $(G_0, G_1)$ is a tree and $m(v)  = 1$, for all but at most one $v \in G_0$.  }

\end{definition}

We note that the graph $(G_0, G_1)$ which (by abuse of notation) we will also simply refer to as $G$ may contain loops and multiple edges. Denote by ${\rm val}(v)$  the {\it valency} of the vertex $v \in G_0$. It is defined to be the number of edges in $G$ incident to $v$, with the convention that a loop is counted twice (see Example~\ref{Examples Brauer graph} (2)). Equivalently we can define  ${\rm val}(v)$ to be  the number of half-edges incident with vertex $v$. 
We call the edge $i$ with vertex $v$ {\it truncated at $v$} if   $m(v) {\rm val}(v) =1$. Note that if $i$ is truncated at both vertices $v$ and $w$, that is if both $m(v) {\rm val}(v) =1$ and $m(w) {\rm val}(w) =1$ then $G$ is the Brauer graph given by a single edge with both vertices $v$ and $w$  of multiplicity 1 and the corresponding Brauer graph algebra is defined to be  $k[x]/(x^2)$.

\begin{Examples}\label{Examples Brauer graph} {\rm In all four examples below, the orientation of the Brauer graph is given by locally embedding each vertex of the Brauer graph into the clockwise oriented plane. If for some $v \in G_0$, we have $m(v) >1$, then we record the value of $m(v)$ in a square box next to the corresponding vertex on the graph. 

(1) This is an example of what is called a  {\it generalised Brauer tree}, that is the underlying graph is a tree with at least two vertices with multiplicity greater than one. For example, let $G = (G_0, G_1, m, \calo)$ be given by

\begin{figure}[H]
	\centering 
\begin{tikzpicture}[auto, thick]

	\node[cblack] (c) at (0,0) {c};
	\node[cw] at ($(c)+(-45:.6)$) {3};

	\node[cblack] (a) at ($(c)+(225:2)$) {a};
	\node[cw] at ($(a)+(-30:.6)$) {2};
	
	\node[cblack] (b) at ($(c)+(135:2)$) {b};
	\node[cblack] (d) at ($(c)+(0:2)$) {d};
	\node[cblack] (e) at ($(d)+(45:2)$-6,3.5) {e};
	\node[cblack] (f) at ($(d)+(0:2)$-6,3.5) {f};
	\node[cblack] (g) at ($(d)+(-45:2)$-6,3.5) {g};
	
	\draw  (a) to node {$1$} (c);
	\draw  (b) to node {$2$} (c);
	\draw  (c) to node[above] {$3$} (d);
	\draw  (d) to node {$4$} (e);
	\draw  (d) to node {$5$} (f);
	\draw  (d) to node[swap] {$6$} (g);
\end{tikzpicture}
\caption{Brauer tree $G_1$.}
\end{figure}
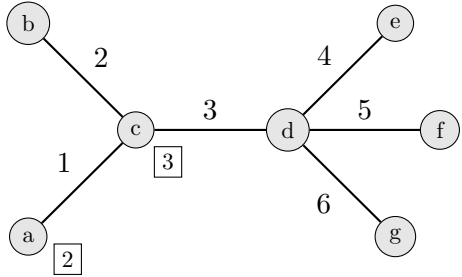

We set $m(i) = 1$ for all $i \in G_0$, $i \neq a,c$ and $m(a) = 2$ and $m(c) = 3$. 
The cyclic ordering of the edges incident with vertex $a$ is given by $1 < 1$, with vertex $b$ it is given by $2$, with vertex $c$ it is given by $1 < 2 < 3 < 1$, with vertex $d$ it is given by $3 < 4 < 5 < 6 <3 $, etc. We have the following valencies for vertices of $G$: val($a$) = val($b$) = val($e$) = val($f$) = val($g$) = 1, val($c$) =3 and val($d$) = 4.

(2) Let the Brauer graph $G_2 = (G_0, G_1, m, \calo)$ be given by

\begin{figure}[H]
	\centering
\begin{tikzpicture}[auto, thick, bend angle=40]
	\node[cblack] (a) at (-4.5,2) {a};
	\node[cblack] (b) at ($(a)+(2,0)$) {b};
	\node[cblack] (c) at ($(b)+(30:2)$) {c};
	
	\draw (a) to [out=150,in=210,looseness=14] node[swap] {1} (a);
	\draw (a) to [bend left] node {$2$}(b) ;
	\draw (a) to [bend right] node[below] {$3$} (b) ;
	\draw (b) to  node {$4$} (c);
\end{tikzpicture}
\caption{Brauer graph $G_2$.}
\end{figure}

We set $m(i) =1$ for $i = a,b,c$. The cyclic ordering of the edges incident with vertex $a$ is given by $1 < 1 < 2< 3 < 1$, with vertex $b$ is given by $2 < 4 < 3 < 2$ and with $c$ is given by $4$.   We have val($a$) = 4, val($b$) =3 and val($c$) =1. Note that the edge 4 is a truncated edge since the vertex $c$ is such that $m(c) val(c) = 1$.

(3)   Let the Brauer graph $G_3 = (G_0, G_1, m, \calo)$ be given by

\begin{figure}[H]
	\centering
\begin{tikzpicture}[auto, thick, bend angle=50]
	\node[cblack] (a) at (0,0) {a};
	\node[cblack] (b) at ($(a)+(0:2.5)$) {b};
	
	\draw (a) to [bend left] node {1} (b) ;
	\draw (a) to node{2} (b);
	\draw (a) to [bend right] node[below] {3} (b) ;
\end{tikzpicture}
\caption{Brauer graph $G_2$.}
\end{figure}

We set $m(a) = m(b) =1$. The cyclic ordering at vertex $a$ is given by $1 < 2 < 3 < 1$ and at vertex  $b$ by $ 1 < 3 < 2 < 1$ and val($a$) = val($b$)  = 3.

(4)  Let the Brauer graph $G_4 = (G_0, G_1, m, \calo)$ be given by

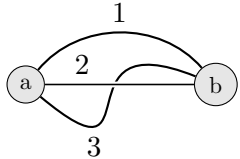
\begin{figure}[H]
	\centering
\begin{tikzpicture}[auto, thick, bend angle=50]
	\node[cblack] (a) at (0,0) {a};
	\node[cblack] (b) at ($(a)+(0:2.5)$) {b};
	
	\draw (a) to [bend left] node {1} (b);
	\draw (a) to [out=320,in=158,looseness=2.5] node[near start, below] {3} (b) ;
	\draw[draw=white,double=black] (a) to node[near start]{2} (b);
\end{tikzpicture}
\caption{Brauer graph $G_4$.} \label{Kalck Example}
\end{figure}

We set $m(a) = m(b) =1$. The cyclic ordering at vertex $a$ is given by $1 < 2 < 3 < 1$ and at vertex  $b$ by $ 1 < 2 < 3 < 1$ and val($a$) = val($b$)  = 3.

}
\end{Examples}

\subsection{Quivers from Brauer graphs}

Let $a$ be a vertex of some Brauer graph $G$, and let $i_1, i_2, \ldots, i_n$ be the edges in $G$ incident with $a$ (note that we might have $i_j = i_k$, for some $j, k$, if the corresponding edge is a loop). 
Suppose that the cyclic ordering at $a$ is given by  $i_1 < i_2 < \ldots < i_n < i_1$. We call $i_1, i_2, \ldots, i_n$ the {\it successor} sequence at $a$ and $i_{j+1}$ is the successor of $i_{j}$, for $ 1 \leq j \leq n-1$ and $i_1$ is the successor of $i_n$, and $i_{k-1}$ is the {\it predecessor} of $i_k$ for $ 2 \leq k \leq n$ and $i_n$ is the predecessor of $i_1$.  

Note that if $v$ is a vertex at edge $i$  with ${\rm val}(v) =1$ and if $m(v)  >1$ then $i > i$ and the successor (and predecessor) of $i$ is $i$, if $m(v)  =1$ then $i$ does not have a successor or predecessor.  

In order to specify more precisely the cyclic ordering at  a given vertex, especially if there is a loop at that vertex,  half-edges are often used. 
In particular, the language of {\it ribbon graphs}, see Section~\ref{Section Ribbon Graphs}, introduced to Brauer graph algebras in \cite{MS}, has the advantage of making the notion of half-edge and their cyclic orderings  very precise.

Given a Brauer graph $G = (G_0, G_1, m, \calo)$ we define a quiver $Q_G = (Q_0, Q_1)$ in the following way. The set of vertices $Q_0$ is given by the set of edges $G_1$ of $G$, denoting the vertex in $Q_0$ corresponding to the edge $i$ in $G_1$ also by $i$. The arrows in $Q$ are induced by the orientation $\calo$. More precisely, let $i$ and $j$ be two edges in $G_0$ incident with a common vertex $v$ and such that $j$ is a direct successor of $i$ in the cyclic ordering of the edges at $v$. Then there is an arrow $\alpha : i \to j$ in $Q_G$.  

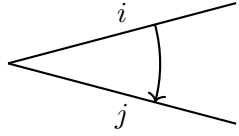
\begin{figure}[H]
	\centering
\begin{tikzpicture}[auto, thick]
	\draw (0,0) to [above] node {$i$} (3,0.8);
	\draw (0,0) to [below] node {$j$} (3,-.8);
	\draw[->] (0,0) +(15:2cm) arc (15:-15:2cm);
\end{tikzpicture}
\caption{Arrow $\alpha$ defined by successor relation $(i,j)$.}
\end{figure}

We say that $\alpha$ is given by the \textit{successor relation $(i,j)$}. 
Since every arrow of $Q_G$ starts and ends at an edge of $G$, there are at most two arrows starting and ending at every vertex of $Q_G$. Every vertex $v \in G_0$  such that $m(v) {\rm val}(v) \geq 2$, gives rise to an oriented cycle $C_v$ in $Q_G$, which is unique up to cyclic permutation.  We call $C_v$ a {\it special cycle at $v$}. Let $C_v$ be such a special cycle at $v$. Then if $C_v$ is a representative in its cyclic permutation class such that $s(C_v) = i = t(C_v)$, $i \in Q_0$,  we say that $C_v$ is a {\it a special $i$-cycle at $v$}. To simplify notation we will simply write $C_v$ for the special $i$-cycle at $v$ and specify if necessary that $s(C_v) = i$ or if more detail is needed, we will specify the first or last arrow in $C_v$.

For $i \in Q_0$ and $v \in G_0$, a special $i$-cycle at $v$ is not necessarily unique,  however, there are at most two special $i$-cycles at $v$ for any $i \in G_1$ and $v \in G_0$ and this happens exactly when $i$ is a loop. An example of this is given in Example~\ref{Example Special Cycles}(2).  

\begin{Examples}\label{Example Special Cycles} {\rm In this example we give the quivers of the Brauer graphs in Example~\ref{Examples Brauer graph}  and list some of the special cycles. 

(1) 

\begin{figure}[H]
	\centering
\begin{tikzpicture}[auto, thick,->]
\node (1) at (0,0) {1};
\node (3) at ($(1)+(40:2.5)$) {3};
\node (2) at ($(3)+(140:2.5)$) {2};
\node (4) at ($(3)+(40:2.5)$) {4};
\node (6) at ($(3)+(-40:2.5)$) {6};
\node (5) at ($(6)+(+40:2.5)$) {5};

\draw (1) to [out=300,in=240,looseness=5] node[pos=0.15] {$\varepsilon$} (1);
\draw (1) to node {$\alpha_1$} (2);
\draw (2) to node {$\alpha_2$} (3);
\draw (3) to node {$\alpha_3$} (1);
\draw (3) to node {$\beta_1$} (4);
\draw (4) to node {$\beta_2$} (5);
\draw (5) to node {$\beta_3$} (6);
\draw (6) to node {$\beta_4$} (3);

\end{tikzpicture}
\caption{Quiver $Q_{G_1}$ of $G_1$.}
\end{figure}
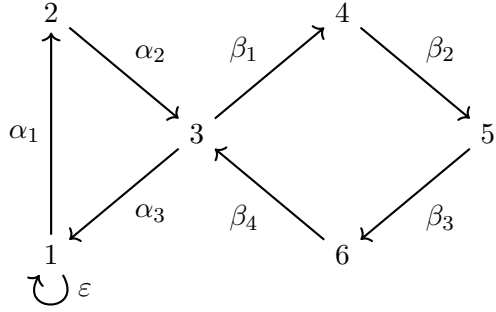

Special cycles in $Q_{G_1}$ are given by $C_a = \epsilon$, the special 1-cycle at $c$ given by $ \alpha_1 \alpha_2 \alpha_3$, the the special 2-cycle at $c$ given by $\alpha_2 \alpha_3 \alpha_1$, the special 3-cycle at $c$ given by $\alpha_3 \alpha_1 \alpha_2$, and the special 3-cycle at $d$ given by  $\beta_1 \beta_2 \beta_3 \beta_4$, etc.

(2) 

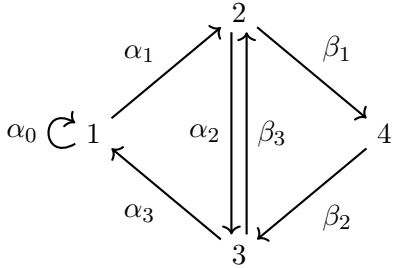
\begin{figure}[H]
\centering
\begin{tikzpicture}[auto, thick,->]
\node (1) at (0,0) {1};
\node (2) at ($(1)+(40:2.5)$) {2};
\node (3) at ($(1)+(-40:2.5)$) {3};
\node (4) at ($(2)+(-40:2.5)$) {4};

\draw (1) to [out=210,in=150,looseness=5] node[left] {$\alpha_0$} (1);

\draw (1) to node {$\alpha_1$} (2);
\draw[transform canvas={xshift=-1mm}] (2) to node[left] {$\alpha_2$} (3);
\draw (3) to node {$\alpha_3$} (1);

\draw (2) to node {$\beta_1$} (4);
\draw (4) to node {$\beta_2$} (3);
\draw[transform canvas={xshift=1mm}] (3) to node[right] {$\beta_3$} (2);

\end{tikzpicture}
\caption{Quiver $Q_{G_2}$ of $G_2$.}
\end{figure}

Special cycles in $Q_{G_2}$ are the special 1-cycles at $a$ corresponding to $ \alpha_0 \alpha_1 \alpha_2 \alpha_3$, and $ \alpha_1 \alpha_2 \alpha_3 \alpha_0$, the special 2-cycle at $a$ given by $\alpha_2 \alpha_3 \alpha_0 \alpha_1$, the special 2-cycle at $b$ given by $ \beta_1 \beta_2 \beta_3$, etc. Note that  there are two distinct 1-cycles at the vertex $a$, namely $\alpha_0 \alpha_1 \alpha_2 \alpha_3$ and $\alpha_1 \alpha_2 \alpha_3 \alpha_0$.

(3) 

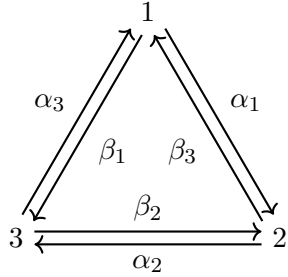
\begin{figure}[H]
\centering
\begin{tikzpicture}[auto, thick,->]
  \node (1) at (90:2) {1};
  \node (2) at (-30:2) {2};
  \node (3) at (210:2) {3};

  \draw[transform canvas={shift={(30:0.085)}}] (1) to  node {$\alpha_1$} (2);
  \draw[transform canvas={shift={(210:0.085)}}] (2) to node {$\beta_3$} (1);

  \draw[transform canvas={shift={(-90:0.085)}}] (2) to node {$\alpha_2$} (3);
  \draw[transform canvas={shift={(90:0.085)}}] (3) to node {$\beta_2$} (2);

  \draw[transform canvas={shift={(150:0.085)}}] (3) to node {$\alpha_3$} (1);
  \draw[transform canvas={shift={(-30:0.085)}}] (1) to node {$\beta_1$} (3);

\end{tikzpicture}
\caption{Quiver $Q_{G_3}$ of $G_3$.}
\end{figure}

The special 1-cycle at $a$ is given by $\alpha_1 \alpha_2 \alpha_3$, the special 1-cycle at $b$ is given by $\beta_1 \beta_2 \beta_3$, etc.

(4) 

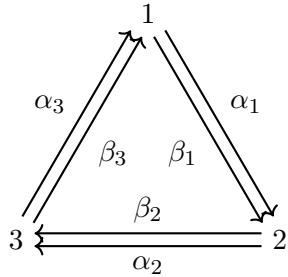
\begin{figure}[H]
	\centering
	\begin{tikzpicture}[auto, thick,->]
	\node (1) at (90:2) {1};
	\node (2) at (-30:2) {2};
	\node (3) at (210:2) {3};
	
	\draw[transform canvas={shift={(30:0.085)}}] (1) to  node {$\alpha_1$} (2);
	\draw[transform canvas={shift={(210:0.085)}}] (1) to node[swap] {$\beta_1$} (2);
	
	\draw[transform canvas={shift={(-90:0.085)}}] (2) to node {$\alpha_2$} (3);
	\draw[transform canvas={shift={(90:0.085)}}] (2) to node[swap] {$\beta_2$} (3);
	
	\draw[transform canvas={shift={(150:0.085)}}] (3) to node {$\alpha_3$} (1);
	\draw[transform canvas={shift={(-30:0.085)}}] (3) to node[swap] {$\beta_3$} (1);
	
	\end{tikzpicture}
	\caption{Quiver $Q_{G_4}$ of $G_4$.}
\end{figure}

The special 1-cycle at $a$ is given by $\alpha_1 \alpha_2 \alpha_3$, the special 1-cycle at $b$ is given by $\beta_1 \beta_2 \beta_3$, etc.

}
\end{Examples}

\subsection{Set of relations and definition of  Brauer graph algebras} We define an ideal of relations $I_G$ in $KQ_G$ generated by three types of relations. For this recall that we identify the set of edges $G_1$ of a Brauer graph $G$ with the set of vertices $Q_0$ of the corresponding quiver $Q_G$ and that we denote the set of vertices of the Brauer graph by $G_0$.

{\it Relations of type I:}     $$C_v^{m(v)} - {C}_{v'}^{m(v')},$$ for any $i \in Q_0$ and for any special $i$-cycles $C_v$ and $C_{v'}$ at $v$ and $v'$ such that both $v$ and $v'$ are not truncated (i.e. ${\rm val}(v) m(v) \neq 1$ and ${\rm val}(v') m(v') \neq 1$).   

{\it Relations of type II:}  $$C_v^{m(v)} \alpha_1,$$ for any $i \in Q_0$, any $v \in G_0$ and where $C_v = 
\alpha_1 \alpha_2 \ldots \alpha_n$ is any special $i$-cycle $C_v$. 

 {\it Relations of type III:} $$\alpha \beta,$$ for any $\alpha, \beta \in Q_1$ such that $\alpha \beta$ is not a subpath of any special cycle except if $\alpha = \beta$ is a loop associated to a vertex $v$ of valency one and multiplicity $m(v) >1$.

The algebra $A_G= KQ_G/I_G$ is called the {\it Brauer graph algebra} associated to the Brauer graph $G$. 

We note that the relations generating $I_G$ do not constitute a minimal set of relations. Many of the relations, in particular many of the relations of type II, are redundant. In \cite{GSST} for every Brauer graph algebra a minimal set of relations is determined: The relations of type I and III are always minimal and the only relations of type II appearing in a minimal generating set of relations are those corresponding to an edge $i$ with vertices $v$ and $w$, such that  $i$ is truncated at vertex $v$ and such that the  immediate  successor of $i$ in the cyclic ordering at $w$ is also truncated at its other endpoint.

\begin{Examples}
{\rm Sets of relations for the examples in~\ref{Examples Brauer graph}.   

(1) Set of relations of the three types defining the Brauer graph algebra $B_1 = KQ_{G_1}/I_{G_1}$: 
 
{\it Type I:} $(\alpha_1 \alpha_2 \alpha_3)^3 - \epsilon^2, (\alpha_3 \alpha_1 \alpha_2)^3 - \beta_1\beta_2\beta_3 \beta_4$

 {\it Type II:} \sloppy $(\alpha_1 \alpha_2 \alpha_3)^3 \alpha_1, (\alpha_2 \alpha_3 \alpha_1)^3 \alpha_2, (\alpha_3 \alpha_1 \alpha_2)^3 \alpha_3$, $ \beta_1 \beta_2 \beta_3\beta_4\beta_1 $, $\beta_2 \beta_3\beta_4\beta_1\beta_2 $, $ \beta_3\beta_4\beta_1\beta_2\beta_3, $ \newline $ \beta_4\beta_1\beta_2\beta_3\beta_4, \epsilon^3 $    

 {\it Type III:} $\epsilon \alpha_1, \alpha_3 \epsilon, \alpha_2\beta_1, \beta_4 \alpha_3$
 
 A minimal set of relations is given by all relations of types I and III and the relations $\beta_2 \beta_3\beta_4\beta_1\beta_2 $ and  $ \beta_3\beta_4\beta_1\beta_2\beta_3$.

In the following, by abuse of notation, we will abbreviate the relations of type II as follows, for $\beta_i \beta_{i+1} \beta_{i+2} \beta_{i+3} \beta_{i+4}$ we write $\beta^5$, etc.

(2)  Set of relations of the three types defining the Brauer graph algebra $B_2 = KQ_{G_2}/I_{G_2}$: 

{\it Type I:} $\alpha_0 \alpha_1 \alpha_2 \alpha_3 - \alpha_1 \alpha_2 \alpha_3 \alpha_0$, $\alpha_2 \alpha_3 \alpha_0 \alpha_1 - \beta_1\beta_2\beta_3, \alpha_3 \alpha_0 \alpha_1 \alpha_2 - \beta_3 \beta_1 \beta_2$ 

{\it Type II:} $\alpha^5, \beta^4$  

{\it Type III:} $\alpha_0^2, \alpha_1 \beta_1, \alpha_2 \beta_3, \alpha_3 \alpha_1, \beta_2 \alpha_3, \beta_3 \alpha_2$

A minimal set of relations is given by all relations of types I and III. 

(3)  Set of relations of the three types defining the Brauer graph algebra $B_3 = KQ_{G_3}/I_{G_3}$: 

{\it Type I:} $\alpha_1 \alpha_2 \alpha_3 - \beta_1 \beta_2 \beta_3$, $\alpha_2 \alpha_3 \alpha_1 - \beta_3 \beta_1 \beta_2$, $\alpha_3 \alpha_1 \alpha_2 - \beta_2 \beta_3 \beta_1$

{\it Type II:} $\alpha^4$, $\beta^4$

{\it Type III:} $\alpha_1 \beta_3, \beta_3 \alpha_1, \alpha_2 \beta_2, \beta_2 \alpha_2, \alpha_3 \beta_1, \beta_1 \alpha_3$ 

A minimal set of relations is given by all relations of types I and III.

(4) Set of relations of the three types defining the Brauer graph algebra $B_4 = KQ_{G_4}/I_{G_4}$: 

{\it Type I:} $\alpha_1 \alpha_2 \alpha_3 - \beta_1 \beta_2 \beta_3$, $\alpha_2 \alpha_3 \alpha_1 - \beta_2 \beta_3 \beta_1$, $\alpha_3 \alpha_1 \alpha_2 - \beta_3 \beta_1 \beta_2$

{\it Type II:} $\alpha^4$, $\beta^4$

{\it Type III:} $\alpha_i \beta_{i+1}, \beta_i \alpha_{i+1}$ for $i  = 1,2$ and $\alpha_3 \beta_1, \beta_3 \alpha_1$

A minimal set of relations is given by all relations of types I and III.

}
\end{Examples}

\begin{remark} {\rm 
Any symmetric Nakayama algebra is a Brauer graph algebra. Let $Q$ be the quiver given by a non-zero cycle $\alpha_1 \alpha_2 \ldots \alpha_n$ with $s(\alpha_1) = t(\alpha_n)$. Let $I = (\alpha^{kn+1})$ for any $k \in \mathbb{Z}_{>0}$. Then $A= KQ/I$ is a Brauer tree algebra with Brauer tree $T$ given  by a star with $n$ edges and multiplicity 1 everywhere, except for the central vertex whose multiplicity is equal to $k$. 

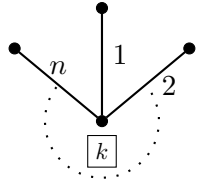
\begin{figure}[H] 
\centering
\begin{tikzpicture}[auto, thick]
	\node[cb] (0) at (0,0) {};
	\node[cb] (1) at ($(0)+(140:1.5)$) {};
	\node[cb] (2) at ($(0)+(90:1.5)$) {};
	\node[cb] (3) at ($(0)+(40:1.5)$) {};
	\node[cw] () at  ($(0)+(-90:.4)$) {$k$};
	\draw (0) to [above, pos=.5] node {$n$}(1);
	\draw (0) to [right, pos=.6] node {$1$}(2);
	\draw (0) to [below, pos=.8] node {$2$}(3);
	\draw[loosely dotted] (0)+(-217:.75) arc (-217:37:.75){};
	
\end{tikzpicture}
	\caption{Brauer graph of a Nakayma algebra with $n$ arrows $\alpha_1, ..... , \alpha_n$ and relations of the form 
	$(\alpha_{i_1} \ldots \alpha_{i_n})^k\alpha_{i_1}$.}
\end{figure}

}
\end{remark}

\subsection{First properties of Brauer graph algebras}

\begin{theorem}
Given a Brauer graph $G= (G_0, G_1, m, \calo)$, the associated Brauer graph algebra $A_G = K_G/I_G$ is a finite dimensional symmetric algebra.
\end{theorem}

\begin{proof}
We start by showing that $I_G$ is admissible. 
 By the definition of a special cycle $C_v$, we have that either $\ell(C_v) >1$ or $m(v) >1$. Thus $I_G \subset KQ^2$. It follows from the relations of type  I, II and III that,  for  $N = {\rm max}_{v \in G_0}\ell( C_v^{m(v)})$, any path $p$ in $Q$ with  $\ell(p) \geq N+1$ is such that   $p \in  I_G$.  
 
 To show that $A_G$ is symmetric, define the following symmetric linear function $f: A_G \to K$ by setting
$$ f(p) = \left\{ \begin{array} {ll}
1 \mbox{ if } p = C_v^{m(v)} \mbox{ for some $v \in G_0$}  \\
0 \mbox{ otherwise.}
\end{array} \right.$$ 
Then it has no non-zero left-ideal in the kernel and thus $A_G$ is symmetric.
\end{proof}

\begin{remark}
{\rm A Brauer graph algebra is indecomposable if and only if its Brauer graph is connected as a graph. }
\end{remark}

The following Theorem is well-known. It also follows from recent work on special multiserial and Brauer configuration algebras in  \cite{GS1, GS2}.

\begin{theorem}\label{BGAspecialbiserial}  Brauer graph algebras are special biserial. 
\end{theorem}

\begin{corollary}
Brauer graph algebras are of tame representation type. A Brauer graph algebra is of finite representation type  if and only if it is a Brauer tree algebra. 
\end{corollary}

\begin{remark}{\rm 
It is possible to introduce scalars in the relations of a Brauer graph algebra. In this case, the resulting Brauer graph algebra is no longer necessarily  symmetric, but instead it is weakly symmetric. The case of a weakly symmetric Brauer graph algebra is treated in detail, for example, in \cite{GSS} and \cite{GSST}. We note that many interesting examples, such as the existence of non-periodic modules with complexity one  appear precisely in the weakly symmetric (non-symmetric) case, see  \cite{Erdmann}. However, in these notes we will focus on the case of symmetric Brauer graph algebras and refer for a rigorous definition of weakly-symmetric Brauer graph algebras to \cite{GSS}.}
\end{remark}

\subsection{Indecomposable projective modules of Brauer graph algebras}

We now list the Loewy series of the projective indecomposable modules of the Brauer graph algebras in the examples above. Since Brauer graph algebras are special biserial, they are biserial and the composition factors of the maximal uniserial submodules of the  indecomposable projectives can conveniently be read off the Brauer graph. In the following the projective indecomposable module at vertex $i$ of $Q_0$ will be denoted by $P_i$ and we denote by $i$ the simple module at vertex $i$.

\begin{Examples} {\rm Let $B_1, \ldots, B_4$ be the Brauer graph algebras with Brauer graphs $G_1, \ldots, G_4$ respectively, given in Example~\ref{Examples Brauer graph}. Then indecomposable projective modules are given by

(1) 

 \begin{figure}[H]
	\centering	
	\begin{tikzpicture}[auto, thick]
	\node (P1) at (0,0) {$P_1$:};
	\node (P11) at (1,0) {$\matrixTwo{1}{1}{\dimEight23123123}{1}$};
	
	\node (P2) at (2.5,0) {$P_2$:};
	\node (P21) at (3.2,0) {$
		\begin{smallmatrix}
		2 \\
		3 \\
		1 \\
		2 \\
		3 \\
		1 \\
		2 \\
		3 \\
		1 \\
		2
		\end{smallmatrix}
		$};
	
	\node (P3) at (4.7,0) {$P_3$:};
	\node (P31) at (5.7,0) {$\matrixTwo{3}{\dimEight12312312}{\dimThree456}{3}$};
	
	\node (P4) at (7.2,0) {$P_4$:};
	\node (P41) at (8,0) {$\dimFive45634$};
	
	\node (P5) at (9.5,0) {$P_5$:};
	\node (P51) at (10.3,0) {$\dimFive56345$};
	
	\node (P6) at (11.8,0) {$P_6$:};
	\node (P61) at (12.6,0) {$\dimFive63456$};
	
	\end{tikzpicture}
	\caption{Indecomposable projective modules over $B_1$.}
\end{figure}

(2)  

\begin{figure}[H]
	\centering	
	\begin{tikzpicture}[auto, thick]
	\node (P1) at (0,0) {$P_1$:};
	\node (P11) at (1,0) {$\matrixTwo{1}{\dimThree123}{\dimThree231}{1}$};
	
	\node (P2) at (3,0) {$P_2$:};
	\node (P21) at (4,0) {$\matrixTwo{2}{\dimThree311}{\dimTwo43}{2}$};
	
	\node (P3) at (6,0) {$P_3$:};
	\node (P31) at (7,0) {$\matrixTwo{3}{\dimThree112}{\dimTwo24}{3}$};
	
	\node (P4) at (9,0) {$P_4$:};
	\node (P41) at (10,0) {$\dimFour4324$};
	
	\end{tikzpicture}
	\caption{Indecomposable projective modules over $B_2$.}
\end{figure}

(3)~\begin{figure}[H]
	\centering	
	\begin{tikzpicture}[auto, thick]
	\node (P1) at (0,0) {$P_1$:};
	\node (P11) at (1,0) {$\matrixTwo{1}{\dimTwo23}{\dimTwo32}{1}$};
	
	\node (P2) at (3,0) {$P_2$:};
	\node (P21) at (4,0) {$\matrixTwo{2}{\dimTwo31}{\dimTwo13}{2}$};
	
	\node (P3) at (6,0) {$P_3$:};
	\node (P31) at (7,0) {$\matrixTwo{3}{\dimTwo12}{\dimTwo21}{3}$};

	\end{tikzpicture}
	\caption{Indecomposable projective modules over $B_3$.}
\end{figure}

(4)  

\begin{figure}[H]
	\centering
	
	\begin{tikzpicture}[auto, thick]
	\node (P1) at (0,0) {$P_1$:};
	\node (P11) at (1,0) {$\matrixTwo{1}{\dimTwo23}{\dimTwo23}{1}$};
		
	\node (P2) at (3,0) {$P_2$:};
	\node (P21) at (4,0) {$\matrixTwo{2}{\dimTwo31}{\dimTwo31}{2}$};
	
	\node (P3) at (6,0) {$P_3$:};
	\node (P31) at (7,0) {$\matrixTwo{3}{\dimTwo12}{\dimTwo12}{3}$};
		
	\end{tikzpicture}
	\caption{Indecomposable projective modules over $B_1$.}
\end{figure}

}
\end{Examples}

\subsection{Brauer graphs as ribbon graphs}\label{Section Ribbon Graphs}

In \cite{MS} the point of view of interpreting a Brauer graph as a ribbon graph has been introduced. A ribbon graph uniquely defines a compact oriented surface and connects the representation theory of Brauer graph algebras to combinatorial surface geometry. We recall the details of this construction. More details on ribbon graphs and associated surfaces, as well as the proofs of the below statements can be found, for example, in \cite{Labourie}.

In order to define ribbon graphs, we start by giving a more formal definition of a graph. 

A  {\it graph} is a tuple $\Gamma=(V, E, s, \iota)$ where 
\begin{enumerate}
\item $V$ is a finite set, called the set of vertices, 
\item $E$ is  finite set, called the set of half-edges
\item $s: E \to V$ is a function, 
\item $ \iota: E \to E$ is an involution without fixed points.
\end{enumerate}

We think of $s$ as the function sending a half-edge to the vertex it is incident with and for an edge that is not a loop, we think of $\iota$ as the function sending a half-edge incident with a vertex to the  half-edge incident with the other vertex  of the same edge. In the case of a loop, $\iota$ sends a half edge incident with the only vertex of the corresponding edge to the other half-edge of the same edge. That is,  a loop in $G$ corresponds to a pair of half-edges $e, e' \in E$ such that $\iota(e) = e'$ and $s(e) = s(\iota(e'))$.

A \emph{ribbon graph} is a graph $\Gamma=(V, E, s, \iota)$ together with a permutation $\sigma: E \to E$ such that the cycles of $\sigma$ correspond to the sets $s^{-1} (v)$, for all $v \in V$. 
In other words a ribbon graph is a graph together with, for every vertex $v \in V$, a cyclic ordering of the half-edges incident with $v$.

\begin{remark} {\rm

1) It follows from the above definitions that a Brauer graph can be regarded
as a (vertex) weighted ribbon graph where the weights are given by the multiplicity function.  

2) The language of ribbon graphs in the context of Brauer graph algebras is particularly useful in the treatment of loops in the Brauer graph. 
}
\end{remark}

Examples of ribbon graphs are planar graphs and graphs locally embedded in
the (oriented) plane such as the Brauer graphs in Example~\ref{Examples Brauer graph}. Note that as abstract graphs $G_3$ and $G_4$ in Example~\ref{Examples Brauer graph} are isomorphic, but that as  ribbon graphs they are not isomorphic. 

 Any embedding of a graph into an oriented surface  gives a
ribbon graph structure on the graph, where the cyclic orderings are induced from the
embedding of the edges around each vertex and the orientation of the surface.

On the other hand, we obtain surfaces from ribbon graphs. For this we first define the geometric realisation of a ribbon graph and then we embed it into an open oriented surface. 

The {\it geometric realisation $\vert \Gamma \vert$} of a ribbon graph $\Gamma$ is the topological space $$\vert \Gamma \vert = (E \times [0,1] )/\sim$$ where $\sim$ is the equivalence relation defined by $(e,t) \sim (\iota(e), 1-t$ for all $t \in [0,1]$ and 
$(e,0) \sim (f,0)$ if $s(e) = f(e)$ and $(e,1) \sim (f,1)$ if $s(\iota(e)) = f(\iota(e))$, for all 
$e,f, \in E$.

\begin{lemma} \cite[2.2.4]{Labourie}
Every ribbon graph $\Gamma$ can be embedded in an open oriented surface with boundary
in such a way that the cyclic orderings around each of its vertices arise from
the orientation of the surface.
\end{lemma}

An example of such a surface is the  \emph{ribbon surface} $S^{\circ}_{\Gamma}$ of $\Gamma$ which is constructed in the following way:  Firstly, an oriented surface is
associated to each vertex and edge of $\Gamma$ as in
Figure~\ref{fig:localsurface}.  Then the structure and orientation of the ribbon graph determines a
gluing of the corresponding surfaces giving an oriented surface  $S^{\circ}_{\Gamma}$ together with an
embedding of $\Gamma$ into $S^{\circ}_{\Gamma}$.

\begin{figure}[H]
	\centering
	\tikzstyle{lines}+=[dashed]
	\tikzstyle directed=[postaction={decorate,decoration={markings,
	    mark=at position .55 with {\arrow[]{Stealth}}}}]
\subfigure{	
	\begin{tikzpicture}[auto, thick, scale=1.1]
	\draw[directed] (0,0.5)--(0,1.5);
	\draw[directed] (0,1.5)--(3,1.5);
	\draw[directed] (3,1.5)--(3,0.5);
	\draw[directed] (3,0.5)--(0,0.5);
	\draw[style=lines] (0.15,1) --(2.85,1);
	\end{tikzpicture}}
\hspace{2cm}
\subfigure{	
	\begin{tikzpicture}[auto, thick, scale=.6, rotate=-60]
	\node (0) at (0,0){};
	\draw[style=lines] (0) circle (0.3cm);
	\draw ($(0)+(-20:2)$) to node {}($(0)+(20:2)$);
	\draw ($(0)+(100:2)$) to node {}($(0)+(140:2)$);
	\draw ($(0)+(220:2)$) to node {}($(0)+(260:2)$);
	\draw[directed] ($(0)+(100:2)$) to  [bend right] node {}($(0)+(20:2)$);
	\draw[directed] ($(0)+(220:2)$) to  [bend right] node {}($(0)+(140:2)$);
	\draw[directed] ($(0)+(-20:2)$) to  [bend right] node {} ($(0)+(260:2)$);
	\draw[style=lines] ($(0)+(0:0.35)$) to node {}($(0)+(0:1.88)$);
	\draw[style=lines] ($(0)+(120:0.35)$) to node {}($(0)+(120:1.88)$);
	\draw[style=lines] ($(0)+(240:0.35)$) to node {}($(0)+(240:1.88)$);
\end{tikzpicture}}	
	\caption{Oriented surfaces corresponding to vertices and edges in $\Gamma$.}\label{fig:localsurface}
\end{figure}
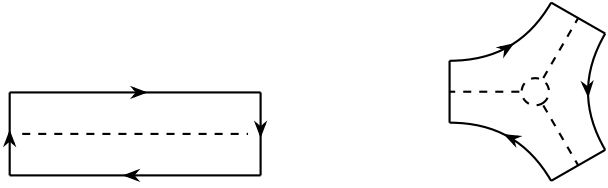

Note that $S^{\circ}_{\Gamma}$ has a number of boundary components corresponding to the faces of $\Gamma$. A
\emph{face} of $\Gamma$ is an equivalence class, up to cyclic permutation, of
$n$-tuples $(e_1,e_2,\ldots ,e_n)$ of half-edges satisfying 
$$
e_{p+1} = \left\{ \begin{array}{lll}
\iota (e_p) & \mbox{ if } & s(e_p ) = s(e_{p-1}) \\ 
\sigma (e_p) &\mbox{ if }& s(e_p ) \neq s(e_{p-1})\\
\end{array} \right.
$$
for all $p$ with $1\leq p\leq n$ (where subscripts are taken
 modulo $n$).

Now let $S$ be a compact oriented surface. An embedding $\Gamma \hookrightarrow S$ is said to be a \emph{filling
  embedding} if
$S\setminus \vert \Gamma \vert =\bigsqcup_{f\in F} D_f,$ where 
each $D_f$ is a disc, and $F$ is a finite set. That is, the
complement of the embedding is a disjoint union of finitely many discs.

\begin{proposition}~\cite[2.2.7]{Labourie} \label{pro:embeddingexists}
Every ribbon graph has a filling embedding  into a compact oriented surface such
that the connected components of $S\setminus \vert \Gamma \vert $ are in bijection with the
faces of $\Gamma$ in the above sense.
\end{proposition}

This is proved by gluing discs onto the ribbon surface of the ribbon graph to
fill in the boundary components. Such an embedding has the following uniqueness
property. We first recall that morphisms
and isomorphisms of ribbon graphs are defined in the natural way, that is they are maps of graphs respecting the permutations.

\begin{proposition}~\cite[2.2.10]{Labourie} \label{pro:fillinguniqueness}
Let $ \Gamma\rightarrow S$ and  $ \Gamma'\rightarrow S'$ be filling embeddings of
ribbon graphs $\Gamma$ and $\Gamma'$ into compact oriented surfaces $S$ and $S'$ and let $f:\Gamma\rightarrow
\Gamma'$ be an isomorphism of ribbon graphs.  Then $f$ induces an
orientation-preserving homeomorphism $f: \vert \Gamma \vert  \rightarrow \vert \Gamma' \vert$
extending to a homeomorphism from $S$ to $S'$.
\end{proposition}

\begin{corollary} \label{c:fillingembedding}
If $\Gamma$ is a ribbon graph, then there is a compact oriented surface
$S_{\Gamma}$ together with a filling embedding $\Gamma\rightarrow S_{\Gamma}$,
unique up to homeomorphism.
\end{corollary}

Thus we see that there is a filling embedding of an arbitrary Brauer graph
(without, or ignoring, multiplicity) into a compact oriented surface, unique up to homeomorphism,  in such a way
that the cyclic ordering around each vertex arises from the orientation of
the surface. 
Conversely, we have the following:

\begin{proposition} \cite[2.2.12]{Labourie}
\label{pro:admitsfilling}
Every compact oriented surface admits a filling ribbon graph.
\end{proposition}

\begin{remark}{\rm  (1)
Note that in~\cite{Antipov} there also is a surface  associated to a Brauer graph. This surface arises from  a CW-complex associated to  the corresponding Brauer graph algebra. Comparing the definitions,
it can be seen that this gives rise to the same surface as the ribbon graph construction in 
(Corollary~\ref{c:fillingembedding}) above. In particular, the $G$-cycles
in~\cite[\S2]{Antipov} correspond to the faces of the Brauer graph as a
ribbon graph.

(2) In Sections 3 and 4 (ideal) triangulations of marked oriented surfaces play an important role (we refer the reader to Section~\ref{gentle algebra mutation} or to \cite{FST} for the definition of an ideal triangulation), these are connected to ribbon graphs in the following way. Let $T$ be a triangulation of a marked oriented surface $(S,M)$ where $S$   might have a boundary and punctures. Then $T$ (including the boundary edges) can be considered as a ribbon graph with boundary where the cyclic ordering of the edges around each vertex is induced by the orientation of $S$ and where the boundary of $S$ induces the set of boundary faces of $T$. Then there is an  embedding $T \hookrightarrow S$ mapping the boundary faces to the boundaries of the boundary components, for more details see for example Section 2.2 in \cite{MS}. 

}
\end{remark}

 \section{Brauer graph algebras and gentle algebras}\label{Section gentle}

In this section we give an explicit construction to show how gentle algebras and Brauer graph algebras are related via trivial extensions. We do this by constructing a ribbon graph for every gentle algebra.

\begin{definition}\label{Defgentle}
{\rm A finite dimensional algebra is {\it gentle} if it is Morita equivalent to a special biserial algebra  $A = KQ/I$, that is, (S0) and (S1), hold as well as 

{\bf (S2)} For every arrow $\alpha$ in $Q$ there exists at most one arrow $\beta$ such that $\alpha \beta \in I$ where $t(\alpha)  = s  (\beta)$ and  there exists at most one arrow $\gamma$ such that $\gamma \alpha \in I$ where $t(\gamma) = s(\alpha)$.

{\bf (S3)} The ideal $I$ is generated by paths of length 2. }
\end{definition}

\begin{Example} {\rm Examples of gentle algebras: 

(1) Tilted algebras of type $A$, see for example, \cite[Chapter VIII]{ASS} for a definition of tilted algebras. 

(2) Gentle algebras arising as Jacobian algebras associated to marked oriented surfaces where all marked points lie in the boundary (see \cite{ABCP}). 

(3) Discrete derived algebras as classified in \cite{V}.
}
\end{Example}

The {\it trivial extension} $T(A) = A \ltimes D(A)$ of an algebra $A$ by its injective cogenerator $D(A) = \Hom_K(A, K)$ is the algebra whose underlying $K$-vector space is given by $A \oplus D(A)$ and where the multiplication is given as follows: 
$$ (a,f)(b,g) = (ab, ag + fb) \;\;\;\; \mbox{ for all } a, b \in A \mbox { and } g,f \in D(A).$$

The trivial extension $T(A)$ is symmetric (see for example \cite{Sch}). 

The following result shows the connection between gentle algebras and Brauer graph algebras. 

\begin{theorem}\cite{PS}
Let $A$ be a finite dimensional $K$-algebra. Then $A$ is gentle if and only if $T(A)$ is special biserial. 
\end{theorem}

Therefore if $A$ is gentle, its trivial extension $T(A)$ is a symmetric special biserial algebra, that is, it is a Brauer graph algebra. 

\begin{Question}
Given a gentle algebra, what is the Brauer graph of its trivial extension? 
\end{Question}

We will answer this question in Theorem~\ref{trivialextension}. 

\subsection{Graph of a gentle algebra} In \cite{S}, to any gentle algebra $A$ is associated a ribbon graph $\Gamma_A$ whose underlying graph structure is determined by the set of maximal paths in $A$. The ribbon graph structure in the form of the cyclic orderings of the edges around the vertices are  also induced by the maximal paths. 

Let $A = KQ/I$ and $\cM$ be the set of maximal paths in $A$. That is $\cM$ consists of all images in $KQ/I$ of paths $p \in KQ$ such that $p \notin I$, but  $\alpha p  \in I$ and $ p \alpha \in  I$, for all non-trivial $\alpha \in KQ$. 

Set $\cM_0 = \{ e_v, v \in V_0\}$ where $V_0$ is the set of vertices $v \in Q_0$ such that one of the following holds: 

(1) There exist unique arrows $\alpha$ and $\beta$ such that $t(\alpha) = v = s(\beta)$ and, furthermore,  $\alpha \beta \notin I$. 

(2) $v$ is a source and there is a single arrow starting at $v$. 

(3) $v$ is a sink and there is a single arrow ending at $v$. 

Set $\overline{\cM} = \cM \cup \cM_0$. 

\begin{Example}\label{discretederived}
{\rm Let $A$ be the gentle algebra with quiver 

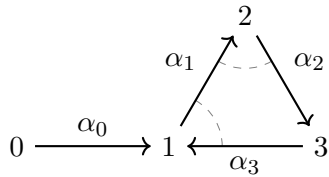
\begin{figure}[H]
\centering
\tikzstyle{help lines}+=[dashed]
\begin{tikzpicture}[auto, thick]+=[dashed]
	\node (1) at (0:0) {1};
	\node (2) at ($(1)+(60:2)$) {2};
	\node (3) at ($(1)+(0:2)$) {3};
	\node (0) at ($(1)+(180:2)$) {0};

	\draw[style=help lines] (1) +(0:.7cm) arc (0:60:.7cm);
	\draw[style=help lines] (2) +(240:.7cm) arc (240:300:.7cm);
	\draw[->] (0) to node {$\alpha_0$} (1);
	\draw[->] (1) to node {$\alpha_1$} (2);
	\draw[->] (2) to node {$\alpha_2$} (3);
	\draw[->] (3) to node {$\alpha_3$} (1);
\end{tikzpicture}
\caption{The dashed lines indicate the relations of length 2.}
\end{figure}

and relations $\alpha_1 \alpha_2$ and $\alpha_3 \alpha_1$. 

Then $V_0 = \{ 0,  3\}$ and  $\overline{\cM} = \{  e_0, \alpha_0 \alpha_1, \alpha_2 \alpha_3, e_3 \}  $.

}
\end{Example}

{\bf Construction of the ribbon graph $\Gamma_A$ of a gentle algebra. }

Let  $A = KQ/I$ be a gentle algebra. 
The vertices of $\Gamma_A$ are in bijection with the elements in $\overline{\cM}$ and the edges of  $\Gamma_A$ are in bijection with the vertices of $Q_0$. Furthermore, there exists an edge
$E_v$ for $v \in Q_0$ between the vertices $m, m' \in \overline{\cM}$ if $v$ as a vertex in $Q_0$ lies in both the paths $m$ and $m'$.  That is if there exists paths $p, q, p', q'$ such that $m = p e_v q$ and $m' = p' e_v q$ where $e_v$ is the trivial path at $v$ and $p, q, p', q'$ might also possibly be trivial paths at $v$. 

We note that every vertex $v$ in $Q_0$ lies in exactly two elements of $\overline{\cM}$ which are not necessarily distinct. If they are not distinct the corresponding edge $E_v$ in $\Gamma_A$ is a loop.  

Note that if a vertex in $\Gamma_A$ corresponds to a trivial path in $Q$ then as a vertex in $\Gamma_A$ it has valency 1, that is, there is only one edge of $\Gamma_A$ incident with that vertex. 
In all other case there are at least two edges incident with every vertex of $\Gamma_A$. 

Let $v$ be a vertex of $\Gamma_A$ such that there are at least two edges of $\Gamma_A$ incident at $v$ and let $a_1 \stackrel{\alpha_{1}}{\longrightarrow} a_2 \stackrel{\alpha_{2}}{\longrightarrow} ... a_{k-1} \stackrel{\alpha_{k-1}}{\longrightarrow} a_k$ be the element in $\overline{\cM}$ corresponding to $v$. 
Let $E_1, E_2, \ldots, E_k$ be the edges in $\Gamma_A$ corresponding to the vertices $a_1, a_2, \ldots, a_k$ respectively. Note that they all have the vertex $v$ in common. Then 
we linearly order $E_1, E_2, \ldots, E_k$ as follows: $E_1 < E_2 < \cdots < E_k$. We complete this linear order to a cyclic order by adding $E_k < E_1$ and we do this at every vertex of $\Gamma_A$. 
This equips the graph $\Gamma_A$ with the structure of a  ribbon graph. 

\begin{Example}
{\rm For the algebra in example~\ref{discretederived} above, we get the the following ribbon graph.

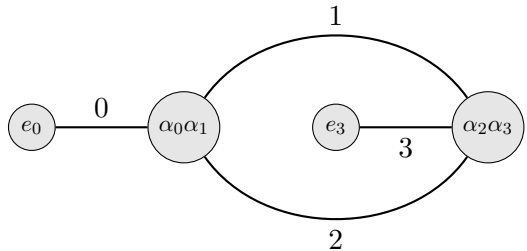
\begin{figure}[H]
	\centering
\begin{tikzpicture}[auto, thick,bend angle=55]
	\node[cblack] (a) at (-2,0) {$e_0$};
	\node[cblack] (b) at ($(a)+(0:2)$) {$\alpha_0\alpha_1$};
	\node[cblack] (c) at ($(b)+(0:4)$) {$\alpha_2\alpha_3$};
	\node[cblack] (d) at ($(b)+(0:2)$) {$e_3$};
	
	\draw (a) to node {$0$}(b);
	\draw (b) to [bend left] node {$1$}(c);
	\draw (b) to [bend right] node[below] {$2$} (c);
	\draw (c) to  node {$3$} (d);
\end{tikzpicture}
\caption{Graph $\Gamma_A$ of the gentle algebra $A$ in Example~\ref{discretederived}.}
\end{figure}

The cyclic ordering here corresponds to the ordering given by the embedding into the clockwise oriented plane induced by the maximal paths containing the given vertex. That is at vertex $\alpha_0 \alpha_1$ we have cyclic ordering $0 < 1< 2 < 0$ at vertex 
$\alpha_2 \alpha_3$ we have cyclic ordering $1 < 2< 3 < 1$, at vertex $e_0$ it is $0$ and at vertex $e_3$ it is $3$. 
}
\end{Example}

\begin{theorem}\label{trivialextension}\cite{S}
Let $A = KQ/I$ be a gentle algebra with ribbon graph $\Gamma_A$. Then $T(A)$ is the Brauer graph algebra with Brauer graph $\Gamma_A$ (and with multiplicity function identically equal to one). 
\end{theorem}

There are many situations when there are naturally associated graphs to gentle algebras, given for example by triangulations or partial triangulations of marked unpunctured surfaces.  These graphs often  coincide with the ribbon graphs defined above. More precisely, using the admissible cuts in the next section we get the following Corollary to Theorem~\ref{trivialextension}. 

\begin{corollary}\label{Cor graphs of gentle algebras}
1) If $A$ is a gentle algebra arising from a triangulation  $T$ of an oriented marked surface with marked points in the boundary as defined in \cite{ABCP} then $\Gamma_A = T$ as ribbon graphs. 

2) If $A$ is a surface algebra as defined in \cite{SR} of a partial triangulation $T$ of an oriented marked surface with marked points in the boundary then $\Gamma_A  = T$ as ribbon graphs. 

3) If $A$ is a tiling algebra as defined in \cite{SimoesParsons} with partial triangulation $T$ then $\Gamma_A = T$. 

\end{corollary}

 We will give three more examples of the graphs of gentle algebras. In particular,  the first two examples illustrate  each of the statements in Corollary~\ref{Cor graphs of gentle algebras}(1) and (3), respectively. In all three examples the orientation of the graph of the corresponding gentle algebra (induced by the maximal paths) is given by the clockwise orientation of the plane. 

\begin{Example}
{\rm Example illustrating Corollary~\ref{Cor graphs of gentle algebras} (1). Let $A= KQ / (\alpha_1 \beta, \beta \gamma, \gamma \alpha_1)$ be the gentle algebra with quiver $Q$ given by  

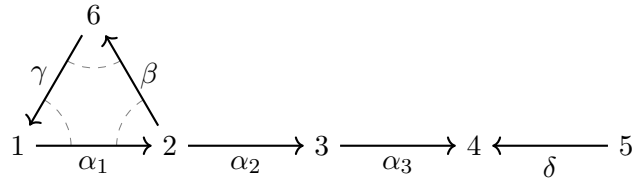
\begin{figure}[H]
\begin{center}
\tikzstyle{help lines}+=[dashed]
\begin{tikzpicture}[auto, thick]+=[dashed]
	\node (1) at (0:0) {1};
	\node (2) at ($(1)+(0:2cm)$) {2};
	\node (3) at ($(2)+(0:2cm)$) {3};
	\node (4) at ($(3)+(0:2cm)$) {4};
	\node (5) at ($(4)+(0:2cm)$) {5};
	\node (6) at ($(1)+(60:2)$) {6};

	\draw[style=help lines] (1) +(0:.7cm) arc (0:60:.7cm);
	\draw[style=help lines] (2) +(120:.7cm) arc (120:180:.7cm);
	\draw[style=help lines] (6) +(240:.7cm) arc (240:300:.7cm);
	
	\draw[->] (1) to [below] node {$\alpha_1$} (2);
	\draw[->] (2) to [below] node {$\alpha_2$} (3);
	\draw[->] (3) to [below] node {$\alpha_3$} (4);
	\draw[->] (5) to node {$\delta$} (4);
	\draw[->] (2) to [pos=0.55, right] node {$\beta$} (6);
	\draw[->] (6) to [pos=0.45, left] node {$\gamma$} (1);
	
\end{tikzpicture}
\caption{Gentle algebra $A$ associated to the triangulation of the 9-gon in Figure~\ref{9-gon}.}
\end{center}
\end{figure}

Then the graph $\Gamma_A$ of $A$ is given by 

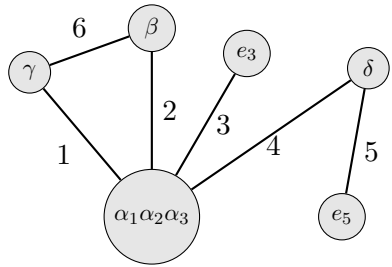
\begin{figure}[H]
\begin{center}
\begin{tikzpicture}[auto, thick]
	\node[cblack] (0) at (-2,0) {$\alpha_1\alpha_2\alpha_3$};
	\node[cblack] (1) at ($(a)+(130:2.5)$) {$\gamma$};
	\node[cblack] (2) at ($(a)+(90:2.5)$) {$\beta$};
	\node[cblack] (3) at ($(a)+(60:2.5)$) {$e_3$};
	\node[cblack] (5) at ($(a)+(0:2.5)$) {$e_5$};
	\node[cblack] (4) at ($(5)+(80:2)$) {$\delta$};
	
	\draw (0) to node {$1$}(1);
	\draw (0) to [right] node {$2$}(2);
	\draw (0) to [right] node {$3$}(3);
	\draw (1) to [pos=0.6] node {$6$}(2);
	\draw (0) to [pos=0.4, right] node {$4$}(4);
	\draw (5) to [pos=0.4, right] node {$5$}(4);
\end{tikzpicture}
\caption{Graph $\Gamma_{A}$ of $A$}
\end{center}
\end{figure}

The graph $\Gamma_{A}$ coincides with the following triangulation of the 9-gon and the gentle algebra $A$ corresponds to the associated Jacobian algebra arising from the associated quiver with potential.

\begin{figure}[H]
\begin{center}
\begin{tikzpicture}[auto, thick,bend angle=55]
	\node (0) at (0,0) {};
	\node[cb] (1) at ($(0)+(180:3)$){};
	\node[cb] (2) at ($(0)+(140:3)$){};
	\node[cb] (3) at ($(0)+(100:3)$){};
	\node[cb] (4) at ($(0)+(60:3)$){};
	\node[cb] (5) at ($(0)+(20:3)$){};
	\node[cb] (6) at ($(0)+(340:3)$){};
	\node[cb] (7) at ($(0)+(300:3)$){};
	\node[cb] (8) at ($(0)+(260:3)$){};
	\node[cb] (9) at ($(0)+(220:3)$){};

	\draw (0) circle (3cm);
	
	\draw[->] (2)+(8:0.75cm) arc (8:-88:.75cm) {node[pos=0.5]{$\gamma$}};
	\draw[->] (4)+(228:1cm) arc (228:192:1cm){node[pos=0.5]{$\beta$}};
	\draw[->] (9)+(88:1.5cm) arc (88:52:1.5cm){node[pos=0.3]{$\alpha_1$}};
	\draw[->] (9)+(48:1.5cm) arc (52:36:1.5cm){node[pos=0.5]{$\alpha_2$}};
	\draw[->] (9)+(28:1.5cm) arc (28:12:1.5cm){node[pos=0.8]{$\alpha_3$}};
	\draw[<-] (6)+(192:1.5cm) arc (192:208:1.5cm){node[pos=0.8, left]{$\delta$}};
	
	\draw (9) to [pos=.55] node {$1$}(2);
	\draw (9) to [pos=.5] node {$2$}(4);
	\draw (9) to [pos=.55] node {$3$}(5);
	\draw (9) to [pos=.55] node {$4$}(6);
	\draw (8) to [pos=.3] node {$5$}(6);
	\draw (2) to node {6}(4);
\end{tikzpicture}
\caption{Triangulation of the 9-gon corresponding to   $\Gamma_{A}$.}\label{9-gon}
\end{center}
\end{figure}
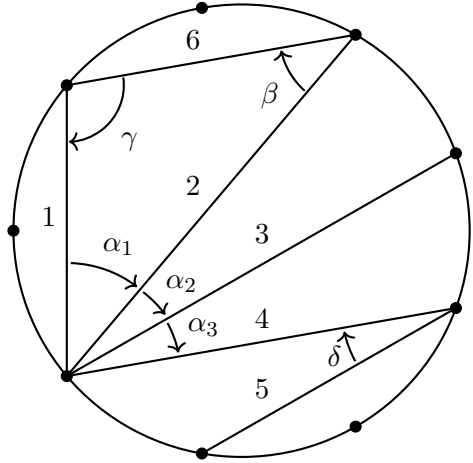

}
\end{Example}

\begin{Example}
{\rm Example illustrating Corollary~\ref{Cor graphs of gentle algebras} (3):
Example of a tiling algebra as defined in \cite{SimoesParsons} (and using the notation of \cite{SimoesParsons}). In Figure~\ref{A10} is the algebra $A_{10}$ corresponding to a tiling of $P_{10,1}$ as in Figure~\ref{Tiling P10}. The algebra $A_{10}$ is defined to be the algebra $KQ_{10} /I$ where $Q_{10}$ is the quiver in Figure~\ref{A10} and where $I= (\alpha_1 \beta_3, \alpha_2 \gamma_1, \gamma_1 \delta_2, \delta_2 \beta_2, \beta_2 \alpha_2, \delta_1 \gamma_2, \epsilon \delta_1)$. 

\begin{figure}[H]
\centering
\tikzstyle{help lines}+=[dashed]
\begin{tikzpicture}[auto, thick]+=[dashed]
\node (1) at (0,0) {1};
\node (2) at ($(1)+(-30:2.5)$) {2};
\node (3) at ($(2)+(30:2.5)$) {3};
\node (4) at ($(2)+(210:2.5)$) {4};
\node (5) at ($(4)+(-30:2.5)$) {5};
\node (6) at ($(5)+(30:2.5)$) {6};
\node (7) at ($(6)+(30:2.5)$) {7};
\node (8) at ($(5)+(210:2.5)$) {8};
\node (9) at ($(5)+(-30:2.5)$) {9};
\node (10) at ($(8)+(-30:2.5)$) {10};

\draw[->] (1) to [pos=.4, below] node {$\alpha_1$} (2);
\draw[->] (2) to [pos=.6, below] node {$\beta_3$} (3);
\draw[->] (2) to [pos=.6, above] node {$\alpha_2$} (4);
\draw[->] (4) to [pos=.4, below] node {$\gamma_1$} (5);
\draw[->] (5) to [pos=.6, below] node {$\delta_2$} (6);
\draw[->] (5) to [pos=.6, above] node {$\gamma_2$} (8);
\draw[->] (6) to [pos=.4, above] node {$\beta_2$} (2);
\draw[->] (7) to [pos=.6, above] node {$\beta_1$} (6);
\draw[->] (9) to [pos=.4, above] node {$\delta_1$} (5);
\draw[->] (10) to node [pos=.6, below] {$\epsilon$} (9);

\draw[style=help lines] (2) +(30:.7cm) arc (30:150:.7cm);
\draw[style=help lines] (2) +(210:.7cm) arc (210:330:.7cm);
\draw[style=help lines] (4) +(30:.7cm) arc (30:-30:.7cm);
\draw[style=help lines] (6) +(150:.7cm) arc (150:210:.7cm);
\draw[style=help lines] (5) +(30:.7cm) arc (30:150:.7cm);
\draw[style=help lines] (5) +(210:.7cm) arc (210:330:.7cm);
\draw[style=help lines] (9) +(150:.7cm) arc (150:210:.7cm);
\end{tikzpicture}
\caption{The algebra $A_{10}$ of a tiling of $P_{10,1}$ \cite{SimoesParsons}} \label{A10}
\end{figure}
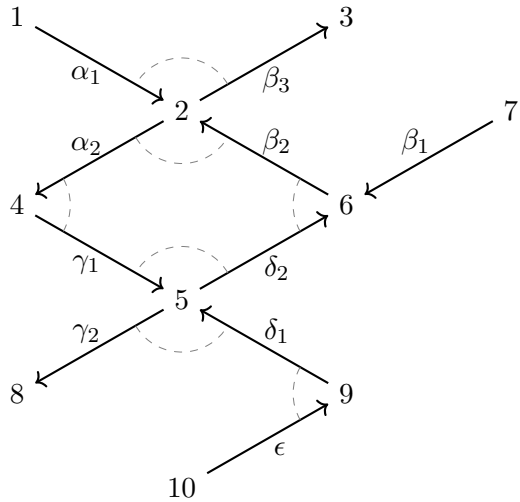

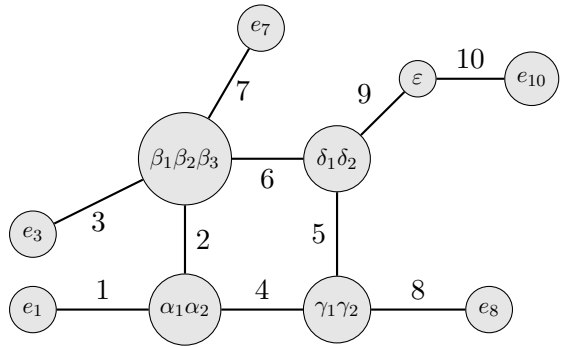
\begin{figure}[H]
\centering
\begin{tikzpicture}[auto, thick]
	\node[cblack] (1) at (0:0) {$e_1$};
	\node[cblack] (2) at ($(1)+(0:2cm)$) {$\alpha_1\alpha_2$};
	\node[cblack] (3) at ($(2)+(0:2cm)$) {$\gamma_1\gamma_2$};
	\node[cblack] (4) at ($(3)+(0:2cm)$) {$e_8$};
	\node[cblack] (5) at ($(1)+(90:1cm)$) {$e_3$};
	\node[cblack] (6) at ($(2)+(90:2cm)$) {$\beta_1\beta_2\beta_3$};
	\node[cblack] (7) at ($(3)+(90:2)$) {$\delta_1\delta_2$};
	\node[cblack] (8) at ($(7)+(45:1.5)$) {$\varepsilon$};
	\node[cblack] (9) at ($(8)+(0:1.5)$) {$e_{10}$};
	\node[cblack] (10) at ($(6)+(60:2)$) {$e_7$};
	
	\draw (1) to node {$1$} (2);
	\draw (2) to node {$4$} (3);
	\draw (3) to node {$8$} (4);
	\draw (5) to [below] node {$3$} (6);
	\draw (2) to [right] node {$2$} (6);
	\draw (6) to [below] node {$6$} (7);
	\draw (3) to node {$5$} (7);
	\draw (6) to [pos=.4, right] node {$7$} (10);
	\draw (7) to node {$9$} (8);
	\draw (8) to node {$10$} (9);
\end{tikzpicture}
\caption{The Graph $\Gamma_{A_{10}}$ of the algebra $A_{10}$.}
\end{figure}

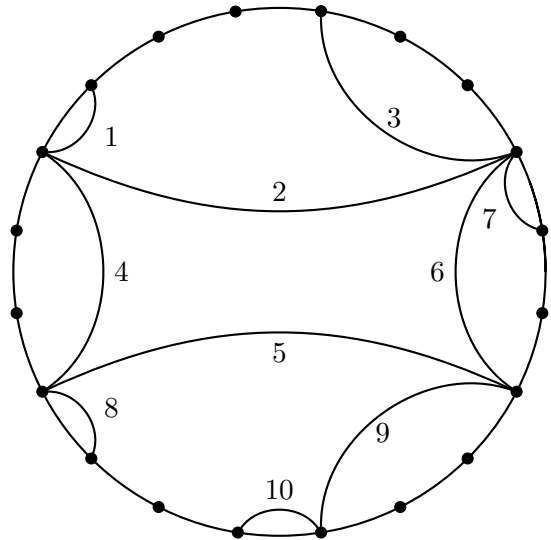
\begin{figure}[H]
	\centering
\begin{tikzpicture}[auto, thick,bend angle=55]
	\node (0) at (0,0) {};
	\node[cb] (1) at ($(0)+(9:3.5)$){};
	\node[cb] (2) at ($(0)+(27:3.5)$){};
	\node[cb] (3) at ($(0)+(45:3.5)$){};
	\node[cb] (4) at ($(0)+(63:3.5)$){};
	\node[cb] (5) at ($(0)+(81:3.5)$){};
	\node[cb] (6) at ($(0)+(99.5:3.5)$){};
	\node[cb] (7) at ($(0)+(117:3.5)$){};
	\node[cb] (8) at ($(0)+(135:3.5)$){};
	\node[cb] (9) at ($(0)+(153:3.5)$){};
	\node[cb] (10) at ($(0)+(171:3.5)$){};
	\node[cb] (11) at ($(0)+(189:3.5)$){};
	\node[cb] (12) at ($(0)+(207:3.5)$){};
	\node[cb] (13) at ($(0)+(225:3.5)$){};
	\node[cb] (14) at ($(0)+(243:3.5)$){};
	\node[cb] (15) at ($(0)+(261:3.5)$){};
	\node[cb] (16) at ($(0)+(279:3.5)$){};
	\node[cb] (17) at ($(0)+(297:3.5)$){};
	\node[cb] (18) at ($(0)+(315:3.5)$){};
	\node[cb] (19) at ($(0)+(333:3.5)$){};
	\node[cb] (20) at ($(0)+(351:3.5)$){};

	\draw (0)+(0:3.5) arc (0:380:3.5);
	
	\draw (1) to [bend left] node {$7$} (2);
	\draw (2) to [bend left, above, pos=0.45] node {$3$} (5);
	\draw (2) to [out=205,in=-25, above] node {$2$} (9);
	\draw (2) to [bend right, left] node {$6$} (19);
	\draw (8) to [bend left, pos=0.3] node {$1$} (9);
	\draw (9) to [bend left] node {$4$} (12);
	\draw (12) to [out=25,in=155, below] node {$5$} (19);
	\draw (12) to [bend left, pos=0.7] node {$8$} (13);
	\draw (15) to [bend left] node {$10$} (16);
	\draw (16) to [bend left, below] node {$9$} (19);
		
\end{tikzpicture}

\caption{The  tiling of $P_{10,1}$ (see \cite{SimoesParsons}) giving rise to the gentle algebra $A_{10}$. The graph $\Gamma_{A_{10}}$ of $A_{10}$ gives the internal arcs of the tiling of  $P_{10,1}$.} \label{Tiling P10}
\end{figure}

}
\end{Example}

\begin{Example}
{\rm  
Kalck \cite{Kalck} has shown that the algebras $A_1 = KQ_1/ (\alpha_1 \alpha_2, \beta_1 \beta_2) $ and $A_2 = KQ_2 /(\alpha_1 \alpha_2, \gamma \alpha_1, \alpha_2 \beta)$ are not derived equivalent. However, both algebras  have the same Avella-Geiss invariants \cite{AG} (these are invariants of gentle algebras such that if two gentle algebras are derived equivalent then they have the same Avella-Geiss invariants) and they also have the same graph as gentle algebras. This graph is the ribbon graph given in Figure~\ref{Kalck Example2} (compare also with the Brauer graph in Figure~\ref{Kalck Example}. Note that the associated Brauer graph algebra is the trivial extension of both $A_1$ and $A_2$).

\begin{figure}[H]
	\centering
	\tikzstyle{help lines}+=[dashed]
\begin{tikzpicture}[auto, thick]
	\node (1) at (-2,0) {1};
	\node (2) at ($(a)+(0:2)$) {2};
	\node (3) at ($(b)+(0:2)$) {3};
	
	\draw[->] (1) to [bend left] node {$\alpha_1$}(2);
	\draw[->] (1) to [bend right] node[below] {$\beta_1$} (2);
	\draw[->] (2) to [bend left] node {$\alpha_2$}(3);
	\draw[->] (2) to [bend right] node[below] {$\beta_2$} (3);

	\draw[style=help lines] (2) +(35:.7cm) arc (35:145:.7cm);
	\draw[style=help lines] (2) +(-35:.7cm) arc (-35:-145:.7cm);
\end{tikzpicture}
\caption{Quiver $Q_1$}
\end{figure}
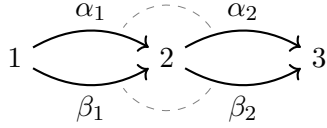

\begin{figure}[H]
	\centering
	\tikzstyle{help lines}+=[dashed]
\begin{tikzpicture}[auto, thick]
	\node (1) at (0,0) {1};
	\node (2) at ($(1)+(120:3)$) {2};
	\node (3) at ($(1)+(180:3)$) {3};
	
	\draw[style=help lines] (1) +(125:.7cm) arc (125:168:.7cm);
	\draw[style=help lines] (2) +(-62:.7cm) arc (-62:-117:.7cm);
	\draw[style=help lines] (3) +(70:.5cm) arc (70:345:.5cm);
	
	\draw[->] (1) to [right] node {$\alpha_1$}(2);
	\draw[->] (2) to [left] node {$\alpha_2$}(3);
	\draw[transform canvas={shift={(-90:0.085)}},<-] (1) to node {$\beta$} (3);
	\draw[transform canvas={shift={(90:0.085)}},->] (3) to node {$\gamma$} (1);
	\end{tikzpicture}
\caption{Quiver $Q_2$}
\end{figure}
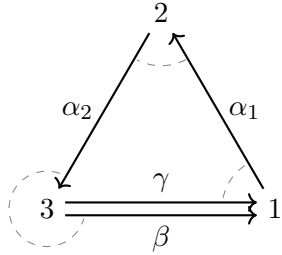

\begin{figure}[H]
	\centering
\begin{tikzpicture}[auto, thick, bend angle=50]
	\node[cblack] (a) at (0,0) {a};
	\node[cblack] (b) at ($(a)+(0:2.5)$) {b};
	
	\draw (a) to [bend left] node {1} (b);
	\draw (a) to [out=320,in=158,looseness=2.5] node[near start, below] {3} (b) ;
	\draw[draw=white,double=black] (a) to node[near start]{2} (b);
\end{tikzpicture}
\caption{The graph $\Gamma_{A_1} = \Gamma_{A_2}$.} \label{Kalck Example2}
\end{figure}
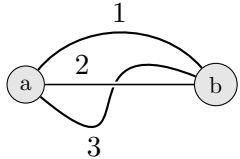

}
\end{Example}

\subsection{Gentle algebras from admissible cuts of Brauer graph algebras} 

Admissible cuts first appear in the PhD thesis of Fern\'andez \cite{F} where they were defined in the context of showing that two schurian algebras $\Lambda$ and $\Lambda'$ have isomorphic trivial extensions if and only if $\Lambda'$ is an admissible cut of $T(\Lambda)$. Recall that $\Lambda$ is a {\it schurian algebra} if for each pair of primitive idempotents $e$ and $f$ in $\Lambda$, we have $\dim_K (e\Lambda f) \leq 1$. 

In \cite{S},  {\it admissible cuts of Brauer graph algebras} are defined as  a slight modification of Fern\'andez notion. 

\begin{definition}
{\rm Let $A = KQ/I$ be a Brauer graph algebra with Brauer graph $G$ and multiplicity function $m$ identically equal to one. An {\it admissible cut} $\Delta$ of $Q$ is a set of arrows containing exactly one arrow from every special cycle $C_v$ (up to rotation), where $v$ is not of valency one as a vertex of $G$. 

We define the {\it cut algebra with cut $\Delta$} to be the algebra $A_\Delta = KQ / \langle I \cup \Delta \rangle $ where $\langle I \cup \Delta \rangle$ is the ideal of $KQ$ generated by $I \cup \Delta$. 
}
\end{definition}

\begin{theorem}\cite{S}
Let $A = KQ/I$ be a Brauer graph algebra with Brauer graph $G$ with multiplicity function identically equal to one. Let $\Delta$ be an admissible cut of $Q$. Then $A_\Delta$ is a gentle algebra and $\Gamma_{A_\Delta} = G$. 
\end{theorem}

\begin{corollary}\cite{S}
Every gentle algebra is the cut algebra of a unique Brauer graph algebra, its trivial extension, and conversely, every Brauer graph algebra with multiplicity function identically equal to one is the trivial extension of a (not necessarily unique) gentle algebra. 
\end{corollary}

\begin{Example}\label{Example Cut Algebras}
{\rm 
Let $\Lambda = KQ/I$ be the Brauer graph algebra with Brauer graph given by 

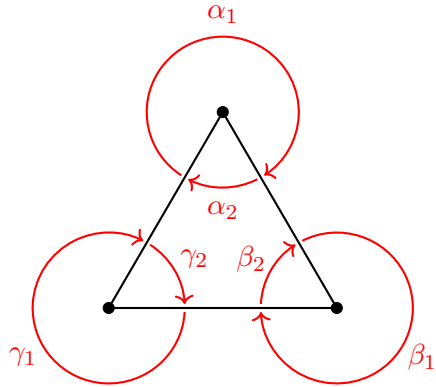
\begin{figure}[H]
	\centering
	\tikzstyle{help lines}+=[dashed]
\begin{tikzpicture}[auto, thick]
	\node[cb] (1) at (0,0) {};
	\node[cb] (2) at ($(1)+(0:3)$) {};
	\node[cb] (3) at ($(1)+(60:3)$) {};
	\node[red] () at ($(1)+(210:1.3)$){$\gamma_1$};
	\node[red] () at ($(1)+(30:1.3)$){$\gamma_2$};
	\node[red] () at ($(3)+(90:1.3)$){$\alpha_1$};
	\node[red] () at ($(3)+(-90:1.3)$){$\alpha_2$};
	\node[red] () at ($(2)+(-30:1.3)$){$\beta_1$};
	\node[red] () at ($(2)+(150:1.3)$){$\beta_2$};
	\draw (1)--(2)--(3)--(1);
	\draw[->, red] (1) +(-3:1cm) arc (-3:-297:1cm);
	\draw[->, red] (1) +(57:1cm) arc (57:3:1cm);
	\draw[->, red] (2) +(177:1cm) arc (177:123:1cm);
	\draw[->, red] (2) +(117:1cm) arc (117:-177:1cm);
	\draw[->, red] (3) +(-63:1cm) arc (-63:-117:1cm);
	\draw[->, red] (3) +(237:1cm) arc (237:-57:1cm);
\end{tikzpicture}
\caption{Example 4.14}
\end{figure}

Let $\Delta_1 = \{\alpha_1, \beta_1, \gamma_1 \}$ and $\Delta_2 = 
\{ \alpha_1, \beta_1, \gamma_2 \}$ be admissible cuts of $Q$. Then the cut algebras are isomorphic to the following algebras
$A_{\Delta_1} \simeq K (Q \setminus \Delta_1) / ( \beta_2 \alpha_2,  \gamma_2 \beta_2,  \alpha_2 \gamma_2)$ and $A_{\Delta_2} \simeq K(Q \setminus \Delta_2) / (\beta_2 \alpha_2)$ and it is easy to see that they are gentle. }
\end{Example}

\begin{remark}
{\rm Let $A = KQ/I$ be a Brauer graph algebra and $\Delta$ and $\Delta'$ be two distinct admissible cuts of $Q$. Then $A_\Delta$ and $A_{\Delta'}$ have the same number of simple modules but they are not necessarily isomorphic nor derived equivalent. It is possible that  $A_\Delta$ is of finite global dimension and $A_{\Delta'}$ is of infinite global dimension as is the case in Example~\ref{Example Cut Algebras} where $A_{\Delta_1}$ is of infinite global dimension whereas $A_{\Delta_2}$ is of finite global dimension.   
}
\end{remark}

\subsection{Gentle surface algebras and quiver mutation}\label{gentle algebra mutation}

In \cite{Ladkani1, Ladkani2} Ladkani studies when two mutation-equivalent Jacobian algebras arising from quivers with potential associated to marked surfaces with all marked points on the boundary are derived equivalent. Note that in this case the Jacobian algebras are gentle algebras \cite{ABCP}.  We will start by recalling Ladkani's result and relate it to the corresponding Brauer graph algebras (given by the trivial extensions of the Jacobian algebras).

For this, recall the Fomin-Zelevinsky quiver mutation. 

\begin{definition}\label{Def Quiver Mutation}
Let $Q$ be a quiver without loops or two cycles and let $k$ be a vertex of $Q$. The \textit{mutation  of $Q$ at $k$} is the quiver  $\mu_k(Q)$ obtained from $Q$ as follows: 

1) for each subquiver $i \to k \to j$,  add a new arrow $ i \to j$; 

2) all arrows with source or target $k$ are reversed

3) remove all newly-created 2-cycles. 
\end{definition}

Quiver mutation is an involution, that is $\mu_k(\mu_k(Q)) = Q$. 

Let $S$ be an oriented surface and let $M$ be a set of marked points in $S$. An {\it arc} is a curve  without self intersections with endpoints in $M$, considered up to homotopy,  and such that it does not cut out a monogon or bigon and such that it is disjoint from the boundary except for the endpoints.  An {\it ideal triangulation} is a maximal collection of non-crossing arcs. Let $T$ be an ideal triangulation of the \textit{marked surface $(S,M)$}. We associate to $T$ a quiver $Q(T)$ such that the vertices of $Q(T)$ correspond to the internal (=non-boundary) edges of $T$ and  where the arrows are given by the immediate successor relations of two edges incident with a common vertex induced by the orientation of the surface.  For more details on triangulations and the associated quivers we refer the reader to \cite{FST}. In what follows we will often refer to an ideal triangulation as a triangulation. 

Let $s$ be an internal edge in $T$. Then the \textit{flip of  $T$ at $s$} is defined by removing $s$ and replacing it by $s'$ such that we obtain the triangulation $T' =  (T\setminus s) \cup s'$, where $s' \neq s$ is the unique edge completing $T \setminus s$ to a triangulation of $(S,M)$.

\begin{proposition}\cite{FST}
Let $T$ be a triangulation of a marked surface $(S,M)$ and let $T'$ be the triangulation obtained by flipping an internal $s$ edge in $T$. Then $Q(T') = \mu_s(Q(T))$. 
\end{proposition} 

From now on let $(S,M)$ be such that $S$ is a surface with boundary $\partial S$ 
and $M$ is a set of marked points such that $M \subset \partial S$. Let $T$ be a triangulation of $(S,M)$ and $A_T$  
the gentle algebra associated to $(S, M, T)$ \cite{ABCP, Labardini} with quiver $Q(T)$. Let $g$ be the genus of $S$ and $b$ the number of boundary components in $S$. 

\begin{definition} {\rm  1) An (ideal) triangle in $T$ is a {\it boundary triangle} if exactly two of its sides are arcs that lie in the boundary of $S$.

2)  Let $B_1, \ldots, B_b$ be the boundary components in $S$. Let $n_i$ be the number of marked points on $B_i$ and let $d_i$ be the number of boundary triangles incident with $B_i$, that is triangles where exactly two sides lie in the boundary component $B_i$.  Note that the $d_i$ depend on $T$. We call the sequence
$$ g, b, (n_1, d_1), \ldots, (n_b, d_b)$$ the {\it parameters} of $T$. }
\end{definition}

For a finite dimensional algebra $A$ denote by $ \cD^\flat(A)$ the bounded derived category of finitely generated $A$-modules. For more details on the definition of bounded derived categories and their properties we refer the reader to the Appendix in Section~\ref{Appendix} and to standard textbooks such as \cite{Happel, Weibel}. 

\begin{theorem} \cite{Ladkani2} \label{Thm Derived eq surface gentle}
Let $(S,M)$ be a marked surface with all marked points in the boundary of $(S,M)$. 
Let $T$ and $T'$ be triangulations of $(S,M)$ having the same parameters. Let $A_T$ and $A_{T'}$ be the associated gentle algebras.  Then 
$$ \cD^\flat(A_T) \simeq   \cD^\flat(A_{T'}).$$
\end{theorem}

\begin{remark} {\rm 1) If $(S,M)$ is a marked disc and $T$ and $T'$ are two triangulations of $(S,M)$ then $ \cD^\flat(A_T) \simeq   \cD^\flat(A_{T'})$ if and only if $T$ and $T'$ have the same number of boundary triangles. Note that by \cite{ABCP} the algebras $A_T$ and $A_{T'}$ are cluster-tilted algebras of Dynkin type $A$ and their derived equivalence classification was already done in \cite{BV}. For the notion of a cluster-tilted algebras, we refer, for example, to \cite{Sch}.  

2) If $T'$ is obtained from $T$ by flipping one arc then $ \cD^\flat(A_T) \simeq   \cD^\flat(A_{T'})$ if and only if $T$ and $T'$ have the same number of boundary triangles. 

3) In \cite{Ladkani2} a boundary triangle is called a {\it dome}. }\end{remark}

\begin{corollary}\label{Der triv ext}
Let $A_T$ and $A_{T'}$ be gentle algebras associated to triangulations $T$ and $T'$ of $(S,M)$  such that $T$ and $T'$ have the same parameters and let $\Lambda_T $ and $\Lambda_{T'} $ be the associated Brauer graph algebras, that is $\Lambda_T = A_T \ltimes D(A_T)$ and $\Lambda_{T'} = A_{T'} \ltimes D(A_{T'})$. Then 
 $$\cD^\flat(\Lambda_T) \simeq  \cD^\flat(\Lambda_{T'}).$$
\end{corollary}

\begin{proof}
This immediately follows from Theorem~\ref{trivialextension} showing that  $\Lambda_T = T(A_T)$ and $\Lambda_{T'} = T(A_{T'})$ and the  result by Rickard \cite{Rickard91} that if two finite dimensional $K$-algebras $A$ and $B$ are derived equivalent then their trivial extensions $T(A)$ and $T(B)$ are derived equivalent. 
\end{proof}

In particular,  in the setting of Corollary~\ref{Der triv ext}, the quiver $Q_{T'}$ of the Brauer graph algebra $\Lambda_{T'}$ is obtained from the quiver $Q_T$ of the Brauer graph algebra $\Lambda_T$ by successive Fomin-Zelevinsky quiver mutations.

\begin{Example}
{\rm Let $T, T'$ and $T''$ be the following triangulations of the $9$-gon giving rise to the associated gentle algebras $A_T, A_{T'}$ and $A_{T''}$. 

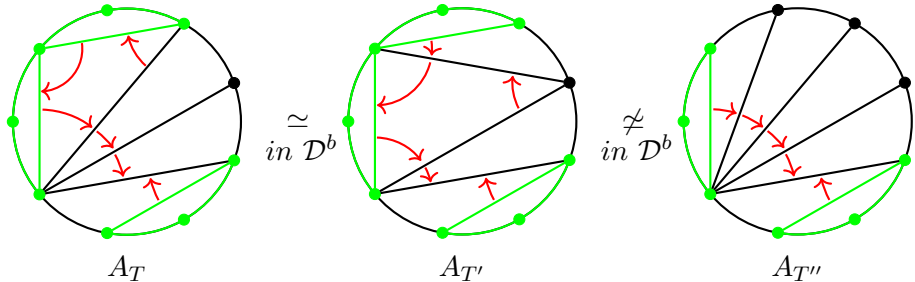
\begin{figure}[H]
	\centering
\subfigure{
\begin{tikzpicture}[auto, thick, scale=.75]
	\draw (0) circle (2cm);
	\node (0) at (0,0) {};
	\node[g] (1) at ($(0)+(180:2)$){};
	\node[g] (2) at ($(0)+(140:2)$){};
	\node[g] (3) at ($(0)+(100:2)$){};
	\node[g] (4) at ($(0)+(60:2)$){};
	\node[cb] (5) at ($(0)+(20:2)$){};
	\node[g] (6) at ($(0)+(340:2)$){};
	\node[g] (7) at ($(0)+(300:2)$){};
	\node[g] (8) at ($(0)+(260:2)$){};
	\node[g] (9) at ($(0)+(220:2)$){};

	\draw[green]  (0)+(60:2cm) arc (60:220:2){};
	\draw[green]  (0)+(340:2cm) arc (340:260:2){};
	\draw[red,->] (2)+(8:0.75cm) arc (8:-88:.75cm){};
	\draw[red,->] (4)+(228:1cm) arc (228:192:1cm){};
	\draw[red,->] (9)+(88:1.5cm) arc (88:52:1.5cm){};
	\draw[red,->] (9)+(48:1.5cm) arc (52:36:1.5cm){};
	\draw[red,->] (9)+(28:1.5cm) arc (28:12:1.5cm){};
	\draw[red,<-] (6)+(192:1.5cm) arc (192:208:1.5cm){};
	
	\draw[green] (9) to node {}(2);
	\draw (9) to node {}(4);
	\draw (9) to node {}(5);
	\draw (9) to node {}(6);
	\draw[green] (8) to node {}(6);
	\draw[green] (2) to node {}(4);
	\node () at ($(0)+(0:2.5)+(0:.5)$){$\simeq$};
	\node () at ($(0)+(-10:2.6)+(0:.5)$){$in$ ${\mathcal D}^b$};
	\node () at ($(0)+(-90:2.6)$){$A_T$};
\end{tikzpicture}}
\hspace{-.5cm}
\subfigure{
\begin{tikzpicture}[auto, thick,scale=.75]
	\draw (0) circle (2cm);
	\node (0) at (0,0) {};
	\node[g] (1) at ($(0)+(180:2)$){};
	\node[g] (2) at ($(0)+(140:2)$){};
	\node[g] (3) at ($(0)+(100:2)$){};
	\node[g] (4) at ($(0)+(60:2)$){};
	\node[cb] (5) at ($(0)+(20:2)$){};
	\node[g] (6) at ($(0)+(340:2)$){};
	\node[g] (7) at ($(0)+(300:2)$){};
	\node[g] (8) at ($(0)+(260:2)$){};
	\node[g] (9) at ($(0)+(220:2)$){};

	\draw[green]  (0)+(60:2cm) arc (60:220:2){};
	\draw[green]  (0)+(340:2cm) arc (340:260:2){};
	\draw[red,->] (2)+(8:1cm) arc (8:-7:1	cm){};
	\draw[red,->] (2)+(-14:1cm) arc (-14:-88:1cm){};
	\draw[red,->] (5)+(207:1cm) arc (207:173:1cm){};
	\draw[red,->] (9)+(88:1cm) arc (88:36:1cm){};
	\draw[red,->] (9)+(28:1cm) arc (28:12:1cm){};
	\draw[red,<-] (6)+(192:1.5cm) arc (192:208:1.5cm){};
	
	\draw[green] (9) to node {}(2);
	\draw (2) to node {}(5);
	\draw (9) to node {}(5);
	\draw (9) to node {}(6);
	\draw[green] (8) to node {}(6);
	\draw[green] (2) to node {}(4);
	\node () at ($(0)+(0:2.5)+(0:.5)$){$\not\simeq$};
	\node () at ($(0)+(-10:2.6)+(0:.5)$){$in$ ${\mathcal D}^b$};
	\node () at ($(0)+(-90:2.6)$){$A_{T'}$};
\end{tikzpicture}}
\hspace{-.5cm}
\subfigure{
\begin{tikzpicture}[auto, thick,scale=.75]
	\draw (0) circle (2cm);
	\node (0) at (0,0) {};
	\node[g] (1) at ($(0)+(180:2)$){};
	\node[g] (2) at ($(0)+(140:2)$){};
	\node[cb] (3) at ($(0)+(100:2)$){};
	\node[cb] (4) at ($(0)+(60:2)$){};
	\node[cb] (5) at ($(0)+(20:2)$){};
	\node[g] (6) at ($(0)+(340:2)$){};
	\node[g] (7) at ($(0)+(300:2)$){};
	\node[g] (8) at ($(0)+(260:2)$){};
	\node[g] (9) at ($(0)+(220:2)$){};

	\draw[green]  (0)+(140:2cm) arc (140:220:2){};
	\draw[green]  (0)+(340:2cm) arc (340:260:2){};
	\draw[red,->] (9)+(88:1.5cm) arc (88:72:1.5cm){};
	\draw[red,->] (9)+(68:1.5cm) arc (68:52:1.5cm){};
	\draw[red,->] (9)+(48:1.5cm) arc (52:36:1.5cm){};
	\draw[red,->] (9)+(28:1.5cm) arc (28:12:1.5cm){};
	\draw[red,<-] (6)+(192:1.5cm) arc (192:208:1.5cm){};
	
	\draw[green] (9) to node {}(2);
	\draw (9) to node {}(3);
	\draw (9) to node {}(4);
	\draw (9) to node {}(5);
	\draw (9) to node {}(6);
	\draw[green] (8) to node {}(6);
	\node () at ($(0)+(-90:2.6)$){$A_{T''}$};
\end{tikzpicture}}
\end{figure}

Then $T$ and $T'$ have the same number of boundary triangles and  $A_T$ and $ A_{T'}$ are derived equivalent. Furthermore, the corresponding Brauer graph algebras  $\Lambda_T = T(A_T)$ and $\Lambda_{T'} = T(A_{T'})$ are derived equivalent. 
But $A_T$ and $A_{T''}$ are not derived equivalent. Note that here we also have that the corresponding Brauer graph algebras  $\Lambda_T$ and $\Lambda_{T^{''}} = T(A_{T^{''}})$ are not derived equivalent (since $T$ is a graph and $T^{''}$ is a tree so that  $\Lambda_{T}$ is of tame representation type whereas  $\Lambda_{T^{''}}$ is of finite representation type). However, this does not always hold, see for example in Example~\ref{Example Cut Algebras} where the gentle algebras $A_{\Delta_1}$ and $A_{\Delta_2}$ are not derived equivalent but have isomorphic trivial extensions, i.e. they give rise to the same Brauer graph algebra. 
}
\end{Example}

 \section{Derived equivalences and mutation of Brauer graph algebras}\label{Section Mutation Brauer graph MS}

Gentle algebras have a remarkable property. Namely their class is closed under derived equivalence, that is any algebra derived equivalent to a gentle algebra is again a gentle algebra. That is any algebra derived equivalent to a gentle algebra is again a gentle algebra \cite{SZ}. 

The same is expected to hold for Brauer graph algebras \cite{AZ}. 

\subsection{Mutation of Brauer graph algebras}

In \cite{Kauer} Kauer defined tilting complexes for Brauer graph algebras and showed that they give rise to derived equivalences in such a way that the endomorphism algebra of the tilting complex is again a Brauer graph algebra where the Brauer graph is obtained from the original one by a simple geometric move on the Brauer graph. We note that the proof at the time of writing was not complete, in that Kauer assumed, but omitted to show, that the tilted algebra has again the structure of a Brauer graph algebra. However, it follows from \cite{Rickard2} that this holds for Brauer tree algebras, see also \cite{B}. In the general case of a Brauer graph algebra that is not necessarily a Brauer tree algebra, this has been shown to hold, for example, in  \cite{Demonet} and also in \cite{AAC}. 

We will describe Kauer's tilting complex construction as well as  the corresponding geometric moves on the Brauer graphs below. We note that Kauer's construction corresponds to a more general construction of two-term tilting complexes  by Okuyama \cite{Okuyama} (based on a manuscript by Rickard \cite{Rickard3}). We will state a reformulation of this result due to \cite{Linckelmann}. For a module $M$, denote by $P(M)$ the projective cover of $M$. The definition of a tilting complex is recalled in the Appendix in Section~\ref{Appendix}.

\begin{theorem}\cite{Okuyama, Rickard3}
Let $A$ be a symmetric $K$-algebra. Let $I, I'$ be disjoint sets of simple $A$-modules such that $I \cup I'$ is a complete set of representatives of the 
isomorphism classes of simple $A$-modules such that for any $S, U \in I$, we have 
$\Ext^1_A(S,U)= \{0\}$.

For any $S \in I$, let $T_S$ be the complex $P(\rad(P(S))) \stackrel{\pi_S}{\rightarrow} P(S)$ with non-zero terms in degrees zero and one, and such that $\Im \pi_S = \rad (P(S))$. For any $S' \in I'$ consider $P(S')$ as a complex concentrated in degree zero.  

Then the complex $T = (\bigoplus_{S \in I} T_S) \oplus (\bigoplus_{S' \in I'} 
P(S'))$ is a tilting complex for $A$. 
\end{theorem}

\subsection{Kauer moves on the Brauer graph}

An example of an Okuyama-Rickard tilting complex is given by the following complex constructed by Kauer in \cite{Kauer}.
Let $A$ be a Brauer graph algebra with Brauer graph $G$. For $I = \{S_0\}$ where $S_0$ is a simple $A$-module  such that the corresponding edge in the Brauer graph is not a loop where the corresponding half edges are such that one is a direct successor of the other, set 
$T = T_{S_0} \oplus \bigoplus_{ S \neq S_0, S \mbox{ simple }}  
P(S)$. Then in \cite{Kauer} this is shown to be a tilting complex for $A$ by verifying its properties one by one. However, that $T$ is a tilting complex also follows directly from the fact that $T$ is an Okuyama tilting complex for $A$. Kauer further shows that  the Brauer graph  algebra $B = \End_{\cK^\flat(P_A)} (T)$ has 
 Brauer graph $G' =(G \setminus {s}) \cup {s'}$ where $s$ is the edge in $G$ corresponding to $S_0$ and where $s'$ is obtained  by one of the following local moves on $G$:

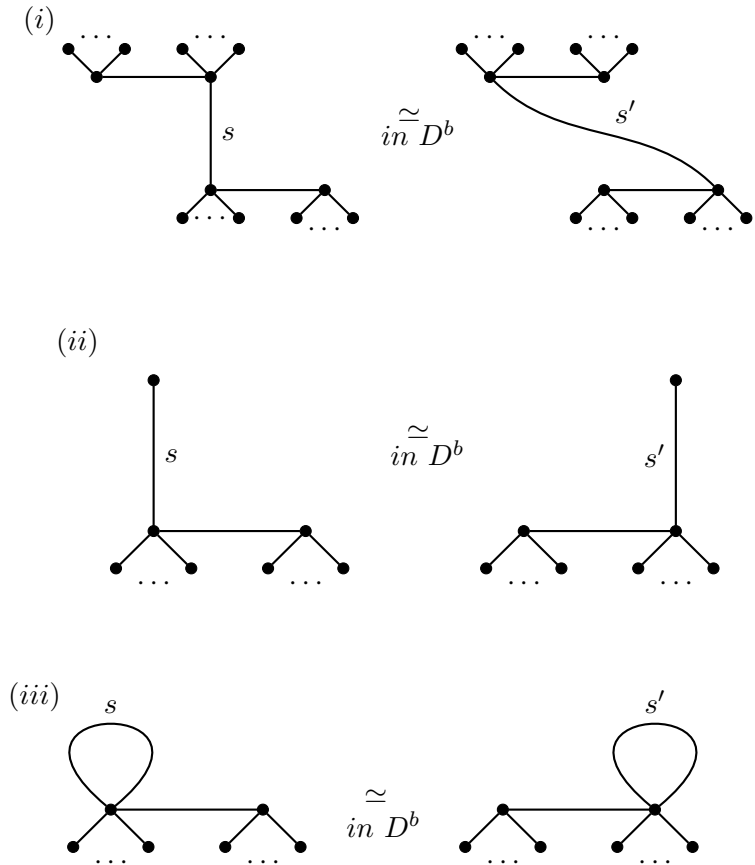
\begin{figure}[H] 
	\centering
\subfigure{	
	\begin{tikzpicture}[auto, thick,scale=.75]
	\node (0) at (-1,1){$(i)$};
	\node[cb] (1) at (0,0){};
	\node[cb] (2) at ($(1)+(45:.7)$){};
	\node[cb] (3) at ($(1)+(135:.7)$){};
	\node () at ($(1)+(87:.7)$){$\cdots$};
	\node[cb] (4) at ($(1)+(0:2)$){};
	\node[cb] (5) at ($(4)+(135:.7)$){};
	\node[cb] (6) at ($(4)+(45:.7)$){};
	\node () at ($(4)+(87:.7)$){$\cdots$};
	\node[cb] (7) at ($(4)+(-90:2)$){};
	\node[cb] (8) at ($(7)+(-45:.7)$){};
	\node[cb] (9) at ($(7)+(-135:.7)$){};
	\node () at ($(7)+(-87:.5)$){$\cdots$};
	\node[cb] (10) at ($(7)+(0:2)$){};
	\node[cb] (11) at ($(10)+(-45:.7)$){};
	\node[cb] (12) at ($(10)+(-135:.7)$){};
	\node () at ($(10)+(-87:.7)$){$\cdots$};
	\node () at ($(10)+(70:1.4)+(0:1)$){$\simeq$};
	\node () at ($(10)+(58:1.2)+(0:1)$){$in$ $D^b$};)
	
	\draw (3)--(1)--(2);
	\draw (1)--(4);
	\draw (5)--(4)--(6);
	\draw (4) to node {$s$}(7);
	\draw (8)--(7)--(9);
	\draw (7)--(10);
	\draw (11)--(10)--(12);
	\end{tikzpicture}}
\hspace{-.5cm}
\subfigure{	
	\begin{tikzpicture}[auto, thick, scale=.75]
	\node[cb] (1) at (0,0){};
	\node[cb] (2) at ($(1)+(45:.7)$){};
	\node[cb] (3) at ($(1)+(135:.7)$){};
	\node () at ($(1)+(87:.7)$){$\cdots$};
	\node[cb] (4) at ($(1)+(0:2)$){};
	\node[cb] (5) at ($(4)+(135:.7)$){};
	\node[cb] (6) at ($(4)+(45:.7)$){};
	\node () at ($(4)+(87:.7)$){$\cdots$};
	\node[cb] (7) at ($(4)+(-90:2)$){};
	\node[cb] (8) at ($(7)+(-45:.7)$){};
	\node[cb] (9) at ($(7)+(-135:.7)$){};
	\node () at ($(7)+(-87:.7)$){$\cdots$};
	\node[cb] (10) at ($(7)+(0:2)$){};
	\node[cb] (11) at ($(10)+(-45:.7)$){};
	\node[cb] (12) at ($(10)+(-135:.7)$){};
	\node () at ($(10)+(-87:.7)$){$\cdots$};
	
	\draw (3)--(1)--(2);
	\draw (1)--(4);
	\draw (5)--(4)--(6);
	\draw (1) to [out=-45, in=135]node {$s'$}(10);
	\draw (8)--(7)--(9);
	\draw (7)--(10);
	\draw (11)--(10)--(12);
	\end{tikzpicture}}	
\end{figure}

\begin{figure}[H] 
	\centering
\subfigure{	
	\begin{tikzpicture}[auto, thick]
	\node (0) at (-1,.5){$(ii)$};
	\node[cb] (1) at (0,0){};
	\node[cb] (2) at ($(1)+(-90:2)$){};
	\node[cb] (3) at ($(2)+(-45:.7)$){};
	\node[cb] (4) at ($(2)+(-135:.7)$){};
	\node () at ($(2)+(-87:.7)$){$\cdots$};
	\node[cb] (5) at ($(2)+(0:2)$){};
	\node[cb] (6) at ($(5)+(-45:.7)$){};
	\node[cb] (7) at ($(5)+(-135:.7)$){};
	\node () at ($(5)+(-87:.7)$){$\cdots$};
	
	\draw (1) to node {$s$}(2);
	\draw (3)--(2)--(4);
	\draw (2)--(5);
	\draw (7)--(5)--(6);
	
	\node () at ($(5)+(70:1.4)+(0:1)$){$\simeq$};
	\node () at ($(5)+(60:1.2)+(0:1)$){$in$ $D^b$};
	\end{tikzpicture}}
\hspace{-.3cm}
\subfigure{	
	\begin{tikzpicture}[auto, thick]
		\node[] (1) at (0,0){};
		\node[cb] (2) at ($(1)+(-90:2)$){};
		\node[cb] (3) at ($(2)+(-45:.7)$){};
		\node[cb] (4) at ($(2)+(-135:.7)$){};
		\node () at ($(2)+(-87:.7)$){$\cdots$};
		\node[cb] (5) at ($(2)+(0:2)$){};
		\node[cb] (6) at ($(5)+(-45:.7)$){};
		\node[cb] (7) at ($(5)+(-135:.7)$){};
		\node () at ($(5)+(-87:.7)$){$\cdots$};
		\node[cb] (8) at ($(5)+(90:2)$){};
		
		\draw (5) to node {$s'$}(8);
		\draw (3)--(2)--(4);
		\draw (2)--(5);
		\draw (7)--(5)--(6);	
	\end{tikzpicture}}	
\end{figure}

\begin{figure}[H] 
	\centering
\subfigure{	
	\begin{tikzpicture}[auto, thick]
	\node (1) at (-1,1.5){$(iii)$};
	\node[cb] (2) at (0,0){};
	\node[cb] (3) at ($(2)+(-45:.7)$){};
	\node[cb] (4) at ($(2)+(-135:.7)$){};
	\node () at ($(2)+(-87:.7)$){$\cdots$};
	\node[cb] (5) at ($(2)+(0:2)$){};
	\node[cb] (6) at ($(5)+(-45:.7)$){};
	\node[cb] (7) at ($(5)+(-135:.7)$){};
	\node () at ($(5)+(-87:.7)$){$\cdots$};
	
	\draw (2) to [out=140, in=40, looseness=50] node {$s$}(2);
	\draw (3)--(2)--(4);
	\draw (2)--(5);
	\draw (7)--(5)--(6);
	
	\node () at ($(5)+(10:1)+(0:.5)$){$\simeq$};
	\node () at ($(5)+(-10:1.1)+(0:.5)$){$in$ $D^b$};
	\end{tikzpicture}}
\hspace{0cm}
\subfigure{	
	\begin{tikzpicture}[auto, thick]
		
		\node[cb] (2) at (0,0){};
		\node[cb] (3) at ($(2)+(-45:.7)$){};
		\node[cb] (4) at ($(2)+(-135:.7)$){};
		\node () at ($(2)+(-87:.7)$){$\cdots$};
		\node[cb] (5) at ($(2)+(0:2)$){};
		\node[cb] (6) at ($(5)+(-45:.7)$){};
		\node[cb] (7) at ($(5)+(-135:.7)$){};
		\node () at ($(5)+(-87:.7)$){$\cdots$};
				
		\draw (5) to [out=140, in=40, looseness=50] node {$s'$}(5);
		\draw (3)--(2)--(4);
		\draw (2)--(5);
		\draw (7)--(5)--(6);	
	\end{tikzpicture}}	
\end{figure}

\begin{definition} {\rm We call a local move as in (i)-(iii) above a {\it Kauer move at $s$} and we adopt the following notation $\mu^+_s(G) = G' = (G \setminus {s}) \cup {s'}$.} \end{definition}

\begin{remark}
{\rm 1) In \cite{Aihara} Aihara refers to the above moves (i)-(iii) as flips of Brauer graphs and he describes them in terms of quiver combinatorics.

2) In \cite{AiharaIyama} Aihara and Iyama introduce the notion of  a (silting) mutation of an algebra in terms of left and right approximations and in \cite{Dugas} the Okuyama-Rickard tilting complexes giving rise to the Kauer moves are expressed in the more general case of weakly symmetric special biserial algebras in terms of silting mutations, that is in terms of left and right approximations. 
}
\end{remark}

\begin{Example}
{\rm  Let $A$ be the Brauer graph algebra with Brauer graph as given by the left hand graph in Figure~\ref{Example Mutation Square}. The Kauer move at $s = 0$ gives rise to the Brauer graph on  the right hand graph in Figure~\ref{Example Mutation Square}. 
 
  \begin{figure}[H] 
	\centering
\subfigure{	
	\begin{tikzpicture}[auto, thick]
	\node[cb] (1) at (0,0){};
	\node[cb] (2) at ($(1)+(0:2)$){};
	\node[cb] (3) at ($(2)+(-90:2)$){};
	\node[cb] (4) at ($(1)+(-90:2)$){};
	\node[cb] (5) at ($(1)+(90:2)$){};
	
	\draw (1) to node {$1$}(2);
	\draw (2) to node {$2$}(3);
	\draw (3) to node {$3$}(4);
	\draw (1) to [left] node {$4$}(4);
	\draw (1) to [left, pos=.6] node {$0$}(3);
	\draw (1) to node {$5$}(5);
	
	\node () at ($(3)+(55:1.5)+(0:1)$){$\simeq$};
	\node () at ($(3)+(40:1.2)+(0:1)$){$in$ $D^b$};
	\end{tikzpicture}}
\hspace{0cm}
\subfigure{	
	\begin{tikzpicture}[auto, thick]
	\node[cb] (1) at (0,0){};
	\node[cb] (2) at ($(1)+(0:2)$){};
	\node[cb] (3) at ($(2)+(-90:2)$){};
	\node[cb] (4) at ($(1)+(-90:2)$){};
	\node[cb] (5) at ($(1)+(90:2)$){};
	
	\draw (1) to node {$1$}(2);
	\draw (2) to node {$2$}(3);
	\draw (3) to node {$3$}(4);
	\draw (1) to [left] node {$4$}(4);
	\draw (2) to [right, pos=.6] node {$0'$}(4);
	\draw (1) to node {$5$}(5);
	\end{tikzpicture}}	
	
	\caption{Example of a Kauer move of type (i). The Brauer graph algebras associated to the two graphs are derived equivalent. }\label{Example Mutation Square}
\end{figure}
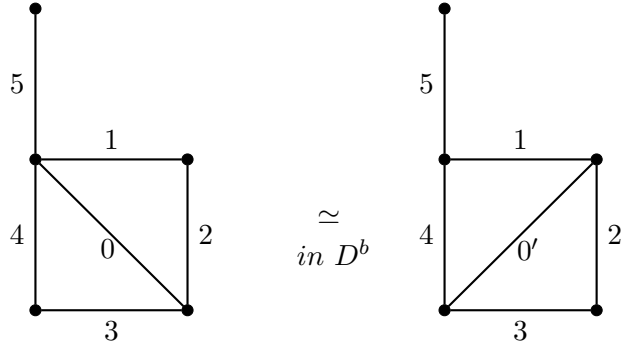

The Okuyama-Rickard complex $T$ giving rise to the equivalence of 
$\cD^\flat(A)$ and $\cD^\flat(B)$, where $B = \End_{\cK^\flat(A)}(T)$ is given by 
$T = T_0 \oplus P_1 \oplus P_2 \oplus P_3 \oplus P_4 \oplus P_5$ where, for $i = 1,2,3,4, 5$, $P_i$ is the stalk complex concentrated in degree zero of the indecomposable projective associated to the edge $i$  and  $T_0 = P_4 \oplus P_2 \stackrel{\pi}{\longrightarrow} P_0$, that is, $T_0$ is given by the following complex

\begin{figure}[H]
	\centering
\begin{tikzpicture}[auto, thick]
	\node (0) at (0,0) {$T_0=$};
	\node (1) at ($(0)+(0:1)$) {$\matrixTwo{4}{\dimThree510}{3}{4}$};
	\node (2) at ($(1)+(0:.7)$) {$\oplus$};
	\node (3) at ($(2)+(0:.7)$) {$\matrixTwo{2}{\dimTwo30}{1}{2}$};
	\node (4) at ($(3)+(0:2)$)  {$\matrixTwo{0}{\dimThree451}{\dimTwo23}{0}$};
	\draw[->] (3) to node {{\scriptsize $\pi_0$}}(4);
	
\end{tikzpicture}
\end{figure}

concentrated in degrees zero and one and where $\Im(\pi_0) = \rad(P_0) = $ 
\begin{tikzpicture}[auto, thick]
	\node (0) at (0,0) {$\matrixTwo{}{\dimThree451}{\dimTwo23}{0}$};
\end{tikzpicture}.

}
\end{Example}

\subsection{Brauer graph algebras associated to triangulations of marked oriented surfaces}

Let $S$ be an oriented surface with boundary and let $M$ be a set of marked points on $S$. Note here we don't necessarily require that $M$ lies in the boundary of $S$, however, we do require that on each boundary component of $S$ there is at least one marked point of $M$.  
We first recall that any two triangulations of $(S,M)$ are flip-connected, that is, one can be obtained from the other by a series of flips of diagonals, see for example \cite{Bu,FST}.

\begin{theorem}\label{flipconnected}
Let $T$ and $T'$ be triangulations of a marked oriented surface $(S,M)$. Then the triangulations $T$ and $T'$ are flip connected. 
\end{theorem}

Recall that every triangulation of $(S,M)$ is a ribbon graph where the cyclic ordering is induced by the orientation of $S$. 
In \cite{MS} we consider  triangulations $T$ of $(S,M)$  (including the boundary arcs) as  Brauer graphs.  Note that the Brauer graph algebras we consider in \cite{MS} are different from the  Brauer graph algebras associated to triangulations in Section~\ref{Section gentle}. Namely, in Section~\ref{Section gentle}, the Brauer graphs consist of the internal arcs of the triangulation whereas in \cite{MS} the boundary edges of the triangulation are part of the Brauer graphs, see Example~\ref{Frozen}. 
 It is observed in \cite{MS} that when regarding $T$ (including the boundary edges)
as a Brauer graph, the flip of $T$ at an internal edge $s$ coincides with applying the Kauer move to $T$ at $s$. Combining this with Theorem~\ref{flipconnected} we get

\begin{theorem}\cite{MS}\label{Thm Flip Brauer Graph}
Let $T$ and $T'$ be two triangulations of a marked oriented surface $(S,M)$. Then the associated Brauer graph algebras $\Lambda_T$ and $\Lambda_{T'}$ are derived equivalent, that is the bounded derived categories of finitely generated modules $\cD^\flat (\Lambda_T)$ and $\cD^\flat (\Lambda_{T'})$ are  equivalent as triangulated categories. 
\end{theorem}

Furthermore, we have

\begin{proposition}\cite{MS}\label{KauerFlipTriangulation}
Let $Q_T$ be the quiver of the Brauer graph algebra associated to a triangulation $T$  of $(S,M)$ and let $T' = (T \setminus s) \cup s'$. Then $Q_T'$ is obtained from $Q_T$ by Fomin-Zelevinsky quiver mutation, that is $Q_{T'} = \mu_s(Q_T)$.
\end{proposition}

Note that in Proposition~\ref{KauerFlipTriangulation} we do not have any restrictions on the triangulation $T$. In particular, it can contain self-folded triangles or punctures with exactly two incident arcs.  


\begin{Example}\label{Frozen}
{\rm 
The Brauer graph algebras associated to the following two Brauer graphs given by the triangulations $T$ and $T'$ of a hexagon (together with the boundary arcs) are derived equivalent.  Note that here, as opposed to the examples in Section~\ref{gentle algebra mutation},  $T$ and $T'$ do not have the same number of boundary triangles, yet the corresponding Brauer graph algebras are  derived equivalent. 
\begin{figure}[H]
\scalebox{.1}
\centering
	\subfigure{
	\begin{tikzpicture}[auto, thick, rotate=30]
	\node (0) at (0,0) {};
	\node[cb] (1) at ($(0)+(30:2)$) {};
	\node[cb] (2) at ($(0)+(90:2)$) {};
	\node[cb] (3) at ($(0)+(150:2)$) {};
	\node[cb] (4) at ($(0)+(210:2)$) {};
	\node[cb] (5) at ($(0)+(270:2)$) {};
	\node () at ($(0)+(240:3)$) {$T$};
	\node[cb] (6) at ($(0)+(330:2)$) {};
			
	\draw (1)--(2)--(3)--(4)--(5)--(6)--(1);
	
	\draw[->, red] (1)+(147:.7) arc (147:-87:.7cm);
	\draw[->, red] (1)+(-93:.7) arc (-93:-207:.7cm);
	\draw[->, red] (2)+(207:.7) arc (207:-27:.7cm);
	\draw[->, red] (2)+(-33:.7) arc (-33:-57:.7cm);
	\draw[->, red] (2)+(-63:.7) arc (-63:-117:.7cm);
	\draw[->, red] (2)+(-123:.7) arc (-123:-147:.7cm);
	\draw[->, red] (3)+(267:.7) arc (267:33:.7cm);
	\draw[->, red] (3)+(27:.7) arc (27:-87:.7cm);
	\draw[->, red] (4)+(327:.7) arc (327:93:.7cm);
	\draw[->, red] (4)+(87:.7) arc (87:63:.7cm);
	\draw[->, red] (4)+(57:.7) arc (57:3:.7cm);
	\draw[->, red] (4)+(-3:.7) arc (-3:-27:.7cm);
	\draw[->, red] (5)+(27:.7) arc (27:-207:.7cm);
	\draw[->, red] (5)+(147:.7) arc (150:33:.7cm);
	\draw[->, red] (6)+(87:.7) arc (87:-147:.7cm);
	\draw[->, red] (6)+(117:.7) arc (117:93:.7cm);
	\draw[->, red] (6)+(177:.7) arc (177:123:.7cm);
	\draw[->, red] (6)+(207:.7) arc (207:183:.7cm);
	
	\draw (4) to node {}(2);
	\draw (4) to node {}(6);
	\draw (2) to node {}(6);

\end{tikzpicture}}
\hspace{-.3cm}
\subfigure{	
	\begin{tikzpicture}
	{\footnotesize 
	\node (mu) at (0,2) {$\mu^{-}$};
	\node () at ($(mu)+(-90:.4)$) {$\simeq$};
	\node () at ($(mu)+(-90:.8)$) {$in$ $D^b$};
	}
	\node () at ($(mu)+(-90:3.5)$) {};
	\end{tikzpicture}}
\hspace{-.3cm}	
\subfigure{
\begin{tikzpicture}[auto, thick, rotate=30]
	\node (0) at (0,0) {};
	\node[cb] (1) at ($(0)+(30:2)$) {};
	\node[cb] (2) at ($(0)+(90:2)$) {};
	\node[cb] (3) at ($(0)+(150:2)$) {};
	\node[cb] (4) at ($(0)+(210:2)$) {};
	\node[cb] (5) at ($(0)+(270:2)$) {};
	\node () at ($(0)+(240:3)$) {$T'$};
	\node[cb] (6) at ($(0)+(330:2)$) {};
			
	\draw (1)--(2)--(3)--(4)--(5)--(6)--(1);
	
	\draw[->, red] (1)+(147:.7) arc (147:-87:.7cm);
	\draw[->, red] (1)+(-93:.7) arc (-93:-147:.7cm);
	\draw[->, red] (1)+(-153:.7) arc (-153:-207:.7cm);
	\draw[->, red] (2)+(207:.7) arc (207:-27:.7cm);
	\draw[->, red] (2)+(-33:.7) arc (-33:-117:.7cm);
	\draw[->, red] (2)+(-123:.7) arc (-123:-147:.7cm);
	\draw[->, red] (3)+(267:.7) arc (267:33:.7cm);
	\draw[->, red] (3)+(27:.7) arc (27:-87:.7cm);
	\draw[->, red] (4)+(327:.7) arc (327:93:.7cm);
	\draw[->, red] (4)+(87:.7) arc (87:63:.7cm);
	\draw[->, red] (4)+(57:.7) arc (57:33:.7cm);
	\draw[->, red] (4)+(27:.7) arc (27:3:.7cm);
	\draw[->, red] (4)+(-3:.7) arc (-3:-27:.7cm);
	\draw[->, red] (5)+(27:.7) arc (27:-207:.7cm);
	\draw[->, red] (5)+(147:.7) arc (150:33:.7cm);
	\draw[->, red] (6)+(87:.7) arc (87:-147:.7cm);
	\draw[->, red] (6)+(177:.7) arc (177:93:.7cm);

	\draw[->, red] (6)+(207:.7) arc (207:183:.7cm);
	\draw (4) to node {}(1);
	\draw (4) to node {}(2);
	\draw (4) to node {}(6);
\end{tikzpicture}}
\caption{Example of two triangulations where the Brauer graph algebras associated to the  Brauer graphs given by the triangulations including the boundary arcs are derived equivalent.}\label{Quiver BGA Frozen}
\end{figure}
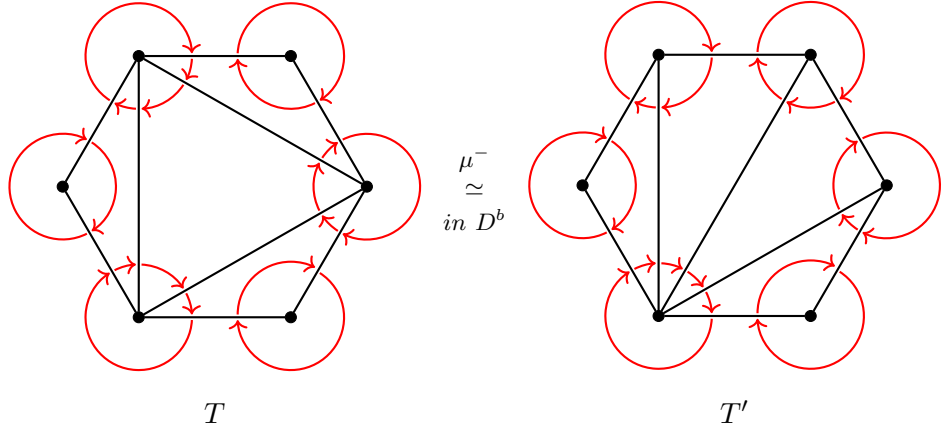

  }
\end{Example}

\begin{remark} {\rm By Theorem~\ref{Thm Flip Brauer Graph} up to derived equivalence there is a unique Brauer graph algebra associated to any marked surface $(S,M)$. This is similar to the corresponding situation for generalised cluster categories as defined by Amiot, see for example \cite{Amiot}. Namely,   if $\cC$ and $\cC' $ are two cluster categories associated to two quivers with potential $(Q,W)$ and $(Q',W')$ associated to two triangulations
$T$ and $T'$ of $(S,M)$, then  
$\cC$ and $\cC'$ are triangle equivalent \cite{KY}.   }
\end{remark}

\subsection{Brauer graph algebras and frozen Jacobian algebras}

In this section we show that the frozen Jacobian algebra associated to a triangulation of a marked disc is closely related to the Brauer graph algebra whose Brauer graph corresponds to the triangulation (including the boundary arcs).  

\begin{definition}\cite{BIRS} {\rm An {\it ice quiver with potential} $(Q,W, F)$ is a quiver with potential $(Q,W)$ and a subset $F$ of vertices of $Q$, called the {\it frozen vertices}. 

The complete path algebra $\widehat{KQ}$ is the completion of the path algebra $ KQ$ with respect to the
ideal $\mathcal R$ generated by the arrows of $Q$.
A potential on $Q$ is an element of the closure $Pot(KQ)$ of the space generated by all non-trivial
cyclic paths of $Q$. We say two potentials are cyclically equivalent if their difference is in the
closure of the space generated by all differences $ a_1 a_2 \ldots a_s - a_2 \ldots a_sa_1$, where $ a_1 a_2 \ldots a_s $ is a cycle.

For a path $p$ in $ Q$, let $\partial_p: Pot(KQ) \to \widehat{KQ} $ be the unique continuous linear map which for a cycle $c$ is defined by 
$\partial_p (c) = \sum_{ upv = c} vu $ where $u$ and $v$ might be the trivial path at $s(c)$ and $t(c)$, respectively.  If $p = a $ for some arrow $a$ in $Q$ then $\partial_a $ is called the \emph{cyclic derivative with respect to  $a$}.

The {\it frozen  Jacobian algebra of $(Q,W,F)$} is given by
$ {\mathcal J}(Q,W,F) = \widehat{KQ}/I_F$
where $\widehat{KQ}$ is the complete path algebra,  $I_F$ is the closure of the  ideal  $\langle \partial_a w | w \in W, a \in Q_1, s(a) \notin F \mbox{ or } t(a) \notin F \rangle$.
}
\end{definition}

In the case of a triangulation of a marked disc $(S,M,T)$, consider the  ice quiver $(Q,W, F)$ given by the quiver $Q=(Q_0, Q_1)$ where $Q_0$ corresponds to the edges of $T$ and where the set $F$ of frozen vertices is given by the boundary arcs and $Q_1$ is induced by the orientation. So in particular $Q_1$ includes a boundary arrow for all marked points incident to more than 2 arcs. The potential is given by 
$$ W = \sum_i C_i - \sum_j D_j$$ 
where the cycles $C_i$ are given by (all) triangles  in $T$ and the cycles $D_j$ are given by the cycles containing a boundary arrow.  Then ${\mathcal J}_T = KQ/\langle \partial_a w | w \in W, a \in Q_1, s(a) \notin F \mbox{ or } t(a) \notin F \rangle$ is the frozen Jacobian algebra of $(Q,W, F)$.  

\begin{remark} {\rm 
All relations  in $I_F$ are commutativity relations and the algebra ${\mathcal J}(Q, W, F)$ is infinite dimensional. }
\end{remark}

\begin{Example}\label{Frozen2} {\rm
Given the following  triangulation of a hexagon, consider the ice quiver with arrows as given in Figure~\ref{Quiver Frozen}  where the frozen vertices correspond to the boundary arcs.  We note that the quiver is almost identical to the quiver of the Brauer graph algebra in Example~\ref{Frozen} associated to the same triangulation of the hexagon, the difference being the absence in the ice quiver of boundary arrows  around vertices of valency two.

\begin{figure}[H]
\centering
\begin{tikzpicture}[auto, thick, rotate=30]
	\node (0) at (0,0) {};
	\node[cb] (1) at ($(0)+(30:2)$) {};
	\node[cb] (2) at ($(0)+(90:2)$) {};
	\node[cb] (3) at ($(0)+(150:2)$) {};
	\node[cb] (4) at ($(0)+(210:2)$) {};
	\node[cb] (5) at ($(0)+(270:2)$) {};
	\node[cb] (6) at ($(0)+(330:2)$) {};
			
	\draw (1)--(2)--(3)--(4)--(5)--(6)--(1);
	{\footnotesize 
	\draw[->, red] (1)+(-93:.7) arc (-93:-207:.7cm);
	\node[red] () at ($(1)+(-140:.5)$){$\delta$};
	\draw[->, red] (2)+(207:.7) arc (207:-27:.7cm);
	\node[red] () at ($(2)+(90:1)$){$\alpha_0$};
	\draw[->, red] (2)+(-33:.7) arc (-33:-57:.7cm);
	\node[red] () at ($(2)+(-45:1)$){$\alpha_1$};
	\draw[->, red] (2)+(-63:.7) arc (-63:-117:.7cm);
	\node[red] () at ($(2)+(-90:1)$){$\alpha_2$};
	\draw[->, red] (2)+(-123:.7) arc (-123:-147:.7cm);
	\node[red] () at ($(2)+(-135:1)$){$\alpha_3$};
	\draw[->, red] (3)+(27:.7) arc (27:-87:.7cm);
	\node[red] () at ($(3)+(-40:.5)$){$\eta$};
	\draw[->, red] (4)+(327:.7) arc (327:93:.7cm);
	\node[red] () at ($(4)+(-130:1)$){$\gamma_0$};
	\draw[->, red] (4)+(87:.7) arc (87:63:.7cm);
	\node[red] () at ($(4)+(75:1)$){$\gamma_1$};
	\draw[->, red] (4)+(57:.7) arc (57:3:.7cm);
	\node[red] () at ($(4)+(30:1)$){$\gamma_2$};
	\draw[->, red] (4)+(-3:.7) arc (-3:-27:.7cm);
	\node[red] () at ($(4)+(-15:1)$){$\gamma_3$};
	\draw[->, red] (5)+(147:.7) arc (150:33:.7cm);
	\node[red] () at ($(5)+(90:.5)$){$\varepsilon$};
	\draw[->, red] (6)+(87:.7) arc (87:-147:.7cm);
	\node[red] () at ($(6)+(-45:1)$){$\beta_0$};
	\draw[->, red] (6)+(117:.7) arc (117:93:.7cm);
	\node[red] () at ($(6)+(195:1)$){$\beta_1$};
	\draw[->, red] (6)+(177:.7) arc (177:123:.7cm);
	\node[red] () at ($(6)+(150:1)$){$\beta_2$};
	\draw[->, red] (6)+(207:.7) arc (207:183:.7cm);
	\node[red] () at ($(6)+(105:1)$){$\beta_3$};
	}
	\draw (4) to node {}(2);
	\draw (4) to node {}(6);
	\draw (2) to node {}(6);
\end{tikzpicture}
\caption{Quiver of the frozen Jacobian algebra associated to this triangulation of the hexagon where the frozen vertices correspond to the boundary arcs of the triangulation.}\label{Quiver Frozen}
\end{figure}
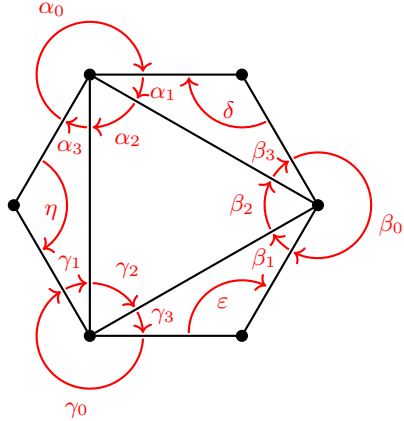

The arrows $\alpha_0, \beta_0, \gamma_0, \delta, \varepsilon, \eta$ are frozen arrows and the boundary arrows are $\alpha_0, \beta_0, \gamma_0$. The potential is given by 
$$W = \alpha_1 \beta_3 \delta + \gamma_3 \varepsilon \beta_1 + \alpha_2 \gamma_2 \beta_2  + \eta \gamma_1 \alpha_3 - \alpha_0 \alpha_1 \alpha_2 \alpha_3 - \beta_0 \beta_1 \beta_2 \beta_3 + \gamma_0 \gamma_1 \gamma_2 \gamma_3.$$ Then 
\begin{align*} I_F =  \langle \beta_3 \delta - \alpha_2 \alpha_3 \alpha_0, & \gamma_2 \beta_2 - \alpha_1 \alpha_0 \alpha_1, \eta \gamma_1 - \alpha_0 \alpha_1 \alpha_2, \gamma_3 \varepsilon - \beta_2 \beta_3 \beta_0,  \alpha_2 \gamma_2 - \beta_3 \beta_0 \beta_1, \\  & \delta \alpha_1 - \beta_0 \beta_1 \beta_2, 
\alpha_3 \eta - \gamma_2 \gamma_3 \gamma_0, \beta_2 \alpha_2 - \gamma_3 \gamma_0 \gamma_1, \varepsilon \beta_1 - \gamma_0 \gamma_1 \gamma_2
 \rangle.\end{align*}

}
\end{Example}

Let $\cF$ be the category of maximal Cohen-Macaulay modules over the Gorenstein tiled $K[x]$-order defined in \cite{DL}.  Then Demonet and Luo show the following

\begin{theorem}\cite{DL}\label{DL}
Let $P_n$ be the disc with $n$ marked points in the boundary and let ${\mathcal F}$ be the associated category of maximal Cohen-Macaulay modules over the Gorenstein tiled $K[x]$-order defined in \cite{DL}. Then 

1) The stable category $\underline{\mathcal F}$ is triangle equivalent to the cluster category of type $A_{n-3}$. 

2) There are bijections between 
the indecomposable objects in $\mathcal F$ and the arcs between marked points in $P_n$, including boundary arcs. 

3) The bijection in 2) induces a bijection between triangulations $T$ in $P_n$ and the cluster-tilting objects $M_T$ in $\cF$.  

3) The frozen Jacobian algebra ${\mathcal J}(Q,W, F)$ associated to a triangulation $T$ of $P_n$ is isomorphic to ${\rm End}_{\mathcal F}(M_T)^{\rm op}$. 
\end{theorem}

We note that $\underline{\mathcal F}$ is a Hom-finite 2-Calabi-Yau triangulated category. So the following Theorem by Palu can be applied. 

\begin{theorem}\cite{Palu}\label{Palu}
Let $\underline{\mathcal C}$ be a Hom-finite 2-Calabi-Yau triangulated category which is the stable category of a Frobenius category $\cC$. Let $\overline{M}, \overline{M}'$ be cluster-tilting objects in $\underline{\mathcal C}$ with pre-images $M, M'$ in $\cC$. Then there is a triangle equivalence 
$$\cD(\End_\cC(M)^{\rm op}{\rm -Mod}) \simeq \cD(\End_\cC(M')^{\rm op}{\rm -Mod})$$ 
where, for $A$ an algebra, $\cD(A  {\rm -Mod})$ denotes the derived category of all $A$-modules. 
\end{theorem}

Combining Theorem~\ref{DL} and Theorem~\ref{Palu}, it was observed in \cite{MS} that we obtain the following:

\begin{theorem}\label{Derived Frozen}
Let ${\mathcal J}(Q,W,F)$ and ${\mathcal J}(Q', W', F')$ be two frozen Jacobian algebras associated to two triangulations of a polygon. Then there is a triangle equivalence
 $$\cD({\mathcal J}(Q,W,F)-{\rm Mod}) \simeq \cD({\mathcal J}(Q',W',F')-{\rm Mod}).$$ 
\end{theorem}

{\bf Summary.} We will summarise the comparison between frozen Jacobian algebras and Brauer graph algebras associated to the same triangulation of a polygon. 
Given a triangulation $T$ of a polygon, let ${\mathcal J}_T$ be the the frozen Jacobian algebra associated to $T$  and let $A_T$ be the Brauer graph algebra where the Brauer graph is given by  $T$ including the boundary arcs. Then the quiver of ${\mathcal J}_T$   is almost identical to the quiver of $A_T$ (compare, for example the quiver in Figure~\ref{Quiver Frozen} with the quiver in  left hand side of Figure~\ref{Quiver BGA Frozen}). The only difference is the absence in the quiver of the frozen Jacobian algebra of boundary arrows around vertices that are incident to only 2 arcs (i.e. around vertices that are not incident with any internal arcs).  

Furthermore, given triangulations $T, T'$ of the same polygon, by Theorem~\ref{Derived Frozen} we have a derived equivalence of the unbounded derived categories of the module category of all modules over the associated frozen Jacobian algebras ${\mathcal J}_{T}$ and ${\mathcal J}_{T'}$, 
$$\cD({\mathcal J}_T-{\rm Mod}) \simeq \cD({\mathcal J}_{T'}-{\rm Mod})$$
and by Theorem~\ref{Thm Flip Brauer Graph} we have a derived equivalence of the bounded derived categories of the module categories of finitely generated modules over the associated Brauer graph algebras $A_T$ and $A_{T'}$
$$\cD^\flat (A_T) \simeq \cD^\flat (A_{T'}).$$
This might indicate that there is a structural connection between the frozen Jacobian algebra and the Brauer graph algebra associated to the same triangulation (of a polygon). That this should be the case has also been highlighted by the results in \cite{Demonet} and in \cite{Ladkani3}.

\section{Auslander-Reiten components}\label{Section AR components}

Auslander-Reiten theory for self-injective special biserial algebras has been well studied. Erdmann and Skowro\'nski classify in \cite{ES} the Auslander-Reiten components for self-injective special biserial algebras. As Brauer graph algebras are symmetric special biserial, this gives the general structure of the Auslander-Reiten quiver for Brauer graph algebras. The result for representation-infinite algebras depends on the notions of domestic and polynomial growth. For the convenience of the reader, we recall these notions.

\subsection{Finite, tame and wild representation type, domestic algebras and algebras of polynomial growth}\label{Section Representation Type}

Let $A$ be a finite dimensional $K$-algebra, then we say that $A$ is of \textit{finite representation type}, if up to isomorphism there are only finitely many distinct indecomposable $A$-modules. 
 
 We say $A$ is  of \textit{tame representation type}, if for any positive integer $n$, there exists a finite number of  $A -K[x]$-bimodules $M_i$, for $1 \leq i \leq d_n$ such that $M_i$ is finitely generated and  is free as a left $K[x]$-module and such that all but a finite number of isomorphism classes of indecomposable $n$-dimensional $A$-modules are isomorphic 
  to $M_i \otimes_{K[x]} K[x]/ (x -a)$, for $a \in K$ and $1 \leq i \leq d_n$. For each $n $, let $\mu(n)$ be the least number of such $A -K[x]$-bimodules. 
   Following \cite{Skowronski1}, we say that $A$ is of {\it polynomial growth} if there exists a positive integer $m$ such that $\mu(n) \leq  n^m$ for all $n \geq 2$. We say that $A$ is {\it domestic} if $\mu(n) \leq m$ for all $n \geq 1$ \cite{Ringel2} and we say that $A$ is {\it $d$-domestic} if $d$ is the least such integer $m$.

Note that every domestic algebra is of polynomial growth \cite{Skowronski2}.

Let $K \langle x,y \rangle$ be the group algebra of the free group in two generators. 
We say that $A$  is of {\it wild representation type} if there is a $K\langle x,y \rangle -A$ bimodule M such that $M$ is free as a $K \langle x,y \rangle$-module, and if $X$ is an indecomposable $K \langle x,y \rangle$-module then $M \otimes_K X$ is an indecomposable $A$-module, and if for some $K \langle x,y \rangle$-module $Y$ we have $M \otimes_K X \simeq M \otimes_K Y$ then $X \simeq Y$.

By \cite{Drozd} an algebra of infinite representation type that is not of tame representation type is of wild representation type.  

\subsection{Domestic Brauer graph algebras}

In the following let $\Lambda$ be a representation-infinite Brauer graph algebra with Brauer graph $G = (G_0, G_1, m, \mathfrak{o})$. Bocian and Skowro\'nski give a characterization of the domestic Brauer graph algebras. 

\begin{theorem}\cite{BS}\label{DomesticBGA}
Let $\Lambda$ be a Brauer graph algebra with Brauer graph $G$. Then 
\begin{itemize}
\item[(1)] $\Lambda$ is $1$-domestic if and only if one of the following holds
\begin{itemize}
\item[$\bullet$] $G$ is a tree with $m(i) =2$ for exactly two vertices $i_0, i_1 \in G_0$ and $m(i) = 1$ for all $ i \in G_0, i \neq i_0, i_1$. 
\item[$\bullet$] $G$ is a graph with a unique cycle of odd length and $m \equiv 1$. 
\end{itemize}
\item[(2)] $\Lambda$ is $2$-domestic if and only if $G$ is a graph with a unique cycle of even length and $m \equiv 1$. 
\item[(3)] There are no $n$-domestic Brauer graph  algebras for $n \geq 3$.  
\end{itemize}
\end{theorem}

\begin{Example}\label{Example of a 1-domestic BGA}
{\rm (1) Example of a Brauer graph of a 1-domestic Brauer graph algebra. 

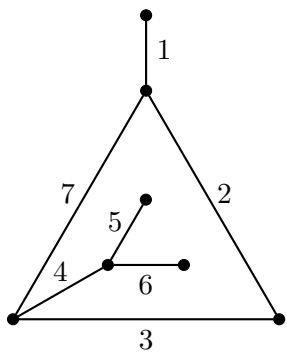
\begin{figure}[H]
	\centering
	\begin{tikzpicture}[auto, thick]
	\node[cb] (1) at (0,0) {};
	\node[cb] (2) at ($(1)+(0:3.5)$) {};
	\node[cb] (3) at ($(1)+(60:3.5)$) {};
	\node[cb] (4) at ($(3)+(90:1)$) {};
	\node[cb] (5) at ($(1)+(30:1.44)$) {};
	\node[cb] (6) at ($(5)+(0:1)$) {};
	\node[cb] (7) at ($(5)+(60:1)$) {};
	
	\draw (1) to [below] node {$3$}(2);
	\draw (2) to [right, pos=.55] node {$2$}(3);
	\draw (3) to [right, pos=.55] node {$1$}(4);
	\draw (3) to [left, pos=.45] node {$7$}(1);
	\draw (1) to [above] node {$4$}(5);
	\draw (5) to [below] node {$6$}(6);
	\draw (5) to [left, pos=.7] node {$5$}(7);
	
	\end{tikzpicture}
	\caption{Brauer graph of a 1-domestic Brauer graph algebra}
\end{figure}

(2) Example of Brauer graph of a 2-domestic Brauer graph algebra.

\begin{figure}[H]
	\centering
	\begin{tikzpicture}[auto, thick]
	\node[cb] (1) at (0,0) {};
	\node[cb] (2) at ($(1)+(40:2)$) {};
	\node[cb] (3) at ($(1)+(-40:2)$) {};
	\node[cb] (4) at ($(1)+(0:1.53)$) {};
	
	\node[cb] (5) at ($(3)+(40:2)$) {};
	\node[cb] (6) at ($(5)+(0:2)$) {};
	\node[cb] (7) at ($(6)+(40:2)$) {};
	\node[cb] (8) at ($(6)+(-40:2)$) {};
		
	\draw (1) to [above, pos=.45] node {$1$}(2);
	\draw (1) to [below, pos=.45] node {$7$}(3);
	\draw (1) to [above, pos=.7] node {$8$}(4);
	\draw (2) to [above, pos=.55] node {$2$}(5);
	\draw (3) to [below, pos=.55] node {$6$}(5);
	\draw (5) to [above] node {$3$}(6);
	\draw (6) to [above, pos=.45] node {$4$}(7);
	\draw (6) to [below, pos=.45] node {$5$}(8);
	
	\end{tikzpicture}
	\caption{Brauer graph of a 2-domestic Brauer graph algebra}
\end{figure}
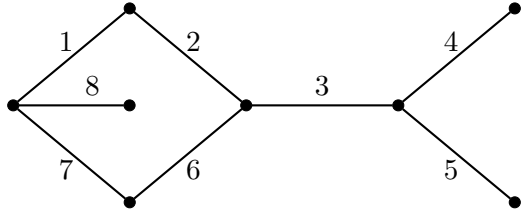
}
\end{Example}

\subsection{The stable Auslander-Reiten quiver of a self-injective special biserial algebra}

We refer the reader to the textbooks \cite{ASS, ARS, Sch} for the definition and general set-up of Auslander-Reiten theory. Let $\Lambda$ be a self-injective $K$-algebra (recall that a finite-dimensional $K$-algebra is self-injective if every projective module is injective). 
In the following denote by  $\Gamma_\Lambda$  the Auslander-Reiten quiver of $\Lambda$ and by $_s\Gamma_\Lambda$  the stable Auslander-Reiten quiver of $\Lambda$. Denote by $\mod - \Lambda$ the module category of finitely generated $\Lambda$-modules and by  $\underline{\mod} -\Lambda$ the stable category of $\Lambda$.  Let $\nu : \Lambda -\mod \to \Lambda-\mod$ be the {\it Nakayama functor},  given by  sending  $N \in \Lambda - \mod$ to $D(\Hom _\Lambda (N, \Lambda)) $. Let $\Omega$ be {\it Heller's syzygy functor}, that is $\Omega : \underline{\mod} -\Lambda \to \underline{\mod} -\Lambda$  given by $M \in \underline{\mod} \Lambda$, $\Omega(M) = \ker \pi$ where $\pi : P_M \rightarrow M$ is the projective cover of $M$. Then  the {\it Auslander-Reiten translate}  $\tau$ is given by $\tau M = \nu \Omega^2 M$, for  $M \in \mod - \Lambda$. Recall that if $\Lambda$ is a symmetric $K$-algebra then $\nu =$ Id and $\tau = \Omega^2$. 

Following Riedtmann \cite{Riedtmann}, any connected component of $_s\Gamma_\Lambda$ 
is of the form $\mathbb{Z} T/G$ where $T$ is an oriented tree and $G$ is an 
admissible automorphism group.  We call $T$ the tree class of $_s\Gamma_\Lambda$. 
The tree class of any component of the stable Auslander-Reiten quiver of an algebra 
containing a periodic module, that is a module $M$ such that $\tau^n M \simeq M$, 
for some $n$, is equal to $A_\infty$ \cite{HPR}. The shapes of the translations 
quivers $\mathbb{Z}A_\infty, \mathbb{Z}A_\infty / \langle \tau^n \rangle, 
\mathbb{Z}A_\infty^\infty, \mathbb{Z}\tilde{A}_{p,q}$ are described in \cite{HPR}. 
By $\tilde{A}_{p,q}$ we denote the following orientation of the quiver with 
underlying extended Dynkin diagram of type $\tilde{A}_n$ 

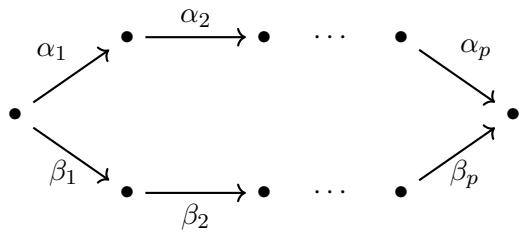
\begin{figure}[H]
	\centering
\begin{tikzpicture}[auto, thick]
	\node (0) at (0,0) {$\bullet$};
	\node (1) at ($(0)+(35:1.8)$) {$\bullet$};
	\node (11) at ($(0)+(-35:1.8)$) {$\bullet$};
	\node (2) at ($(1)+(0:1.8)$) {$\bullet$};
	\node (22) at ($(11)+(0:1.8)$) {$\bullet$};
	\node (x) at ($(2)+(0:.9)$) {$\cdots$};
	\node (3) at ($(2)+(0:1.8)$) {$\bullet$};
	\node (xx) at ($(22)+(0:.9)$) {$\cdots$};
	\node (33) at ($(22)+(0:1.8)$) {$\bullet$};
	\node (4) at ($(3)+(-35:1.8)$) {$\bullet$};
	
	\draw[->] (0) to [pos=0.6] node {$\alpha_1$}(1);
	\draw[->] (0) to [below, pos=0.4] node {$\beta_1$}(11);
	\draw[->] (1) to node {$\alpha_2$}(2);
	\draw[->] (11) to [below] node {$\beta_2$}(22);
	\draw[->] (3) to [pos=.4]node {$\alpha_p$}(4);
	\draw[->] (33) to [below, pos=0.6] node {$\beta_p$}(4);

\end{tikzpicture}
\caption{$\tilde{A}_{p,q}$}
\end{figure}

If a connected component of  $_s\Gamma_\Lambda$ does not contain any projective or injective module, we call that component {\it regular}.

The following two theorems give a complete description of the possible components of the stable Auslander-Reiten quiver of a self-injective special biserial algebra of tame representation type.

\begin{theorem}\cite[Theorem 2.1]{ES}\label{Thm ES domestic}
Let $\Lambda$ be a self-injective special biserial algebra. Then the following are equivalent: 

\begin{itemize}
\item[(1)] $\Lambda$ is representation-infinite domestic. 
\item[(2)] $\Lambda$ is representation-infinite of polynomial growth. 
\item[(3)] $_s\Gamma_\Lambda$ has a component of the form $
\mathbb{Z}\tilde{A}_{p,q}$.
\item[(4)] $_s\Gamma_\Lambda$ is infinite but has no component of the form 
$\mathbb{Z}A_\infty^\infty$.
\item[(5)] All but a finite number of components of $\Gamma_\Lambda$ are of the form $ \mathbb{Z}A_\infty / \langle \tau \rangle$.
\item[(6)] $_s\Gamma_\Lambda$ is a disjoint union of $m$ components of the form $\mathbb{Z}\tilde{A}_{p,q}$, $m$ components of the form $\mathbb{Z}A_{\infty}/ \langle \tau^p \rangle$, and $m$ components of the form $\mathbb{Z}A_{\infty}/ \langle \tau^q \rangle$. 
\end{itemize}
\end{theorem}

\begin{theorem}\cite[Theorem 2.2]{ES}\label{Thm ES non-domestic}
Let $\Lambda$ be a self-injective special biserial algebra. Then the following are equivalent: 
\begin{itemize}
\item[(1)] $\Lambda$ is not of polynomial growth. 
\item[(2)] $_s\Gamma_\Lambda$ has a component  of the form 
$\mathbb{Z}A_\infty^\infty$.
\item[(3)] $\Gamma_\Lambda$ has infinitely many (regular) components of the form 
$\mathbb{Z}A_\infty^\infty$.
\item[(4)]  $_s\Gamma_\Lambda$ is a disjoint union of a finite number of components of the form $\mathbb{Z}A_{\infty}/ \langle \tau^n \rangle$, with $n >1$, infinitely many components of the form $\mathbb{Z}A_{\infty}/ \langle \tau \rangle$ and infinitely many components of the form $\mathbb{Z}A_\infty^\infty$.
\end{itemize}
\end{theorem}


\subsection{Green walks, double stepped Green walks and exceptional tubes}

In \cite{Gr} Sandy Green defined a {\it walk around the Brauer tree} and showed that for certain modules it encodes  their minimal projective resolution.

 Let $A$ be a Brauer graph algebra. We call a simple $A$-module with uniserial projective cover a {\it uniserial simple $A$-module}. The uniserial simple $A$-modules correspond exactly to the truncated edges in the Brauer graph. Recall that a truncated edge is a leaf in the Brauer graph where the leaf vertex has multiplicity one. Green showed that the successive terms in the minimal projective resolution of a uniserial simple $A$-module are given by 'walking' around the Brauer tree. Roggenkamp showed  that the same is true for a Brauer graph algebra \cite{Ro}.

\subsubsection{Definition of Green's walk around the Brauer graph}

A Green walk  on a Brauer graph is given by a graph theoretic path on the graph underlying the Brauer graph. It is defined in terms of successor relations. Let $A = KQ/I$ be a Brauer graph algebra with Brauer graph $G$. 

{\bf Definition of Green walks.} 
 Suppose first that $G$ consists of a single loop $i_0$. Then there are two Green walks of period 1 on $G$  both given by $i_0$, see  Figure~\ref{Green walk single loop}.
\begin{figure}[H] 
	\centering
	\begin{tikzpicture}[auto, thick]

	\node[] (0) at (0,0) {};
	\node[cb] (1) at ($(0,0)+(30:2)$) {};
	\node[] (2) at ($(1)+(-30:2)$) {};
	\node (i1) at ($(1)+(90:2)$){$i_0$};
	\node (a1) at ($(1)+(-90:.3)$){$a_0=a_1$};
	
	\draw (1) to [out=140, in=40, looseness=70] node {}(1);
	\draw[->,red] ($(1)+(40:.4cm)$) to [out=-30, in=210, looseness=10] node {}($(1)+(-220:.4)$);
	\draw[->,green] (1) +(133:.7cm) arc (133:45:.7cm);
		
	\end{tikzpicture}
	\caption{Green walk if the Brauer graph consists of a single loop}\label{Green walk single loop}
\end{figure}
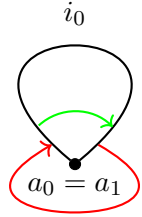

Now suppose that $G$ contains more than one edge and 
let $i_0$ be an edge in $G$ with vertices $a_0$ and $a_1$. Suppose first that $a_0 \neq a_1$, that is suppose that $i_0$ is not a loop.  Let $i_1$ be the successor of $i_0$ at the vertex $a_1$. 
Now let $a_2$ be the other vertex of $i_1$ and let $i_2$ be the successor of $i_1$ at vertex $a_2$.  Note that if $i_1$ is a loop then $a_1 = a_2$ and the successor $i_2$ of $i_1$ appears as in Figure~\ref{Green walk loops 1 and 2} (1) and (2). 
\begin{figure}[H]
	\centering
	\begin{tikzpicture}[auto, thick]
	\node (a) at (0,4) {(1)};
	\node[cb] (0) at (0,0) {};
	\node[cb] (1) at ($(0)+(30:2)$) {};
	\node[cb] (2) at ($(1)+(-30:2)$) {};
	\node (3) at ($(1)+(110:1.5)$) {};
	\node (4) at ($(1)+(70:1.5)$) {};
	\node (p1) at ($(1)+(89:1)$){$\cdots$};
	\node (i1) at ($(1)+(90:2.7)$){$i_1$};
	\node (a0) at ($(0)+(-30:.5)$){$a_0$};
	\node (a3) at ($(2)+(210:.5)$){$a_3$};
	\node (aa) at ($(1)+(-90:.5)$){$a_1=a_2$};
	
	\draw (3)--(1)--(4);
	\draw (0) to node {$i_0$}(1);
	\draw (1) to node {$i_2$}(2);
	
	\draw (1) to [out=140, in=40, looseness=100] node {}(1);
	\draw[loosely dotted] (3) to node {}($(3)+(105:.6)$);
	\draw[loosely dotted] (4) to node {}($(4)+(75:.6)$);
	\draw[loosely dotted] ($(0)+(210:.05)$) to node {}($(0)+(210:.6)$);
	\draw[loosely dotted] ($(2)+(-30:.05)$) to node {}($(2)+(-30:.6)$);
	
	\draw[->,red] (1) +(207:.7cm) arc (207:142:.7cm);
	\draw[->,red] (1) +(38:.7cm) arc (38:-27:.7cm);
	\node (b) at ($(a)+(0:6)$) {$(2)$};
	\node[cb] (0b) at ($(0)+(0:6)$)  {};
	\node[cb] (1b) at ($(0b)+(30:2)$) {};
	\node[cb] (2b) at ($(1b)+(-30:2)$) {};
	\node[cb] (3b) at ($(1b)+(110:1.5)$) {};
	\node[cb] (4b) at ($(1b)+(70:1.5)$) {};
	\node (pb) at ($(1b)+(89:1)$){$\cdots$};
	\node (i1b) at ($(1b)+(90:2.7)$){$i_1$};
	\node (a1a2b) at ($(1b)+(-90:.5)$){$a_1=a_2$};
	\node (pb) at ($(1b)+(-89:1)$){$\cdots$};
	
	\draw (0b)--(1b)--(2b);
	\draw (1b) to [left, pos=.7] node {$i_2$}(3b);
	\draw (1b) to [right,pos=.7] node {$i_0$}(4b);
	\draw (1b) to [out=140, in=40, looseness=100] node {}(1b);
	\draw[loosely dotted] ($(3b)+(105:.05)$) to node {}($(3b)+(105:.6)$);
	\draw[loosely dotted] ($(4b)+(75:.05)$) to node {}($(4b)+(75:.6)$);
	\draw[loosely dotted] ($(0b)+(210:.05)$) to node {}($(0b)+(210:.6)$);
	\draw[loosely dotted] ($(2b)+(-30:.05)$) to node {}($(2b)+(-30:.6)$);
	
	\draw[->,green] (1b) +(135:.7cm) arc (135:113:.7cm);
	\draw[->,green] (1b) +(67:.7cm) arc (67:45:.7cm);
	
	\end{tikzpicture}
	\caption{}\label{Green walk loops 1 and 2}
\end{figure}

  Let $a_3$ be the other vertex of $i_2$ and let $i_3$ be the successor of $i_2$ at $a_3$ (treating the case that $i_2$ is a loop as above). 
 In this way, we define an infinite sequence of edges $i_0, i_1, i_2, \ldots$, called a \textit{Green walk starting at $i_0$}.  Let $k$ be the minimal number $n$ such that $i_{n+k}=i_k$ for all $k \geq 0$.  Then we say that $(i_0, i_1, \ldots, i_{k-1})$ is a \textit{Green walk of period $k$}. Since $G$ is a finite graph, every Green walk is periodic.  
 
 Note that if   $i_0$ is a loop, that is if $a_0 = a_1$, then there are two distinct Green walk starting at $i_0$ depending on the  successor $i_1$  of $i_0$ chosen. This is illustrated in Figure~\ref{Green walks starting on loops}. 

\begin{figure}[H]
	\centering
	\begin{tikzpicture}[auto, thick]
	\node (a) at (0,4) {$(1)$};
	\node[] (0) at (0,0) {};
	\node[cb] (1) at ($(0,0)+(30:2)$) {};
	\node[] (2) at ($(1)+(-30:2)$) {};
	\node (i1) at ($(1)+(90:2)$){$i_0$};
	\node (a1) at ($(1)+(-90:.5)$){$a_0=a_1$};
	
	\draw ($(0)+(90:.5)$) to [bend left] node {}($(2)+(90:.5)$);
	\draw (1) to [out=140, in=40, looseness=70] node {}(1);
	\draw[->,red] (1) +(183:.7cm) arc (183:142:.7cm);
	\draw[->,red] (1) +(38:.7cm) arc (38:-3:.7cm);
	\draw[->,green] (1) +(133:.7cm) arc (133:45:.7cm);
	
	\node (b) at ($(a)+(0:4.5)$) {$(2)$};
	\node (0b) at ($(0)+(0:4.5)$) {};
	\node[cb] (1b) at ($(0b)+(30:2)$) {};
	\node (3b) at ($(1b)+(110:1.5)$) {};
	\node (4b) at ($(1b)+(70:1.5)$) {};
	\node (pb) at ($(1b)+(89:1)$){$\cdots$};
	\node (i0b) at ($(1b)+(90:2.7)$){$i_0$};
	\node (a0b) at ($(1b)+(-90:.5)$){$a_0=1$};
		
	\draw (1b) to [left, pos=.7] node {}(3b);
	\draw (1b) to [right,pos=.7] node {}(4b);
	\draw (1b) to [out=140, in=40, looseness=100] node {}(1b);
	\draw[loosely dotted] ($(3b)+(105:.05)$) to node {}($(3b)+(105:.6)$);
	\draw[loosely dotted] ($(4b)+(75:.05)$) to node {}($(4b)+(75:.6)$);
	\draw[->,green] (1b) +(135:.7cm) arc (135:113:.7cm);
	\draw[->,green] (1b) +(67:.7cm) arc (67:45:.7cm);
	\draw[->,red] (1b) +(40:1cm) arc (43:-220:1cm);
	\node (c) at ($(b)+(0:4.5)$) {$(3)$};
	\node (0c) at ($(0b)+(0:4.5)$) {};
	\node[cb] (1c) at ($(0c)+(30:2)$) {};
	\node[] (2c) at ($(1c)+(-30:2)$) {};
	\node (3c) at ($(1c)+(110:1.5)$) {};
	\node (4c) at ($(1c)+(70:1.5)$) {};
	\node (pc) at ($(1c)+(89:1)$){$\cdots$};
	\node (i0c) at ($(1c)+(90:2.7)$){$i_0$};
	
	\draw ($(0c)+(90:.5)$) to [bend left] node {}($(2c)+(90:.5)$);	
	\draw (1c) to [left, pos=.7] node {}(3c);
	\draw (1c) to [right,pos=.7] node {}(4c);
	\draw (1c) to [out=140, in=40, looseness=100] node {}(1c);
	\draw[loosely dotted] ($(3c)+(105:.05)$) to node {}($(3c)+(105:.6)$);
	\draw[loosely dotted] ($(4c)+(75:.05)$) to node {}($(4c)+(75:.6)$);
	\draw[->,red] (1c) +(183:.7cm) arc (183:142:.7cm);
	\draw[->,red] (1c) +(38:.7cm) arc (38:-3:.7cm);
	\draw[->,green] (1c) +(135:.7cm) arc (135:113:.7cm);
	\draw[->,green] (1c) +(67:.7cm) arc (67:45:.7cm);
	
	\end{tikzpicture}
	\caption{Green walk starting on loops}\label{Green walks starting on loops}
\end{figure}
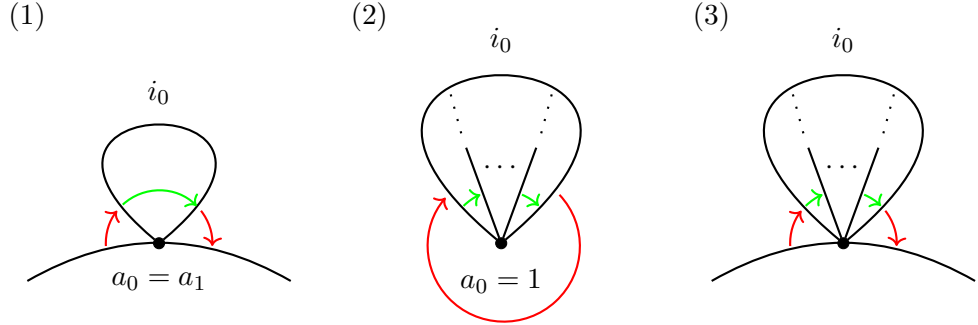

Every edge appears  in exactly two (not necessarily distinct) Green walks, corresponding to the two successors of a given edge. Any  edge corresponding to a leaf (truncated or not) appears twice in the same Green walk.

In order to avoid the complication with loops in the definition of a Green walk,  many authors consider half-edges instead of edges using, for example, the language of ribbon graphs, see \cite{DF, Ro, Du, AAC}. 

\begin{remark}{\rm Green walks are independent of the multiplicities in the Brauer graph, they only depend on the graph and the orientation. 
}
\end{remark}

Given a Green walk $i_0, i_1, i_2, \ldots $ at some edge $i_0$ of a Brauer graph $G$, a {\it double stepped Green walk at $i_0$} is given by the sequence of edges $i_0, i_2, i_4, \ldots $ in $G$.

\begin{theorem}\cite{Gr, Ro}\label{Thm Green walk}  Let $A$ be a Brauer graph algebra with Brauer graph $G$. 

(1) Let $S$ be a uniserial simple $A$-module and let $i_0$ be the edge in $G$ corresponding to $S$ and let $i_0, i_1, \ldots$ be the Green walk starting at $i_0$. Then  the $n$th projective in a minimal projective resolution of $S$ is given by $P_{i_n}$, the indecomposable projective associated to the edge $i_n$. 

(2) Let $j_0$ be the edge in $G$ corresponding to the simple top of a maximal uniserial submodule $M$ of some projective indecomposable module and let $j_0, j_1, \ldots$ be the Green walk starting at $j_0$.  Then the $n$th projective in a minimal projective resolution of $M$ is given by $P_{j_n}$. 
\end{theorem}

\begin{Example}
{\rm 
Example of the two Green walks on the 1-domestic Brauer graph given in Example~\ref{Example of a 1-domestic BGA}  and the  minimal projective resolution of the simple corresponding to the edge 1 given by the walk starting at 1. 
\begin{figure}[H]
	\centering
	\begin{tikzpicture}[auto, thick]
	\node[cb] (1) at (0,0) {};
	\node[cb] (2) at ($(1)+(0:4)$) {};
	\node[cb] (3) at ($(1)+(60:4)$) {};
	\node[cb] (4) at ($(3)+(90:2)$) {};
	\node[cb] (5) at ($(1)+(30:1.55)$) {};
	\node[cb] (6) at ($(5)+(0:1.3)$) {};
	\node[cb] (7) at ($(5)+(60:1.3)$) {};
	
	\draw (1) to [below] node {$3$}(2);
	\draw (2) to [right, pos=.55] node {$2$}(3);
	\draw (3) to [right, pos=.8] node {$1$}(4);
	\draw (3) to [left, pos=.45] node {$7$}(1);
	\draw (1) to [above, pos=.4] node {$4$}(5);
	\draw (5) to [below, pos=.47] node {$6$}(6);
	\draw (5) to [left, pos=.7] node {$5$}(7);
	
	\draw[->, green] (1)+(-5:1.8cm) arc (-5:-295:1.8cm){};
	\draw[->, green] (2)+(115:1.8cm) arc (115:-175:1.8cm){};
	\draw[->, green] (3)+(85:.9cm) arc (85:-40:1.35cm){};
	\draw[<-, green] (3)+(95:.9cm) arc (95:220:1.35cm){};
	\draw[->, green] (4)+(265:.9cm) arc (265:-85:.9cm){};
	
	\draw[->, blue] (5)+(200:.5cm) to [out=80,in=195]node {} ($(5)+(70:.6cm)$);
	\draw[<-, blue] (5)+(220:.45cm) to [out=-20,in=230] node {} ($(5)+(-5:.6cm)$);
	\draw[->, blue] (5)+(55:.6cm) to [out=-20,in=100] node {} ($(5)+(5:.6cm)$);
	
	\draw[->, blue] (5)+(220:.6cm) to [out=-30,in=140] node {} ($(1)+(3:1.9cm)$);
	\draw[->, blue] (1)+(3:2cm) to [out=5,in=230,looseness=.5] node {} ($(2)+(123:.8cm)$);
	\draw[->, blue] (2)+(125:.9cm) to [out=130,in=-20,looseness=.8] node {} ($(3)+(245:.8cm)$);
	\draw[->, blue] (3)+(245:.9cm) to [out=-45,in=120,looseness=0.6] node {} ($(5)+(200:.6cm)$);
	
	\draw[blue] (7)+(230:.5cm) to [out=120,in=150] node {} ($(7)+(60:.3cm)$);
	\draw[<-,blue] (7)+(250:.5cm) to [out=0,in=-30] node {} ($(7)+(60:.3cm)$);
	
	\draw[blue] (6)+(170:.55cm) to [out=60,in=90] node {} ($(6)+(0:.3cm)$);
	\draw[<-,blue] (6)+(-170:.55cm) to [out=-60,in=-90] node {} ($(6)+(0:.3cm)$);
	 	
	\end{tikzpicture}
	\caption{ Brauer graph with two distinct Green walks.}
\end{figure}

\begin{figure}[H]
	\centering
\begin{tikzpicture}[auto, thick]
	\node  (0) at (0,0) {$\dots$};
	\node  (1) at ($(0)+(-55:2)$) {$1$};
	\node  (2) at ($(1)+(55:2)$) {$\dimFour1271$};
	\node  (3) at ($(2)+(-55:2)$) {$\dimThree127$};
	\node  (4) at ($(3)+(55:2)$) {$\matrixTwo{7}{\dimTwo12}{\dimTwo43}{7}$};
	\node  (5) at ($(4)+(-55:2)$) {$\dimThree743$};
	\node  (6) at ($(5)+(55:2)$) {$\matrixTwo{3}{\dimTwo74}{2}{3}$};
	\node  (7) at ($(6)+(-55:2)$) {$\dimTwo32$};
	\node  (8) at ($(7)+(55:2)$) {$\matrixTwo{2}{3}{\dimTwo71}{2}$};
	\node  (9) at ($(8)+(-55:2)$) {$\dimThree271$};
	\node  (10) at ($(9)+(55:2)$) {$\dimFour1271$};
	\node  (11) at ($(10)+(0:1.4)$) {$1$};

	\draw[->>] (0) to node {}(1);
	\draw[right hook ->] (1) to node {}(2);
	\draw[->>] (2) to node {}(3);
	\draw[right hook ->] (3) to node {}(4);
	\draw[->>] (4) to node {}(5);
	\draw[right hook ->] (5) to node {}(6);
	\draw[ ->>] (6) to node {}(7);
	\draw[right hook ->] (7) to node {}(8);
	\draw[->>] (8) to node {}(9);
	\draw[right hook ->] (9) to node {}(10);
	\draw[->] (10) to node {}(11);

\end{tikzpicture}
\caption{Projective resolution of the uniserial simple module associated to edge 1.}
\end{figure}

}
\end{Example}

It follows immediately from Theorem~\ref{Thm Green walk} that 

\begin{corollary}
Any uniserial simple $A$-module and any maximal uniserial submodule of an indecomposable projective $A$-module  is $\Omega$-periodic. 
\end{corollary}

As a consequence by \cite{HPR}, it follows that any uniserial simple $A$-module  and any maximal uniserial submodule of an indecomposable projective $A$-module lies in a component of tree class $A_\infty$. 

Moreover, in \cite{Du} the precise location of these modules is given

\begin{proposition}
Let $A$ be a Brauer graph algebra. An indecomposable (string) module $M$ is at the mouth of an exceptional tube in the stable Auslander-Reiten quiver of $A$ if and only if $M$ is uniserial simple or a maximal uniserial submodule of an indecomposable projective $A$-module. 
\end{proposition}

\subsection{Exceptional Tubes}

Exceptional tubes are Auslander-Reiten components of the form $\mathbb{Z}A_{\infty}/ \langle \tau^{n} \rangle$, where the integer $n \geq 1$ denotes the rank of the tube and where for tubes of rank one, we  distinguish between the exceptional tubes -- that is, the tubes of rank one consisting of string modules, of which there
are finitely many -- and the homogeneous tubes of rank one, which consist only of
band modules.

\begin{theorem}\cite[Theorem 4.3]{Du} \label{Thm Du exceptional}
Let $A$ be a representation-infinite Brauer graph algebra with Brauer graph $G$. 
\begin{itemize}
\item[(1)] The number of exceptional tubes in $_s\Gamma_A$ is given by the number of double-stepped Green walks on $G$. 
\item[(2)] The length of the double-stepped Green walk gives the rank of the associated exceptional tube. 
\end{itemize}
\end{theorem}

\begin{Example} 
{\rm (1) Suppose that $A$ is a Brauer graph algebra with Brauer graph as given in Figure~\ref{Example tubes}. Then there are three distinct Green walks on $G$ of lengths 1,5, and 6, and thus 4 double-stepped Green walks of lengths 1, 3, 3, and 5 giving rise to four exceptional tubes of respective ranks 1, 3,3, and 5. 

\begin{figure}[H]
	\centering
	\begin{tikzpicture}[auto, thick]
	\node[cb] (1) at (0,0) {};
	\node[cb] (2) at ($(1)+(0:3.5)$) {};
	\node[cb] (3) at ($(2)+(90:3.5)$) {};
	\node[cb] (4) at ($(1)+(90:3.5)$) {};
	\node[cb] (5) at ($(1)+(45:1.5)$) {};
	\node[red] (arcR1) at ($(1)+(45:3.5)$) {1};
	\node[red] (arcR2) at ($(1)+(70:2.65)$) {2};
	\node[red] (arcR4) at ($(1)+(45:2.5)$) {4};
	\node[red] (arcR6) at ($(1)+(20:2.65)$) {6};
	\node[green] (arcG3) at ($(2)+(-45:1.4)$) {3};
	\node[green] (arcR4) at ($(1)+(45:-1.4)$) {4};
	\node[green] (arcR5) at ($(4)+(135:1.4)$) {5};

	\draw (1)--(2)--(3)--(4)--(1)--(5);
	\draw (3)+(45:1cm) circle (1cm);
	\draw[red, ->] ($(1)+(50:.75)$) .. controls ($(5)+(90:1.8)$) and ($(5)+(0:1.8)$) .. ($(1)+(40:.75)$);
	\draw[->, red] (2)+(177:1.7cm) arc (177:93:1.7cm){};
	\draw[->, red] (3)+(267:1.7cm) arc (267:183:1.7cm){};
	\draw[->, red] (4)+(-3:1.7cm) arc (-3:-87:1.7cm){};
	\draw[->, red] (1)+(87:1.7cm) to [out=-10,in=160, pos=.3]node {3} ($(1)+(50:.65cm)$);
	\draw[->, red] (1)+(42:.65cm) to [out=-90,in=90, pos=.7] node {5} ($(1)+(3:1.7cm)$);
	\draw[->, blue] (3)+(90:1.35cm) to [out=-25,in=115,below, pos=.4] node {1} ($(3)+(0:1.35cm)$);
	
	\draw[->, green] (1)+(-3:1.7cm) arc (-3:-267:1.7cm){};
	\draw[->, green] (2)+(87:1.7cm) arc (87:-177:1.7cm){};
	\draw[->, green] (4)+(267:1.7cm) arc (267:3:1.7cm){};
	\draw[->, green] (3)+(45:2.1cm) to [out=-20,in=0,left, pos=.5, looseness=1.3]node{2} ($(2)+(87:1.8)$);
	\draw[->, green] (4)+(0:1.8cm) to [out=90,in=110,below, pos=.5, looseness=1.3]node{1} ($(3)+(48:2.1)$);
		
	\end{tikzpicture}
	\caption{Brauer graph with three distinct Green walks.}\label{Example tubes}
\end{figure}
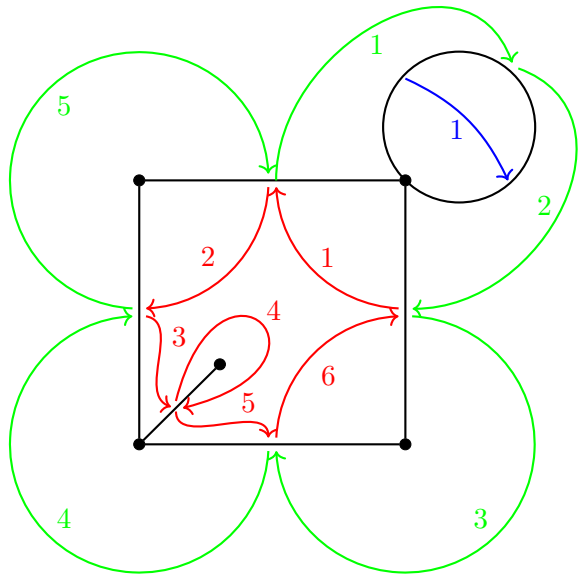
 
 More precisely,  the stable Auslander-Reiten quiver of $A$ is a disjoint union of infinitely many components of the forms $\mathbb{Z}A_{\infty}/ \langle \tau \rangle$ and $\mathbb{Z}A_\infty^\infty$  and  
 two components of the form $\mathbb{Z}A_{\infty}/ \langle \tau^3 \rangle$,  
 one component of the form $\mathbb{Z}A_{\infty}/ \langle \tau^6 \rangle$ and
  one component of the form $\mathbb{Z}A_{\infty}/ \langle \tau \rangle$ (given by a the string corresponding to the arrow inside the unique loop and not a band module as the ones giving rise to the infinitely many other component of the form $\mathbb{Z}A_{\infty}/ \langle \tau \rangle$).
    
(2)  For completeness we include one example of a non-planar Brauer graph $G'$ and the Green walks on $G'$. Suppose that $A'$ is a Brauer graph algebra with Brauer graph $G'$.   
\begin{figure}[H]
	\centering
	\begin{tikzpicture}[auto, thick]
	\node[cb] (1) at (0,0) {};
	\node[cb] (2) at ($(1)+(0:5)$) {};
	\node[cb] (3) at ($(1)+(60:4.1)$) {};
	\node[cb] (4) at ($(3)+(-90:1.5)$){};
	\node[cb] (5) at ($(0)+(25:7)$) {};
	\node[red] (Red1) at ($(5)+(40:.8)$) {1};
	\node[red] (Red2) at ($(2)+(-40:.8)$) {2};
	\node[red] (Red3) at ($(1)+(210:.8)$) {3};
	\node[red] (Red4) at ($(3)+(90:.9)$) {4};
	
	\draw (4)--(2)--(3)--(1);
	\draw (1)--(2)--(5);
	\draw[red, ->] ($(5)+(190:1.9cm)$) .. controls ($(5)+(-260:2.5)$) and ($(5)+(-20:2.5)$) .. ($(5)+(-113:1.9cm)$); 
	\draw[red, ->] ($(5)+(-113:2cm)$) .. controls ($(5)+(-90:3.7)$) and ($(1)+(-20:3.7)$) .. ($(1)+(-2:2.1cm)$); 
	\draw[red, ->] ($(1)+(-2:2cm)$) .. controls ($(1)+(-90:2.5)$) and ($(1)+(150:2.5)$) .. ($(1)+(62:2cm)$);
	\draw[red, ->] ($(3)+(-122:2cm)$) .. controls ($(3)+(-210:2)$) and ($(3)+(45:2)$) .. ($(3)+(-48:1.2cm)$);
	\draw[->, red] ($(2)+(128.5:3.35)$) to [out=-50, in= 170,above, pos=.7] node {$5$} ($(2)+(69:1)$); 
	\draw[->, red] ($(2)+(69:1.1)$) to [out=90, in= -30, right,  pos=.7] node {$6$} ($(5)+(207:1.6)$); 
	\draw[->, red] ($(1)+(23:1.5)$) to [out=-10, in= 120, below,  pos=.2] node {$7$} ($(0)+(2:2)$); 
	\draw[->, red] ($(0)+(2:2.1)$) to [out=45, in= -165,below,  pos=.4] node {$8$} ($(2)+(147:1.6)$); 
	\draw[<-, blue] ($(5)+(193:2)$) to [bend right, right, pos=.4] node {$1$} ($(5)+(204:2)$); 
	\draw[->, blue] ($(2)+(144:2.55)$) to [out=0, in=200, right, pos=.5] node {$3$} ($(2)+(131:1.1)$); 
	\draw[->, blue] ($(2)+(131:1.8)$) to [out=140, in=0, looseness=1, above, pos=.8] node {$4$} ($(1)+(59:3)$); 
	\draw[->, blue] ($(1)+(59:2.9)$) to [out=-70, in=110, looseness=1,right, pos=.6] node {$5$} ($(1)+(27:1.3)$); 
	\draw[draw=white,double=white] ($(2)+(147:1.8)$) to [out=160, in=185,looseness=4.2] node {} ($(5)+(190:2cm)$); 
	\draw[->, red] ($(2)+(147:1.8)$) to [out=160, in=185,looseness=4.2,  below, pos=.25] node {$9$} ($(5)+(190:2cm)$); 
	\draw[draw=white,double=white] ($(5)+(193:2.15)$) to [out=235, in=20] node {} ($(4)+(-33:.9)$); 
	\draw[->, blue] ($(5)+(193:2.15)$) to [out=235, in=20, above, pos=.5] node {$2$} ($(4)+(-32:.9)$); 
	\draw[draw=white,double=black] (1)--(5)--(4);
		 	
	\end{tikzpicture}
\caption{Picture of Brauer graph $G'$ of $A'$ and the two distinct Green walks on $G'$ of lengths 5 and 9, respectively}
\end{figure}
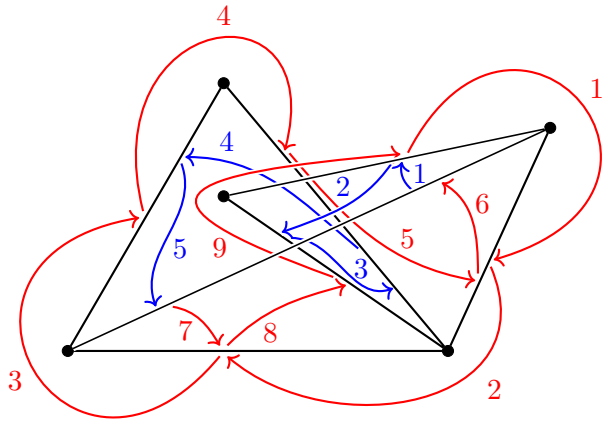

 Then the stable Auslander-Reiten quiver of $A'$ is a disjoint union of infinitely many components of the forms $\mathbb{Z}A_{\infty}/ \langle \tau \rangle$ and $\mathbb{Z}A_\infty^\infty$  and 
  one component of the form $\mathbb{Z}A_{\infty}/ \langle \tau^5 \rangle$
 and
 one component of the form $\mathbb{Z}A_{\infty}/ \langle \tau^9 \rangle$.

}
\end{Example}

\subsection{Auslander-Reiten components of non-domestic Brauer graph algebras}

Let $A$ be a representation-infinite non-domestic Brauer graph algebra with Brauer graph $G$.
It follows directly from Theorem~\ref{Thm ES non-domestic} and Theorem~\ref{Thm Du exceptional} that  
 $_s\Gamma_\Lambda$ is the disjoint union of 
  \begin{itemize}
\item[$\bullet$] $k$ components of the form $\mathbb{Z}A_{\infty}/ \langle \tau^{l_i} \rangle$, where $k$ is the number of double stepped Green walks and $l_i$ is the length of the $i$th double-stepped Green walk, for $1 \leq i \leq k$, \item[$\bullet$] infinitely many components of the form $\mathbb{Z}A_{\infty}/ \langle \tau \rangle$ 
\item[$\bullet$] infinitely many components of the form $\mathbb{Z}A_{\infty}^{\infty}$. 
\end{itemize}

\subsection{Auslander-Reiten components of domestic Brauer graph algebras}

Let $A$ be a representation-infinite $m$-domestic Brauer graph algebra with Brauer graph $G$.
It follows from Theorem~\ref{Thm ES domestic} and Theorem~\ref{Thm Du exceptional} that  
 $_s\Gamma_\Lambda$ is the disjoint union of 
 \begin{itemize}
\item[$\bullet$] $m$ components of the form $\mathbb{Z}\tilde{A}_{p,q}$
\item[$\bullet$] $m$ components of the form $\mathbb{Z}A_{\infty}/ \langle \tau^p \rangle$
\item[$\bullet$] $m$ components of the form $\mathbb{Z}A_{\infty}/ \langle \tau^q \rangle$ 
\item[$\bullet$] infinitely many components of the form $\mathbb{Z}A_{\infty}/ \langle \tau \rangle$. 
 \end{itemize}

More precisely, 
 
 \begin{theorem}\cite[Theorem 4.4, Corollary 4.5]{Du}\label{Thm Drew domestic}
 Let $A$ be a representation-infinite domestic Brauer graph algebra with Brauer graph $G$ with $n$ edges. If $G$ has a cycle then it is unique and let $n_1$ be the number of (additional) edges on the inside of the cycle and  $n_2$ the number of (additional) edges on the outside of the cycle. In the notation above:
 
\begin{itemize}
\item[(1)] If $A$ is 1-domestic, then  $m =1$ and  $p +q = 2n$. Furthermore, 
\begin{itemize}
\item[$\bullet$] if $G$ is a tree, then  $p = q =n$,
\item[$\bullet$] if $G$ has a unique cycle (of odd) length $l$, then  $p = l + 2n_1$ and $q =  l + 2 n_2$.  
\end{itemize} 
\item[(2)] If $A$ is 2-domestic with unique cycle (of even) length $l$, then $m=2$ and  $p = \frac{l}{2} + n_1$ and 
 $q = \frac{l}{2} + n_2$ such that $p +q = n$.
\end{itemize} 
 \end{theorem}

\begin{Example} {\rm

(1)  Example of the stable Auslander-Reiten quiver of a 1-domestic Brauer graph algebra with Brauer graph (including the two distinct Green walks) given by 

 \begin{figure}[H]
	\centering
	\begin{tikzpicture}[auto, thick]
	\node[cb] (1) at (0,0) {};
	\node[cb] (2) at ($(1)+(0:4)$) {};
	\node[cb] (3) at ($(1)+(60:4)$) {};
	\node[cb] (4) at ($(3)+(90:2)$) {};
	\node[cb] (5) at ($(1)+(30:1.55)$) {};
	\node[cb] (6) at ($(5)+(0:1.3)$) {};
	\node[cb] (7) at ($(5)+(60:1.3)$) {};
	
	\draw (1) to [below] node {$3$}(2);
	\draw (2) to [right, pos=.55] node {$2$}(3);
	\draw (3) to [right, pos=.8] node {$1$}(4);
	\draw (3) to [left, pos=.45] node {$7$}(1);
	\draw (1) to [above, pos=.4] node {$4$}(5);
	\draw (5) to [below, pos=.47] node {$6$}(6);
	\draw (5) to [left, pos=.7] node {$5$}(7);
	
	\draw[->, red] (1)+(-5:1.8cm) arc (-5:-295:1.8cm){};
	\draw[->, red] (2)+(115:1.8cm) arc (115:-175:1.8cm){};
	\draw[->, red] (3)+(85:.9cm) arc (85:-40:1.35cm){};
	\draw[<-, red] (3)+(95:.9cm) arc (95:220:1.35cm){};
	\draw[->, red] (4)+(265:.9cm) arc (265:-85:.9cm){};
	
	\draw[->, blue] (5)+(200:.5cm) to [out=80,in=195]node {} ($(5)+(70:.6cm)$);
	\draw[<-, blue] (5)+(220:.45cm) to [out=-20,in=230] node {} ($(5)+(-5:.6cm)$);
	\draw[->, blue] (5)+(55:.6cm) to [out=-20,in=100] node {} ($(5)+(5:.6cm)$);
	
	\draw[->, blue] (5)+(220:.6cm) to [out=-30,in=140] node {} ($(1)+(3:1.9cm)$);
	\draw[->, blue] (1)+(3:2cm) to [out=5,in=230,looseness=.5] node {} ($(2)+(123:.8cm)$);
	\draw[->, blue] (2)+(125:.9cm) to [out=130,in=-20,looseness=.8] node {} ($(3)+(245:.8cm)$);
	\draw[->, blue] (3)+(245:.9cm) to [out=-45,in=120,looseness=0.6] node {} ($(5)+(200:.6cm)$);
	
	\draw[blue] (7)+(230:.5cm) to [out=120,in=150] node {} ($(7)+(60:.3cm)$);
	\draw[<-,blue] (7)+(250:.5cm) to [out=0,in=-30] node {} ($(7)+(60:.3cm)$);
	
	\draw[blue] (6)+(170:.55cm) to [out=60,in=90] node {} ($(6)+(0:.3cm)$);
	\draw[<-,blue] (6)+(-170:.55cm) to [out=-60,in=-90] node {} ($(6)+(0:.3cm)$);
	 	
	\end{tikzpicture}
	\caption{Example 5.5.1}
\end{figure}
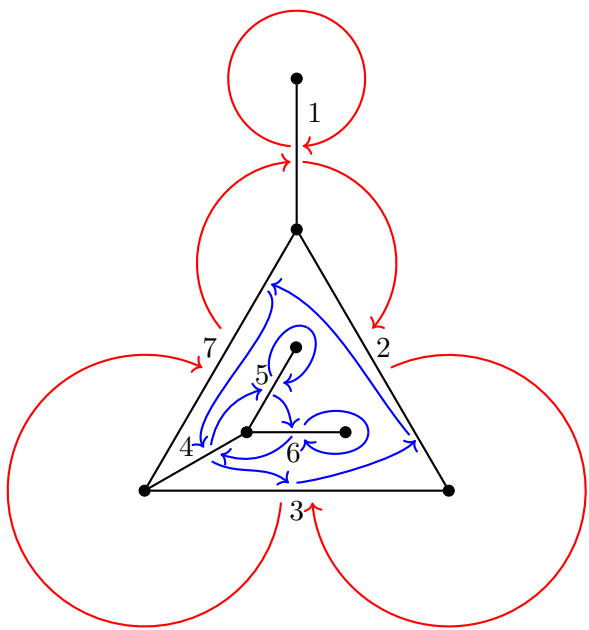

Here $l = 3$, $n_1 = 3$ and $n_2 =1$ and therefore by Theorem~\ref{Thm Drew domestic}, $p = 9$ and $q = 5$ and  the stable Auslander-Reiten quiver consists of  one component  of the form $\mathbb{Z}\tilde{A}_{9,5}$, one of the form
 $\mathbb{Z}A_{\infty}/ \langle \tau^9 \rangle$ and one of the form $\mathbb{Z}A_{\infty}/ \langle \tau^5 \rangle$  as well as infinitely many components of the form $\mathbb{Z}A_{\infty}/ \langle \tau \rangle$. Note that $p =9$ is the length of the unique double stepped Green walk on the inside of the 3-cycle and $q=5$ is the length of the unique double stepped Green walk on the outside of the 3-cycle. 

(2)   Example of the stable Auslander-Reiten quiver of a 2-domestic Brauer graph algebra with Brauer graph (including the two distinct Green walks) given by 

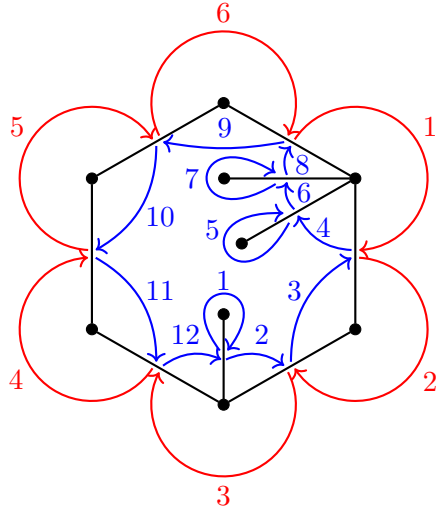
\begin{figure}[H]
	\centering
	\begin{tikzpicture}[auto, thick]
	\node (0) at (0,0) {};
	\node[cb] (1) at ($(0)+(30:2)$) {};
	\node[cb] (2) at ($(0)+(90:2)$) {};
	\node[cb] (3) at ($(0)+(150:2)$) {};
	\node[cb] (4) at ($(0)+(210:2)$) {};
	\node[cb] (5) at ($(0)+(270:2)$) {};
	\node[cb] (6) at ($(0)+(330:2)$) {};
	\node[cb] (7) at ($(1)+(180:1.725)$) {};
	\node[cb] (8) at ($(1)+(210:1.725)$) {};
	\node[cb] (9) at ($(0)+(-90:.8)$) {};
	\node[red] (R1) at ($(1)+(35:1.2)$){1};
	\node[red] (R6) at ($(2)+(90:1.2)$){6};
	\node[red] (R5) at ($(3)+(145:1.2)$){5};
	\node[red] (R4) at ($(4)+(214:1.2)$){4};
	\node[red] (R3) at ($(5)+(-90:1.2)$){3};
	\node[red] (R2) at ($(6)+(-35:1.2)$){2};
	\node[blue] (G1) at ($(5)+(90:1.65)$){1};
	\node[blue] (G2) at ($(1)+(200:2)$){5};
	\node[blue] (G3) at ($(1)+(180:2.15)$){7};
	
	\draw (1)--(2)--(3)--(4)--(5)--(6)--(1);
	\draw (1) to node {}(7);
	\draw (1) to node {}(8);
	\draw (5) to node {}(9);
	
	\draw[->, red] (1)+(147:.95) arc (147:-87:.95cm);
	\draw[->, red] (2)+(207:.95) arc (207:-27:.95cm);
	\draw[->, red] (3)+(267:.95) arc (267:33:.95cm);
	\draw[->, red] (4)+(327:.95) arc (327:93:.95cm);
	\draw[->, red] (5)+(27:.95) arc (27:-207:.95cm);
	\draw[->, red] (6)+(87:.95) arc (87:-147:.95cm);
	
	\draw[blue, ->] ($(5)+(94:.72)$) .. controls ($(5)+(115:1.9)$) and ($(5)+(65:1.9)$) .. ($(5)+(86:.72)$);
	\draw[->, blue] ($(5)+(87:.6)$) to [bend left, above, pos=.6] node {$2$} ($(5)+(33:.95)$);
	\draw[->, blue] ($(6)+(207:.95)$) to [bend left, left, pos=.6] node {$3$} ($(6)+(93:.95)$);
	\draw[->, blue] ($(1)+(-93:.95)$) to [bend left, right, pos=.8] node {$4$} ($(1)+(213:.9)$);
	\draw[blue, ->] ($(1)+(214:1.05)$) .. controls ($(1)+(230:2.4)$) and ($(1)+(190:2.4)$) .. ($(1)+(206:1.05)$);
	\draw[->, blue] ($(1)+(207:.9)$) to [bend left, right, pos=.64] node {$6$} ($(1)+(183:.9)$);
	\draw[blue, ->] ($(1)+(184:1.05)$) .. controls ($(1)+(200:2.4)$) and ($(1)+(160:2.4)$) .. ($(1)+(176:1.05)$);
	\draw[->, blue] ($(1)+(177:.9)$) to [bend left, right, pos=.35] node {$8$} ($(1)+(153:.95)$);
	\draw[->, blue] ($(2)+(-33:.95)$) to [above, pos=.49, out=190, in=350] node {$9$} ($(2)+(-147:.95)$);
	\draw[->, blue] ($(3)+(27:.95)$) to [bend left, right, pos=.6] node {$10$} ($(3)+(-87:.95)$);
	\draw[->, blue] ($(4)+(87:.95)$) to [bend left, right, pos=.4] node {$11$} ($(4)+(-27:.95)$);
	\draw[->, blue] ($(5)+(147:.95)$) to [bend left, above, pos=.4] node {$12$} ($(5)+(93:.6)$);
		
	\end{tikzpicture}
	\caption{Example 5.8}
\end{figure}

Here $l = 6$, $n_1 = 3$ and $n_2 =0$ and therefore by Theorem~\ref{Thm Drew domestic}, $p = 6$ and $q = 3$ and  the stable Auslander-Reiten quiver consists of  two components  of the form $\mathbb{Z}\tilde{A}_{6,3}$, two of the form
 $\mathbb{Z}A_{\infty}/ \langle \tau^6 \rangle$ and one of the form $\mathbb{Z}A_{\infty}/ \langle \tau^3 \rangle$  as well as infinitely many components of the form $\mathbb{Z}A_{\infty}/ \langle \tau \rangle$. Note that $p =6$ is the length of the two double stepped Green walks on the inside of the 6-cycle and $q=3$ is the length of the two double stepped Green walks on the outside of the 6-cycle.

}
\end{Example}

\subsection{Position of modules in the Auslander-Reiten quiver}

In this section we determine, for any representation-infinite Brauer graph algebra (domestic or non-domestic) which simple modules and radicals of indecomposable projective  modules (and hence which projective modules) lie in exceptional tubes and which do not. For that we define the notion of exceptional edges and we will see that the simple modules and the radicals of the projective indecomposable modules associated to the exceptional edges lie in the exceptional tubes of $_s\Gamma_A$.  

\begin{definition}\cite{Du}
{\rm (1) Let $G$ be a Brauer graph such that $G$ is not a Brauer tree (that is if $G$ is a tree then there are at least two vertices of multiplicity greater than 2). An \textit{exceptional subtree} of $G$ is  a subgraph $T$ of $G$ such that
\begin{itemize}
\item[(i)] $T$ is a tree, 
\item[(ii)] there exists at unique vertex $v$ in $T$ such that $(G \setminus T) \cup v$ is connected,  
\item[(iii)] every vertex of $T$ has multiplicity 1 except  perhaps $v$.
\end{itemize}
We call $v$ the \textit{connecting vertex} of $T$.
 
 (2) Let $G$ be a Brauer graph. An edge $e$ in $G$  is called an \textit{exceptional edge} if $e$ lies in an exceptional subtree of $G$ and it is called a \textit{non-exceptional edge} otherwise. 
}
\end{definition}

Note that the non-exceptional edges in a Brauer graph $G$ are all connected to each other and that a non-exceptional edge is never truncated.

\begin{figure}[H]
	\centering
	\begin{tikzpicture}[auto, thick]
	\node[cb, red] (0) at (0,0) {};
	\node[cb] (1) at ($(0)+(0:1.2)$) {};
	\node[cb, blue] (2) at ($(1)+(45:1.2)$) {};
	\node[cb] (3) at ($(1)+(-45:1.2)$) {};
	\node[cb] (4) at ($(3)+(45:1.2)$) {};
	\node[cb] (5) at ($(4)+(0:1.2)$) {};
	\node[cb] (6) at ($(5)+(0:1.2)$) {};
	\node[cb] (7) at ($(6)+(0:1.5)$) {};
	\node[cb,green] (8) at ($(7)+(0:1.2)$) {};
	\node[cb,green] (9) at ($(8)+(0:1.2)$) {};
	\node[cb,green] (10) at ($(9)+(45:1.2)$) {};
	\node[cb,green] (11) at ($(9)+(-45:1.2)$) {};
	\node[cb,red] (12) at ($(0)+(-90:1.2)$) {};
	\node[cb, blue] (13) at ($(2)+(135:1.2)$) {};
	\node[cb, blue] (14) at ($(2)+(45:1.2)$) {};
	\node[cbl] (a) at ($(2)+(0:.7)$){2};
	\node[cw] (b) at ($(7)+(-90:.5)$){2};
	\node[cg] (c) at ($(8)+(-90:.5)$){2};
	
	\draw (0)--(1)--(2);
	\draw[blue] (13) -- (2)--(14);
	\draw (1)--(3)--(4)--(2);
	\draw (4)--(5);
	\draw (5) to node {e} (6);
	\draw (6) to node {f}(7);
	\draw (7)--(8);
	\draw ($(0)+(180:.5)$) circle (.5);
	\draw[red] (12)--(0);
	\draw[green] (8) to node {g} (9)--(10);
	\draw[green] (9)--(11);

	\end{tikzpicture}
	\caption{Example of a Brauer graph with three exceptional subtrees}\label{Figure Exceptional Subtrees}
\end{figure}
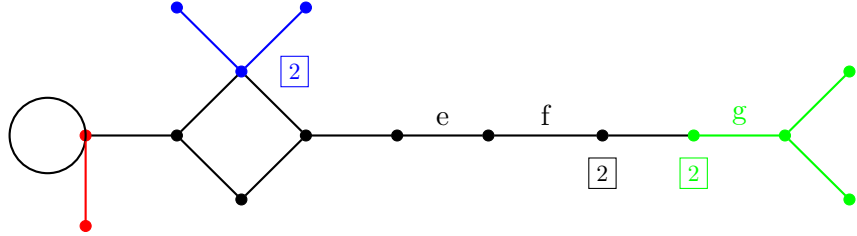

\begin{theorem}\cite[Theorem 4,18]{Du}
Let $A$ be a representation-infinite Brauer graph algebra with Brauer graph $G$ and let $e$ be an edge in $G$. Then the simple module $S_e$ associated to $e$ and the radical  $\rad P_e$ of the indecomposable projective module $P_e$ associated to $e$ lie in an exceptional tube of $_s\Gamma_A$ if and only if $e$ is an exceptional edge. 

Furthermore, $S_e$ and $\rad P_e$ lie in the same exceptional tube if and only if $e$ occurs twice within the same double-stepped Green walk. 
\end{theorem}

\begin{theorem}\cite[Theorem 4.21]{Du} Let $A$ be a representation-infinite Brauer graph algebra with Brauer graph $G$ and let $e$ and $f$ be two not necessarily distinct non-exceptional edges in $G = (G_0, G_1, m, {\mathfrak o})$. Then $S_e$ and $\rad P_f$ lie in the same component of $_s\Gamma_A$ if and only if $A$ is 1-domestic or there exists an even length path 

\[
\xymatrix{
        & \stackbin[v_0]{}{\bullet} \ar@{-}[r]^{e=e_1} &  \stackbin[v_1]{}{\bullet} \ar@{-}[r]^{e_2} &  \stackbin[v_2]{}{\bullet} \ar@{..}[r] &  \stackbin[v_{2n-1}]{}{\bullet} \ar@{-}[r]^{f = e_{2n}} & \stackbin[v_{2n}]{}{\bullet}  \\
}
\]

of non-exceptional edges  such that 
 \begin{itemize}
\item[(i)] none of the edges $e_i$, for $1  \leq i \leq 2n$,   are a loop, 
\item[(ii)] $e_i$ and $e_{i+1}$ are the only non-exceptional edges incident with $v_i$, for $1 \leq i \leq 2n-1$, 
\item[(iii)] $m(v_i) = 1$, for $1 \leq i \leq 2n-1$,  except if $e_i = e_{i+1}$ in which case $m(v_i) = 2$, for $1 \leq i \leq 2n-1$.  
\end{itemize}
\end{theorem}

\begin{Example}
 {\rm For the Brauer graph algebra with Brauer graph given by the graph in Figure~\ref{Figure Exceptional Subtrees}, the simple module $S_e$ associated to the edge $e$ and the $\rad P_f$ where $P_f$ is the indecomposable projective module associated to the edge$f$, are in the same component of the Auslander-Reiten quiver. Since the edge corresponding to $e$ is not exceptional, this component is not an exceptional tube but of the form $\mathbb{Z}A_{\infty}^{\infty}$.  
On the other hand, the simple module $S_g$ as well as $\rad P_g$, the radical of the indecomposable projective associated to the edge $g$ lie in the same exceptional tube (of rank 27) since the edge $g$ is exceptional and occurs twice in the same double-stepped Green walk. }
\end{Example}

 \section{Appendix}\label{Appendix}

\subsection{Derived equivalences: Rickard's Theorem} 

Let $A$ be a finite dimensional $K$-algebra. We denote by $\cD^\flat(A)$ the bounded
derived category of finite dimensional $A$-modules. 

If $M$ is a tilting $A$-module then  
the categories $\cD^\flat(A)$ and $\cD^\flat(B)$, where $B = End_A(M)$, are 
equivalent \cite{Happel, CPS}. 

Rickard showed that all derived equivalences of finite dimensional algebras are of a similar nature. Namely let $\mod$-$A$ be the category of all finitely presented $A$-modules and
$P_A$ the category of all finitely generated projective $A$-modules. We denote by 
$\cK^\flat(P_A)$ the homotopy category of bounded complexes in $P_A$. 

\begin{theorem}\cite[1.1]{Rickard}\label{Rickard}
Let $A$ and $B$ be two finite-dimensional algebras. The following are
equivalent:

(a) $\cD^\flat(A)$ and $\cD^\flat(B)$ are equivalent as triangulated categories.

(b) $\cK^\flat(P_A)$ and $\cK^\flat(P_B)$ are equivalent as triangulated categories.

(c) $B$ is isomorphic to the endomorphism ring of an object $T$ of $\cK^\flat(P_A)$, that is $B \simeq \End_{\cK^\flat(P_A)}(T)$, such that

(i) For $n \neq 0$, $Hom(T, T[n]) = 0$.

(ii) ${\rm add}(T)$, the full subcategory of $\cK^\flat(P_A)$  consisting of direct summands of
direct sums of copies of $T$, generates $\cK^\flat(P_A)$ as a triangulated category.

Moreover, any equivalence as in (a) restricts to an equivalence between the full
subcategories consisting of objects isomorphic to bounded complexes of projectives
(which are equivalent to $\cK^\flat(P_A)$ and $\cK^\flat(P_B)$ respectively). 
\end{theorem}

A complex $T$ as in Theorem~\ref{Rickard} (c) is called a \textit{tilting complex over $A$}. 

\newpage

\parskip1pt

\bibliographystyle{plain}

\bigskip

\textsc{Department of Mathematics, University of Leicester, University Road, Leicester LE1 7RH, 
United Kingdom} \newline \textit{E-mail address:}  \texttt{schroll@leicester.ac.uk }

\end{addmargin}

\end{document}